\documentclass[12pt,a4paper,english,reqno,a4paper]{amsart}
\usepackage[T1]{fontenc}
\usepackage[utf8]{inputenc}
\usepackage{color}
\usepackage{verbatim}
\usepackage{amstext}
\usepackage{amsthm}
\usepackage{stmaryrd}
\usepackage{setspace}
\makeatletter

\numberwithin{equation}{section}
\numberwithin{figure}{section}
\theoremstyle{plain}
\newtheorem{thm}{\protect\theoremname}[section]
  \theoremstyle{remark}
  \newtheorem{rem}[thm]{\protect\remarkname}
  \theoremstyle{plain}
  \newtheorem{lem}[thm]{\protect\lemmaname}
  \theoremstyle{definition}

\usepackage{amsfonts}\usepackage{amsthm}
\usepackage{epsfig}\usepackage{mathrsfs}
\usepackage{amstext}
\usepackage{esint}
\usepackage[T1]{fontenc}
\usepackage[utf8]{inputenc}

\@ifundefined{definecolor}
 {\usepackage{color}}{}

\makeatletter
\newcommand{\Rmnum}[1]{\expandafter\@slowromancap\romannumeral#1@}\makeatother

\numberwithin{equation}{section}

\allowdisplaybreaks

\newcommand{\set}[1]{\left\{#1\right\}}

\newcommand{\defs}{:=}

\DeclareMathOperator{\di}{div}

\DeclareMathOperator*{\dist}{dist}

\newcommand{\dif}{\mathrm{d}}

\DeclareSymbolFont{lettersA}{U}{pxmia}{m}{it}
\DeclareMathSymbol{\piup}{\mathord}{lettersA}{"19}

\newcommand{\Real}{\mathbb R}

\newcommand{\mr}[1]{\mathrm{#1}}
\newcommand{\mb}[1]{\mathbf{#1}}

\newcommand{\mcc}{\mathcal{C}}

\newcommand{\mcp}{\mathcal{P}}

\newcommand{\mcr}{\mathcal{R}}

\makeatother

\makeatother

\usepackage{babel}
  \providecommand{\definitionname}{Definition}
  \providecommand{\lemmaname}{Lemma}
  \providecommand{\remarkname}{Remark}
\providecommand{\theoremname}{Theorem}

\begin{document}

\title[]{On admissible positions of Transonic Shocks for Steady Euler Flows In A 3-D Axisymmetric Cylindrical Nozzle}

\author{Beixiang Fang}

\author{Xin Gao}

\address{B.X. Fang: School of Mathematical Sciences, MOE-LSC, and SHL-MAC, Shanghai
Jiao Tong University, Shanghai 200240, China }

\email{\texttt{bxfang@sjtu.edu.cn}}

\address{X. Gao: School of Mathematical Sciences, Shanghai
	Jiao Tong University, Shanghai 200240, China }

\email{\texttt{sjtu2015gx@sjtu.edu.cn}}

\keywords{steady Euler system; axisymmetric nozzles; transonic shocks; receiver pressure; existence;}
\subjclass[2010]{35A01, 35A02, 35B20, 35B35, 35B65, 35J56, 35L65, 35L67, 35M30, 35M32, 35Q31, 35R35, 76L05, 76N10}

\date{\today}

\begin{abstract}

This paper concerns with the existence of transonic shocks for steady Euler flows in a 3-D axisymmetric cylindrical nozzle, which are governed by the Euler equations with the slip boundary condition on the wall of the nozzle and a receiver pressure at the exit. Mathematically, it can be formulated as a free boundary problem with the shock front being the free boundary to be determined. In dealing with the free boundary problem, one of the key points is determining the position of the shock front. To this end, a free boundary problem for the linearized Euler system will be proposed, whose solution gives an initial approximating position of the shock front. Compared with 2-D case, new difficulties arise due to the additional 0-order terms and singularities along the symmetric axis. New observation and careful analysis will be done to overcome these difficulties. Once the initial approximation is obtained, a nonlinear iteration scheme can be carried out, which converges to a transonic shock solution to the problem.

\end{abstract}

\maketitle
\tableofcontents{}

\section{Introduction}

In this paper, we are concerned with the existence of steady transonic shocks, especially the position of the shock front, in a 3-D axisymmetric cylindrical nozzle (see Figure \ref{fig:1}). The steady flow in the nozzle is supposed to be governed by the Euler system which reads:\begin{align}
&\di(\rho \mathbf{u})=0,\label{eq1}\\
&\di(\rho \mathbf{u}\otimes \mathbf{u})+\nabla p=0,\label{eq2}\\
&\di(\rho(e+\frac12 |\mathbf{u}|^2+\displaystyle\frac{p}{\rho})\mathbf{u})=0,\label{eq3}
\end{align}
where  $\mb{x}\defs(x_1,x_2,x_3)$ are the space variables, ``$\di $'' is the divergence operator with respect to $\mb{x}$, $\mathbf{u}=(u_1,u_2,u_3)^T $ is the velocity field, $\rho$, $p$ and $e$ stand for the density, pressure, and the internal energy respectively. Moreover, the fluid is supposed to be a polytropic gas with the state equation
\[
	p=A(s){\rho}^{\gamma},
\]
where $s$ is the entropy, $\gamma>1$ is the adiabatic exponent, and $A(s)=(\gamma-1)\exp(\displaystyle\frac{s-s_0}{c_v})$ with $c_v$ the specific heat at constant volume.

\begin{figure}[!h]
	\centering
	\includegraphics[width=0.45\textwidth]{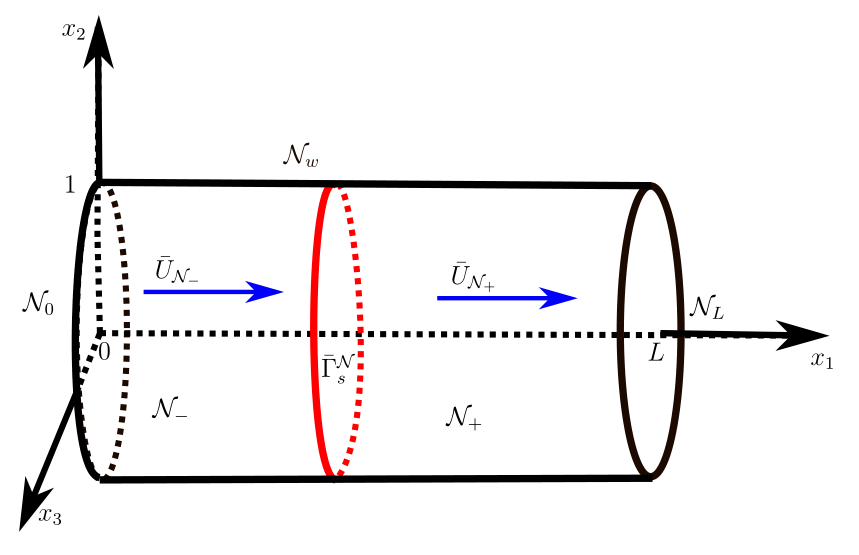}
	\caption{The transonic shock flows in the cylindrical nozzle.\label{fig:1}}
\end{figure}

It is well-known, in the Euler system, the equation $\eqref{eq3}$ can be replaced by the Bernoulli's law below:
\begin{equation}
\di (\rho \mathbf{u} B)=0,
\end{equation}
where the Bernoulli constant $B$ is given by, with $i = \displaystyle\frac{\gamma p}{(\gamma -1)\rho}$ being the enthalpy,
\begin{equation}
B =\displaystyle\frac12 |\mathbf{u}|^2 + i.
\end{equation}

For Euler flows, a shock front is a strong discontinuity of the distribution functions of the state parameters $U_{\mathcal {N}} = (u_1, u_2, u_3, p,\rho)^T$ of the fluid. Let $\Gamma_{\mr{s}}^{\mathcal {N}} \defs\set{x_1 = \psi_{\mathcal {N}}(x')}$, with $ x'\defs (x_2,x_3)$, be the position of a shock front, then the following Rankine-Hugoniot(R-H) conditions should be satisfied:
\begin{align}
&[(1 , - \nabla_{x'}\psi_{\mathcal {N}}(x'))\cdot \rho \mathbf{u}] = 0,\label{eq810}\\
&[((1 , - \nabla_{x'} \psi_{\mathcal {N}} (x'))\cdot \rho \mathbf{u})\mathbf{u}] + (1 , - \nabla_{x'}\psi_{\mathcal {N}}(x'))^T [p] = 0,\\
&[(1 , - \nabla_{x'} \psi_{\mathcal {N}}(x'))\cdot \rho \mathbf{u} B] = 0,\label{eq813}
\end{align}
where $[\cdot]$ denotes the jump of the quantity across the shock front $ \Gamma_{\mr{s}}^{\mathcal {N}} $, and $\nabla_{x'} \defs (\partial_{x_2}, \partial_{x_3})$. Moreover, the entropy condition $ [p]>0 $ should also hold, which means that the pressure increases across  $ \Gamma_{\mr{s}}^{\mathcal {N}} $.

\subsection{Steady plane normal shocks in a flat cylindrical nozzle.}

Given a flat cylindrical nozzle:
\begin{equation}
\mathcal {N}: = \left\{(x_1,x_2,x_3)\in {\mathbb{R}}^3: 0 < x_1 < L,\  0 \leq x_2^2 + x_3^2  < 1\right\},
\end{equation}
with the entrance $ \mathcal{N}_0 $, the exit $ \mathcal{N}_L $, and the wall $ \mathcal{N}_w $ being
\begin{equation*}
\begin{split}
\mathcal{N}_0 &\defs \mathcal{N} \cap \{(x_1,x_2,x_3)\in {\mathbb{R}}^3 : x_1 = 0\},\\
\mathcal{N}_L &\defs \mathcal{N} \cap \{(x_1,x_2,x_3)\in {\mathbb{R}}^3 : x_1 = L \}, \\
\mathcal{N}_w &\defs \mathcal{N} \cap \{(x_1 ,x_2 ,x_3 )\in {\mathbb{R}}^3 : x_2^2 + x_3^2 = 1\},
\end{split}
\end{equation*}
it is well-known that there may exist plane normal shocks in it with the shock front being a plane perpendicular to the $ x_1 $-axis (see Figure \ref{fig:1}).
Let $x_1 = \bar{x}_{\mr{s}}$ with $\bar{x}_{\mr{s}} \in (0,L)$ be the position of the plane shock front,  $\bar{U}_{\mathcal {N}_-}\defs(\bar{q}_-,0,0,\bar{p}_-,\bar{\rho}_-)$ be the state of the uniform supersonic flow ahead of it, and $\bar{U}_{\mathcal {N}_+}\defs(\bar{q}_+,0,0,\bar{p}_+,\bar{\rho}_+)$ be the state of the uniform subsonic flow behind it. (In this paper, the subscript ``$ - $'' will be used to denote the parameters ahead of the shock front and the subscript ``$ + $'' behind the shock front.) Then the R-H conditions $\eqref{eq810}$-$\eqref{eq813}$ yield
\begin{equation}\label{eq20}
\left\{
\begin{aligned}
&[\bar{\rho} \bar{q}] = \bar{\rho}_{+} \bar{q}_{+} - \bar{\rho}_{-} \bar{q}_{-}=0,\\
&[\bar{p}+\bar{\rho} \bar{q}^2] =  (\bar{p}_{+}+\bar{\rho}_{+} \bar{q}_{+}^2) - (\bar{p}_{-}+\bar{\rho}_{-} \bar{q}_{-}^2)=0,\\
&[\bar{B}] =  \bar{B}_{+} - \bar{B}_{-}=0.
\end{aligned}
\right.
\end{equation}
Then direct computations yield that
\begin{align}
&\bar{p}_+=\Big((1+{\mu}^2)\bar{M}_-^2-{\mu}^2\Big)\bar{p}_-,\label{eq39}\\
&\bar{q}_+  = {\mu}^2\Big(\bar{q}_- + \frac{2}{\gamma+1}\frac{\bar{c}_{-}^2}{\bar{q}_-}\Big),\\
&\bar{\rho}_+ =\frac{{\bar{\rho}_- \bar{q}_-}}{\bar{q}_+}=\frac{{\bar{\rho}_- \bar{q}_-^2}}{\bar{c}_{*}^2}\label{eqqe},
\end{align}
where $c = \sqrt{\displaystyle\frac{\gamma p}{\rho}}$ is the sonic speed, $M = \displaystyle\frac{q}{c}$ is the Mach number, and
\begin{equation}
\begin{aligned}
\bar{c}_{*}^2 := \bar{q}_+ \bar{q}_- = {\mu}^2\Big(\bar{q}_-^2 + \frac{2}{\gamma+1}\bar{c}_-^2\Big),\,\,\,{\mu}^2=\frac{\gamma-1}{\gamma+1},\,\,\,
\bar{M}_-^2 = \displaystyle\frac{\bar{q}_-^2}{\bar{c}_-^2}.
\end{aligned}
\end{equation}

\begin{rem}
	By applying the entropy condition $ [\bar{p}]>0 $, the equation \eqref{eq39} yields that $ \bar{M}_->1 $, that is, $\bar{q}_-> \bar{c}_{*}$. Then, since $\bar{q}_+ \bar{q}_- = \bar{c}_{*}^2$, it is obvious that $\bar{q}_+ < \bar{c}_{*}$, which implies that  $ \bar{M}_+ <1 $. 	
	That is, the flow is supersonic ahead of the shock front, and is subsonic behind it.
\end{rem}

\begin{rem}\label{rem:plane_normal_shocks}
	For each $\bar{x}_{\mr{s}} \in (0,L)$, it is obvious that $ (\bar{U}_{\mathcal {N}_+};\ \bar{U}_{\mathcal {N}_-};\ \bar{x}_{\mr{s}}) $ gives a plane transonic shock solution to the steady Euler system \eqref{eq1}-\eqref{eq3} in the following sense: the position of the shock front is $ \bar{\Gamma}_{{\mr{s}}}^{\mathcal {N}}\defs\set{x_{1}=\bar{x}_{\mr{s}},\ x_2^2 + x_3^2  \leq 1} $, the state of the fluid $ {U}_{\mathcal {N}}\equiv\bar{U}_{\mathcal {N}_-} $ within the region between $ \mathcal {N}_0 $ and $ \bar{\Gamma}_{{\mr{s}}}^{\mathcal {N}}  $ in the nozzle $ \mathcal {N} $, and the state of the fluid  $ {U}_{\mathcal {N}}\equiv\bar{U}_{\mathcal {N}_+} $ within the region between $ \bar{\Gamma}_{{\mr{s}}}^{\mathcal {N}} $ and $ \mathcal {N}_L $.
	Therefore, as the state $ \bar{U}_{\mathcal {N}_-} $ is given, the state of the flow behind the shock front $ \bar{U}_{\mathcal {N}_+} $ is uniquely determined by \eqref{eq39}-\eqref{eqqe}, while the position of the plane shock front could be arbitrary in the flat nozzle $ \mathcal {N} $.
\end{rem}

\subsection{The steady transonic shock problem in an axisymmetric perturbed nozzle.}

Based on the above steady plane normal shock solutions in a flat cylindrical nozzle, this paper is going to study the existence of transonic shocks in an axisymmetric 3-D nozzle, which is a small perturbation of a flat cylindrical nozzle, with a pressure condition at the exit.

Let $ (z, r, \varpi) $ be the cylindrical coordinate in $ \Real^{3} $ such that
\begin{align*}
  (x_1, x_2, x_3) = (z, r \cos \varpi,r\sin \varpi),
\end{align*}
 and the perturbed nozzle is axisymmetric with respect to $ x_1 $-axis:
\begin{equation}
\widetilde{\mathcal {N}}: = \{(z,r,\varpi)\in {\mathbb{R}}^3: 0 < z < L,\  0< r  < 1 + \varphi(z)\},
\end{equation}
where, with $ \sigma>0 $ sufficiently small and $ \Theta(\cdot) $ a given function defined on $ [0,\ L] $,
\begin{equation}
\varphi(z) := \int_{0}^{z} \tan(\sigma \Theta(\tau))\dif \tau.
\end{equation}

In this paper, we further assume that the states of the fluid in the nozzle are also axisymmetric  with respect to $ x_1 $-axis such that $ U_{\mathcal {N}} $ are independent of $ \varpi $, and
\[
	-u_2 \sin \varpi + u_3 \cos \varpi = 0.
\]
Then direct computations yield that the 3-D steady Euler equations $\eqref{eq1}$-$\eqref{eq3}$ are reduced to
\begin{equation}\label{eq536}
\left\{
\begin{array}{l}
\partial_z (r \rho u) + \partial_r(r\rho v) = 0,\\
\\
\partial_z (\rho u v) + \partial_r (p+\rho v^2) + \displaystyle\frac{\rho v^2}{r} = 0,\\
\\
\partial_z (p+ \rho u^2) + \partial_r (\rho u v) + \displaystyle\frac{\rho u v }{r} = 0,\\
\\
\partial_z (r \rho u B ) +\partial_r (r \rho v B) =0,
\end{array}
\right.
\end{equation}
where
\begin{equation}
u = u_1,\,\,\,v = u_2 \cos \varpi + u_3 \sin \varpi.
\end{equation}
Hence, it suffices to determine $U \defs (\theta,p,q,s)$, with $\theta = \arctan\displaystyle\frac{v}{u}$ the flow angle and $q = \sqrt{u^2 + v^2}$ the magnitude of the flow velocity, as functions of variables $ (z,r) $.

Moreover, under the cylindrical coordinate, the position of a shock front which is axisymmetric with respect to $ x_1 $-axis can be denoted as $ \set{z=\psi_{\widetilde{\mathcal {N}}}(r)}  $, and the Rankine-Hugoniot conditions $\eqref{eq810}$-$\eqref{eq813}$ across it become
\begin{align}
&[\rho u]-\psi_{\widetilde{\mathcal {N}}}^{'}[\rho v]=0,\label{eq889}\\
&[\rho uv]-\psi_{\widetilde{\mathcal {N}}}^{'}[p+\rho v^2]=0,\\
&[p+\rho u^2]-\psi_{\widetilde{\mathcal {N}}}^{'}[\rho u v]=0,\\
&[B]=0.\label{eq899}
\end{align}

In the $ (z,r) $-space, let
\begin{equation}\label{eq801}
\widetilde{\mathcal{N}}:  = \{ (z, r)\in \mathbb{R}^2 : 0 < z < L,\ 0 < r < 1+\varphi(z)\},
\end{equation}
denotes the domain of the nozzle with the boundaries
\begin{equation}
\begin{aligned}
&E_0 = \{ (z, r)\in \mathbb{R}^2 :  z = 0,\,\,\, 0 < r < 1+ \varphi(0) \},\\
&W_0 = \{ (z, r)\in \mathbb{R}^2 : 0 < z < L,\,\,\,  r = 0\},\\
&E_L= \{ (z, r)\in \mathbb{R}^2 :  z = L,\,\,\, 0 < r < 1+\varphi(L)\},\\
&W_{\varphi} = \{ (z, r)\in \mathbb{R}^2 : 0 < z < L,\,\,\,  r = 1+\varphi(z)\},
\end{aligned}
\end{equation}
in which $ E_0 $ is the entry, $ E_L  $ is the exit, $ W_{\varphi} $ is the nozzle wall, and $W_0 $ denotes the symmetry axis of the nozzle.
In this paper, we are going to determine the steady flow pattern with a single shock front in $ \widetilde{\mathcal{N}}$, satisfying the slip boundary condition on the nozzle wall $ W_{\varphi} $, for given supersonic states at the entry $ E_0 $ and given pressure condition at the exit $ E_L $.
The problem is formulated as a free boundary problem described in detail below.

\vskip 0.5cm

{\bf The Free Boundary Problem {$\textit{FBPC}$}}

 \begin{figure}[!h]
	\centering
	\includegraphics[width=0.45\textwidth]{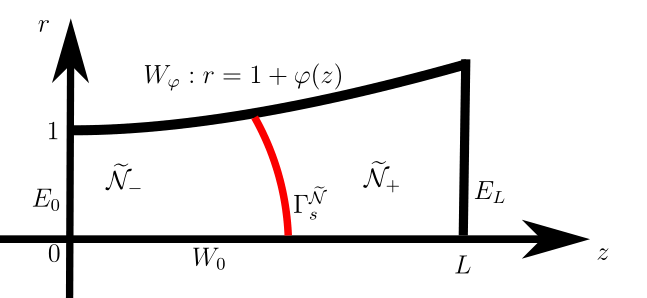}
	\caption{The shock front in the cylindrical coordinate.\label{fig:2}}
\end{figure}

Let $\alpha\in(\frac12,1)$. Given $ \bar{U}_-\defs (0,\bar{p}_-,\bar{q}_-,\bar{s}_-) $, $P_{\mr{e}} \in \mcc^{2,\alpha}(\bar{\mathbb{R}}_+)$, $\Theta \in \mcc^{2,\alpha}([0,L])$, try to determine the states of the fluid $ U $ in the nozzle with a single shock front  $\Gamma_{\mr{s}}^{\widetilde{\mathcal {N}}}\defs \{z = \psi_{\widetilde{\mathcal {N}}}(r)\}$ ( see Figure \ref{fig:2}) such that:
\begin{enumerate}
	\item The nozzle domain $ \widetilde{\mathcal{N}} $ is divided by $ \Gamma_{\mr{s}}^{\widetilde{\mathcal {N}}} $ into two parts:
	\begin{align}
	& \widetilde{\mathcal{N}} _- = \{ (z, r)\in \mathbb{R}^2 : 0 < z < \psi_{\widetilde{\mathcal {N}}} (r), \ 0 < r < 1+ \varphi(z)\},\\
	& \widetilde{\mathcal{N}} _+ = \{ (z, r)\in \mathbb{R}^2 : \psi_{\widetilde{\mathcal {N}}}(r) < z < L,\ 0 < r < 1+ \varphi(z)\},
	\end{align}
	where $ \widetilde{\mathcal{N}} _- $ denotes the region of the supersonic flow ahead of the shock front, and $ \widetilde{\mathcal{N}} _+ $ is the region of the subsonic flow behind it;
	
	\item In $ \widetilde{\mathcal{N}}_- $, the states of the fluid  $ U =  U_-(z,r) $, which satisfies the Euler system  $\eqref{eq536}$, given supersonic state at the entry of the nozzle
	\begin{equation}\label{eq808}
	U_-= \bar{U}_-,\quad \text{on} \quad  E_0,
	\end{equation}
	and the slip boundary condition on the wall of the nozzle
	\begin{equation}\label{eq909}
	\theta_- = \sigma \Theta(z), \quad \text{on} \quad W_{\varphi}\cap\overline{\widetilde{\mathcal{N}}_-};
	\end{equation}
	
	\item In $ \widetilde{\mathcal{N}}_+$, the states of the fluid  $ U=  U_+(z,r) $, which satisfies the Euler system  $\eqref{eq536}$, the slip boundary condition on the wall of the nozzle
	\begin{equation}\label{eq911}
	\theta_+ = \sigma \Theta(z), \quad\text{on} \quad W_{\varphi}\cap\overline{\widetilde{\mathcal{N}}_+},
	\end{equation}
	and given pressure at the exit of the nozzle
	\begin{equation}\label{eq910}
	p_+ = p_{\mr{e}}(L,r) :=  \bar{p}_+ +  \sigma P_{\mr{e}} (r),\quad \text{on}\quad E_L;
	\end{equation}
	
	\item On the shock front $ \Gamma_{\mr{s}}^{\widetilde{\mathcal {N}}} $, the Rankine-Hugoniot conditions $\eqref{eq889}$-$\eqref{eq899}$ hold for the states $(U_- , U_+)$;
	
	\item Finally, on the axis $ \Gamma_{2}$, under the assumption of axisymmetric, both $U_-$ and $U_+$ satisfy
	\begin{equation}\label{eq45}
	\theta_- = 0 , \quad  \partial_r (p_-, q_-, s_-)=0, \quad \partial_r^2 \theta_- = 0, \quad \text{on} \quad W_0\cap\overline{\widetilde{\mathcal{N}}_-},
	\end{equation}
	and
	\begin{equation}\label{eq46}
	\theta_+ = 0,\quad  \partial_r (p_+, q_+, s_+) =0, \quad \partial_r^2 \theta_+ = 0,\quad \text{on} \quad W_0\cap\overline{\widetilde{\mathcal{N}}_+}.
	\end{equation}
\end{enumerate}
\begin{rem}
$\alpha\in (\frac12,1)$ is a sufficient condition, which is needed to establish a prior estimates of the solution. One is referred to Section 3 for details.
\end{rem}

\begin{rem}
The condition \eqref{eq45} guarantees that the compatibility conditions hold for the supersonic solution (see Lemma \ref{lem1} in Section 4). The condition \eqref{eq46} will be verified in Section 3 - Section 6.
\end{rem}
\vskip 0.5cm

This paper will deal with the problem {\bf {$\textit{FBPC}$}} and establish the existence of the transonic shock solution in the 3-D axisymmetric nozzle by showing the following theorem.

\begin{thm}\label{thm:ExistenceNozzleShocks}
	Assume that
	\begin{equation}\label{eq:assumption_001}
		\Theta(z) > 0, \quad \text{for any } z\in(0,L),
	\end{equation}
	and
	\begin{equation}\label{eq20000}
	\Theta(0) = \Theta'(0) = \Theta''(0) = 0.
	\end{equation}

	Let
	\begin{align}
	\mathcal{R}(z) : =& \int_{0}^{L} \Theta(\tau) \dif\tau  - \dot{k}\int_0^{z} \Theta(\tau)\dif\tau, \label{eq:criterion_function}\\
	\mathcal{P}_{\mr{e}} :=& 2\displaystyle\frac{1 - \bar{M}_+^2}{\bar{\rho}_+^2 \bar{q}_+^3} \int_0^{1}  t P_{\mr{e}}(t) \dif t, \label{eq:criterion_pressure_exit}
	\end{align}
 with $\dot{k} : = [\bar{p}]\Big(\displaystyle\frac{\gamma -1}{\gamma \bar{p}_+} + \displaystyle\frac{1}{\bar{\rho}_+ \bar{q}_+^2}\Big) >0$,
	such that
	\begin{equation}\label{eq:85}
	\begin{split}
		&\mathcal{R}^*\defs \sup\limits_{z\in(0,L)}\mathcal{R} (z) = \int_{0}^{L} \Theta(\tau) \dif\tau,\\
		&\mathcal{R}_*\defs \inf\limits_{z\in(0,L)}\mathcal{R}(z) = (1-\dot{k})\int_{0}^{L} \Theta(\tau) \dif\tau.
	\end{split}
	\end{equation}

	Then, if
	\begin{equation}\label{eq:067}
	\mathcal{R}_* < \mcp_{\mr{e}} < \mathcal{R}^*,
	\end{equation}
	then there exists a sufficiently small constant $ \sigma_0 > 0 $, such that for any $ 0<\sigma\leq\sigma_0 $, there exists at least a transonic shock solution $ (U_- , U_+;\ \psi_{\widetilde{\mathcal {N}}}(r)) $ to the free boundary problem {\bf {$\textit{FBPC}$}}.	
\end{thm}

\begin{rem}
	In Theorem \ref{thm:ExistenceNozzleShocks}, the assumption \eqref{eq:assumption_001} is imposed in order to the simplicity of the presentation of this paper. The existence of the transonic shock solutions can also be established for general $ \Theta(\cdot) $ if \eqref{eq:067} holds and there exists $ z_*\in(0,L) $ such that
	\begin{equation}
		\mcr(z_*) = \mcp_{\mr{e}},\quad\text{and }\quad\Theta(z_*)\not=0.
	\end{equation}
	Actually, for general $ \Theta(\cdot) $, similar as the existence of transonic shock solutions in a 2-D nozzle established in \cite{FB63}, there may exist more than one transonic shock solutions to the free boundary problem {\bf {$\textit{FBPC}$}}.
\end{rem}

In dealing with the free boundary problem {\bf {$\textit{FBPC}$}}, one of the key difficulties is determining the position of the shock front.
However, there is no information on it since the problem {\bf {$\textit{FBPC}$}} is going to be solved near the steady plane normal shock solutions and, as pointed out in Remark \ref{rem:plane_normal_shocks}, the position of the shock front can be arbitrary in the flat nozzle.
This difficulty also arises for 2-D transonic shock problem in an almost flat nozzle and Fang and Xin successfully overcome it in \cite{FB63} by designing a free boundary problem of the linearized 2-D Euler system based on the background normal shock solution, which provides information on the position of the shock front as long as it is solvable: the free boundary can be regarded as an initial approximating position of the shock front.
It turns out that this idea also works for the 3-D axisymmetric case studied in this paper, and an initial approximating position of the shock front can be obtained by solving the free boundary problem of the linearized 3-D axisymmetric Euler system based on the plane normal shocks (see Section 4 for details).
Different from the problem for 2-D case in \cite{FB63}, there exist additional 0-order terms and singularities along the symmetric axis in the linearized Euler system for 3-D axisymmetric case. These differences will bring new difficulties in solving the free boundary problem and determining the position of the free boundary.
They need further observation and careful analysis which will be done in this paper.
Once the initial approximation of the transonic shock solution is obtained, nonlinear iteration process similar as in \cite{FB63} can be executed which converges to a transonic shock solution to the problem {\bf {$\textit{FBPC}$}}.

The study on gas flows with shocks in a nozzle plays a fundamental role in the operation of turbines, wind tunnels and rockets.
Thanks to steady efforts made by many mathematicians, there have been plenty of results on it from different viewpoints and for different models, for instance, see \cite{BW2018,CG2006,CG2003,CS2005,CS2008,
CS2009,CR,FB63,LJ2013,WS94,XZ2009,XZ2005,YH2006,YH2007} and references therein.
For steady multi-dimensional flows with shocks in a finite nozzle, in order to determine the position of the shock front, Courant and Friedrichs pointed out in \cite{CR} that, without rigorous mathematical analysis, additional conditions are needed to be imposed at the exit of the nozzle and the pressure condition is preferred among different possible options.
From then on, many mathematicians have been working on this issue and there have been many substantial progresses.
In particular, Chen-Feldman proved in \cite{CG2003} the existence of transonic shock solutions in a finite flat
nozzle for multi-dimensional potential flows with given potential value at the exit and an assumption that the shock front passes through a given point. Later in \cite{CG2006}, with given vertical component of the velocity at the exit, Chen-Chen-Song established the existence of the shock solutions for 2-D steady Euler flows. See also \cite{ParkRyu_2019arxiv} for a recent result for 3-D axisymmetric case. Both existence results are established under the assumption that the shock front passes through a given point, which is employed to deal with the same difficulty as the problem in this paper that the position of the shock front of the unperturbed shock solutions can be arbitrary in the nozzle.
Without such an artificial assumption, recently in \cite{FB63}, Fang-Xin establish the existence of the transonic shock solutions in an almost flat nozzle with the pressure condition at the exit, as suggested by Courant-Friedrichs in \cite{CR}. It is interesting that the results in \cite{FB63} indicate that, for a generic nozzle and given pressure condition at the exit, there may exist more than one shock solutions, that is, there may exist more than one admissible positions of the shock front.
It should be noted that, in a diverging nozzle which is an expanding angular sector, the position of the shock front can be uniquely determined by the pressure condition at the exit under the assumption that the flow states depend only on the radius (see \cite{CR}). And the structural stability of this shock solution under small perturbation of the nozzle boundary as well as the pressure condition at the exit has been established for 2-D case in a series of papers \cite{LiXinYin2009MRL}, \cite{LJ2013} by Li-Xin-Yin and in \cite{CS2009} by Chen. See also \cite{Yong,WS94} for a recent advance towards 3-D axisymmetric case.

\subsection{Organization of the paper.}

The paper is organized as follows. In Section 2, the problem {\bf {$\textit{FBPC}$}} is reformulated by a modified Lagrange transformation, introduced by Weng-Xie-Xin in \cite{WS94}, which straightens the stream line without the degeneracy along the symmetric axis. Then the free boundary problem for the linearized Euler system is described, which serves to determine the initial approximation. Finally, the main theorem to be proved is given.
In Section 3, we shall establish a well-posed theory for boundary value problem of the elliptic sub-system of the linearized Euler system at the subsonic state behind the shock front. It turns out that there exists a solution to the problem if and only if a solvability condition is satisfied for the boundary data. This solvability condition will be employed to determine the position of the free boundary.
In Section 4, we prove the existence of the initial approximation by applying the theorem proved in Section 3. Then a nonlinear iteration scheme will be described, starting from the initial approximation, in Section 5. Finally, in Section 6, the nonlinear iteration scheme will be verified to be well-defined and contractive, which concludes the proof for the main theorem.

\section{Reformulation by Lagrange Transformation and Main Results}

For 2-D steady Euler system, it is convenient to introduce the Lagrange transformation which straighten the streamline (see, for instance, \cite{CS2005,FB63,LJ2013}). The idea also applies to the 3-D steady axisymmetric Euler system. However, degeneracy occurs along the symmetric axis such that it is not invertible.
Weng-Xie-Xin introduced in \cite{WS94} a modified Lagrange transformation which successfully overcame this difficulty. We are going to apply this modified Lagrange transformation to reformulate the problem {\bf {$\textit{FBPC}$}}.

\subsection{The modified Lagrange transformation.}

For the 3-D steady axisymmetric Euler system, the Lagrange transformation is defined as, with $\eta(z,0) = 0$,
\begin{equation}
\left\{
\begin{aligned}
&\xi = z,\\
&\eta = \int_{(0,0)}^{(z,r)} t \rho u(s,t)\dif t - t \rho v(s,t)\dif s.
\end{aligned}
\right.
\end{equation}
It can be easily verified that its Jacobi matrix degenerates along the symmetric axis $ r=0 $, which yields that it is not invertible. This difficulty can be overcome with a modification introduced in \cite{WS94} by Weng-Xie-Xin, which will be used in this paper.

The modification is as follows. Let
\begin{equation}
\left\{
\begin{aligned}
&\tilde{\xi} = \xi,\\
&\tilde{\eta} = \eta^{\frac12}.
\end{aligned}
\right.
\end{equation}
Then the Jacobian of the modified transformation is
\[ \displaystyle\frac{\partial(\tilde{\xi},\tilde{\eta})}{\partial(z,r)}= \left|\begin{matrix}
\partial_z \tilde{\xi}& \partial_r\tilde{\xi}\\
\partial_z \tilde{\eta}& \partial_r\tilde{\eta}
\end{matrix}\right|=
\left|\begin{matrix}
1& 0\\
-\displaystyle\frac{r \rho v}{2\tilde{\eta}} & \displaystyle\frac{r \rho u}{2\tilde{\eta}}
\end{matrix}\right| = \displaystyle\frac{r \rho u}{2\tilde{\eta}}.\]
For the background steady plane normal shock solution  $\bar{U}_\pm$, it is easily seen that
\begin{equation}
{\eta }(z,r) = \int_{0}^{r} t \bar{\rho}_\pm \bar{q}_\pm \dif t = \displaystyle\frac12 \bar{\rho}_\pm \bar{q}_\pm r^2,
\end{equation}
such that $ \tilde{\eta} (z,r) = r $, where it is assume that, without loss of generality, $\bar{\rho}_+\bar{q}_+ =\bar{\rho}_-\bar{q}_-=2$.  Therefore, it can be anticipated that, if  $U_\pm$ is close to the background solution $\bar{U}_\pm$, then there exist positive constants  $C_1$ and  $C_2$, depending only on  $\bar{U}_\pm$, such that $C_1 r \leq \tilde{\eta} (z,r) \leq C_2 r$, which implies that there exists a constant  $ C_3 $, depending on $\bar{U}_\pm$, such that the Jacobian of the modified Lagrange transformation satisfies
\begin{equation}
\displaystyle\frac{\partial(\tilde{\xi},\tilde{\eta})}{\partial(z,r)} \geq C_3 > 0,
\end{equation}
that is, it does not degenerate and is invertible.

Under the modified Lagrange transformation, we have $r(\tilde{\xi},0) = 0$ and
\begin{align}
\dif z = \dif \tilde{\xi},\quad \dif r = \displaystyle\frac{2\tilde{\eta}}{r\rho u}\dif \tilde{\eta} + \displaystyle\frac{v}{u}\dif \tilde{\xi}.
\end{align}
It follows that
\begin{align}
r = r(\tilde{\xi}, \tilde{\eta})= \left(2 \int_{0}^{\tilde{\eta}}\displaystyle\frac{2 t }{\rho u (\tilde{\xi}, t)} \dif t\right)^{\frac12}.
\end{align}
In particular, at the background state, one has
\begin{align}
  \bar{r} = \Big(2 \int_{0}^{\tilde{\eta}}\displaystyle\frac{2 t }{\bar{\rho}_\pm \bar{q}_\pm} \dif t\Big)^{\frac12} = \sqrt{\frac{2}{\bar{\rho}_\pm \bar{q}_\pm}}\tilde{\eta}.
\end{align}

Then, it follows after direct computations that, under the modified Lagrange transformation, the Euler  system $\eqref{eq536}$ becomes
\begin{equation}\label{eq3927}
\begin{aligned}
&\partial_{\tilde\xi} \Big(\displaystyle\frac{2{\tilde\eta}}{r\rho u}\Big) - \partial_{\tilde\eta} \Big(\displaystyle\frac{v}{u}\Big) = 0,\\
&\partial_{\tilde\xi} v + \displaystyle\frac{r}{2{\tilde\eta}}\partial_{\tilde\eta} p = 0,\\
&\partial_{\tilde\xi} \Big(u+ \displaystyle\frac{p}{\rho u}\Big) - \displaystyle\frac{r}{2{\tilde\eta}} \partial_{\tilde\eta} \Big(\displaystyle\frac{pv}{u}\Big) - \displaystyle\frac{pv}{r\rho u^2} = 0,\\
&\partial_{\tilde\xi} B =0.
\end{aligned}
\end{equation}
For simplicity of the notations, we drop `` $ \tilde{} $ '' hereafter as there is no confusion taking place.

Further computation yields that the system \eqref{eq3927} can be rewritten as
\begin{align}
&\partial_{\eta} p - \displaystyle\frac{2{\eta}}{r}\cdot \displaystyle\frac{\sin\theta}{\rho q} \partial_{\xi} p + \displaystyle\frac{2{\eta}}{r}\cdot q \cos\theta\cdot \partial_{\xi} \theta = 0,\label{eq2916}\\
&\partial_{\eta} \theta - \displaystyle\frac{2{\eta}}{r}\cdot\displaystyle\frac{\sin\theta}{\rho q}\partial_{\xi} \theta + \displaystyle\frac{2{\eta}}{r}\cdot\displaystyle\frac{\cos\theta}{\rho q}\displaystyle\frac{M^2 - 1}{\rho q^2} \partial_{\xi} p + \displaystyle\frac{2{\eta} }{r^2}\cdot\displaystyle\frac{\sin\theta}{\rho q} = 0,\label{eq2917}\\
&\rho q \partial_{\xi} q + \partial_{\xi} p = 0,\label{eq2918}\\
&\partial_{\xi} s = 0,\label{eq2919}
\end{align}
where the equation \eqref{eq2918} can also be replaced by the following equation in the conservative form:
\begin{equation}\label{eq2902}
\partial_{{\xi}} {B} = 0.
\end{equation}

\begin{rem}
	It is easy to see that the equations \eqref{eq2918} and \eqref{eq2919} are transport equations which are hyperbolic. Moreover, the equations $\eqref{eq2916}$ and $\eqref{eq2917}$ can be rewritten in the matrix form as below:
	\begin{equation}\label{eq080}
	\partial_{\eta} (\theta, p )^T + A(U)\partial_{\xi} (\theta, p)^T + a(U)= 0,
	\end{equation}
	where $a(U) = \Big( \displaystyle\frac{2{\eta}}{r^2}\cdot\displaystyle\frac{\sin\theta}{\rho q},\, 0\Big)^T$, and
	
	\[ A(U) = \displaystyle\frac{2{\eta}}{r}\displaystyle\frac{1}{\rho q}\begin{pmatrix}
	 -\sin\theta&\displaystyle\frac{M^2 -1}{\rho q^2}\cos\theta\\     \rho q^2\cos\theta&
	 -\sin\theta
	\end{pmatrix}. \]	
	Direct computations yield that the eigenvalues of $A(U)$ are
	\begin{equation}
	\lambda_\pm = \displaystyle\frac{2{\eta}}{r}\displaystyle\frac{1}{\rho q}\Big(- \sin\theta \pm \sqrt{M^2 -1}\cos\theta\Big),
	\end{equation}
	which yields that, for supersonic flows with the Mach number $M>1$ such that $ \lambda_\pm $ are real, the system $\eqref{eq080}$ is hyperbolic, while for subsonic flows with $M<1$ such that $ \lambda_\pm $ are a pair of conjugate complex numbers, the system $\eqref{eq080}$ is elliptic. Therefore, the system \eqref{eq2916}-\eqref{eq2919} is hyperbolic as $ M>1 $, while it is elliptic-hyperbolic composite as $ M<1 $.
\end{rem}

Let ${\Gamma}_{\mr{s}} \defs\{ ({{\xi}}, {{\eta}})\in \mathbb{R}^2 : {\xi} = {\psi}({\eta}),\,\,0 <{\eta} <1 \}$ be the position of a shock front, under the modified Lagrange transformation, the Rankine-Hugoniot conditions  $\eqref{eq889}$-$\eqref{eq899}$ across the shock front become
\begin{align}
&\displaystyle\frac{2\eta}{r}\Big[\displaystyle\frac{1}{\rho u}\Big] + \psi'\Big[\displaystyle\frac{v}{u}\Big] = 0,\label{eq2903}\\
&[v] - \psi' \displaystyle\frac{r}{2\eta} [p] = 0,\label{eq2904}\\
&\Big[ u+ \displaystyle\frac{p}{\rho u}\Big] + \psi' \displaystyle\frac{r}{2\eta} \Big[\displaystyle\frac{p v}{u}\Big] =0,\label{eq2905}\\
&[B] = 0.\label{eq2906}
\end{align}
Using the equation $\eqref{eq2904}$, we can eliminate the quantity $\psi'$ in the equations $\eqref{eq2903}$ and $\eqref{eq2905}$ respectively to obtain
\begin{align}\label{eq82}
&G_1(U_+, U_-) : = \Big[\displaystyle\frac{1}{\rho u}\Big][p] + \Big[\displaystyle\frac{v}{u}\Big][v]= 0,\\
&G_2(U_+, U_-) : = \Big[u+ \displaystyle\frac{p}{\rho u}\Big] [p]  + \Big[\displaystyle\frac{p v}{u}\Big][v]=0,
\end{align}
moreover, we denote
\begin{align}
&G_3(U_+, U_-): = \Big[\frac12 q^2+i\Big]=0,\label{eq060}\\
&G_4(U_+, U_-;\psi'): = [v] - \psi'\displaystyle\frac{r}{2\eta}[p]=0.\label{eq3999}
\end{align}

Under the modified Lagrange transformation, the domain $\widetilde{\mathcal {N}}$ becomes
\begin{equation}
\begin{aligned}
{\Omega} = \{ ({\xi}, {\eta})\in \mathbb{R}^2 : 0 < {\xi} < L, \,\,\,0 < {\eta} < 1\},
\end{aligned}
\end{equation}
and it is separated by a shock front ${\Gamma}_{\mr{s}} $
into two parts (see Figure \ref{fig:3}):
the supersonic region and subsonic region, denoted by
\begin{equation}\label{eq:supersonic_region_Lagrange}
{\Omega}_- = \{ ({\xi}, {\eta})\in \mathbb{R}^2 : 0 < {\xi} < \psi({\eta}), \,\,\,0 < {\eta} < 1\},
\end{equation}
\begin{equation}
{\Omega}_+ = \{ ({\xi}, {\eta})\in \mathbb{R}^2 : \psi({\eta}) < {\xi} < L, \,\,\,0 < {\eta} < 1\}
\end{equation}
respectively.
The boundaries $E_0$, $W_0$, $E_L$, $W_\varphi$ become
\begin{equation}\label{eq73}
\begin{aligned}
&{\Gamma}_1 = \{ ({{\xi}}, {{\eta}})\in \mathbb{R}^2 : {\xi} =0,\,\,\, 0 <{\eta} <1 \},\\
&{\Gamma}_2 = \{ ({{\xi}}, {{\eta}})\in \mathbb{R}^2 :  0 <{\xi} < L, \,\,\,{\eta} =0 \},\\
&{\Gamma}_3 = \{ ({{\xi}}, {{\eta}})\in \mathbb{R}^2 : {\xi}  = L ,\,\,\, 0<{\eta} <1 \},\\
&{\Gamma}_4 = \{ ({{\xi}}, {{\eta}})\in \mathbb{R}^2 :  0 < {\xi} < L , \,\,\,{\eta} =1 \}.
\end{aligned}
\end{equation}
Thus, the free boundary problem {\bf {$\textit{FBPC}$}} is reformulated as follows under the modified Lagrange transformation.

\vskip 0.5cm

{\bf The Free Boundary Problem {$\textit{FBPL}$}}

\begin{figure}[!h]
	\centering
	\includegraphics[width=0.45\textwidth]{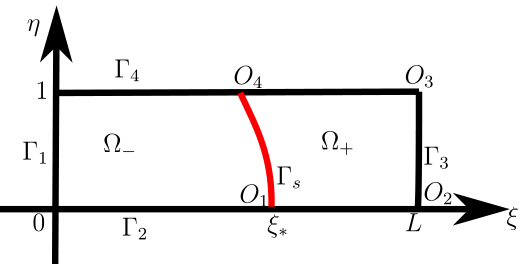}
	\caption{The transonic shock flows in the Lagrangian coordinate.\label{fig:3}}
\end{figure}

Try to determine $ (U_{-}(\xi,\eta),\ U_{+}(\xi,\eta),\ \psi(\eta)) $ in $ \Omega $ such that:
\begin{enumerate}
	\item $ U_{-}(\xi,\eta) $ satisfies the equations  \eqref{eq2916}-\eqref{eq2919} in $ \Omega_- $, and the following boundary conditions:
	\begin{align}
	&U_-= \bar{U}_-,&\quad  &\text{on} \,\,\,  \Gamma_1,\label{eq808}\\
	&\theta_- = \sigma \Theta(\xi).&\quad &\text{on} \,\,\, \Gamma_{4}\cap\overline{\Omega_-}.\label{eq909}
	\end{align}

	\item  $ U_{+}(\xi,\eta) $ satisfies the equations  \eqref{eq2916}-\eqref{eq2919} in $ \Omega_+ $, and the following boundary conditions:
	\begin{align}
	&\theta_+ = \sigma \Theta(\xi), &\quad &\text{on} \quad \Gamma_{4}\cap\overline{\Omega_+},\label{eq911}\\
	&p_+ = p_{\mr{e}}(L,\eta) := \bar{p}_+ +  \sigma P_{\mr{e}}  (r(L,\eta)),&\quad &\text{on}\quad \Gamma_3,\label{eq910}
	\end{align}
	where
	\begin{equation}
	r(L,\eta) =  \Big(2 \int_{0}^{{\eta}}\displaystyle\frac{2 t }{\rho u (L, t)} \dif t\Big)^{\frac12}.
	\end{equation}
	
	\item On the shock front $ \Gamma_{\mr{s}} $, the Rankine-Hugoniot conditions $\eqref{eq2903}$-$\eqref{eq2906}$ hold for the states $(U_{-} , U_{+})$;
	
	\item Finally, on $ \Gamma_{2} $, both $U_{-}$ and $ U_{+}$ satisfy
	\begin{equation}\label{45}
	\theta_- = 0 , \quad \partial_\eta (p_{-}, q_-, s_-)=0,\quad \partial_\eta^2\theta_- = 0, \quad \text{on} \quad \Gamma_{2}\cap\overline{\Omega_-},
	\end{equation}
	and
	\begin{equation}\label{46}
	\theta_+ = 0 , \quad \partial_\eta (p_+, q_+, s_+) = 0,\quad \partial_\eta^2\theta_+ = 0,\quad
 \text{on} \quad \Gamma_{2}\cap\overline{\Omega_+}.
	\end{equation}
	
\end{enumerate}

\vskip 0.5cm

In this paper, the free boundary problem {\bf $\textit{FBPL}$} is going to be solved near the background solution  $\bar{U}_\pm$.
Once it is solved, the existence of shock solutions to the problem  {\bf $\textit{FBPC}$} can be established since the modified Lagrange transformation is invertible.

\subsection{The free boundary problem for the initial approximation.}

To solve the free boundary problem {\bf $\textit{FBPL}$}, one of the key step is to obtain information on the position of the shock front. Motivated by the ideas in \cite{FB63}, we are going to design a free boundary problem for the linearized Euler system based on the background shock solution $\bar{U}_\pm$, whose solution could serve as an initial approximation.

\begin{figure}[!h]
	\centering
	\includegraphics[width=0.45\textwidth]{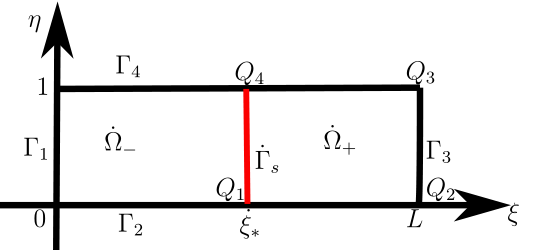}
	\caption{The domain for the linearized problem.\label{fig:4}}
\end{figure}

Assume the initial approximating position of the shock front is
\begin{equation}
\dot{\Gamma}_{\mr{s}}= \{(\xi,\eta) : \xi = \dot{\xi}_*, \,\,\,0< \eta <1\},
\end{equation}
where $0<\dot{\xi}_*<L$ is unknown and will be determined later (see Figure \ref{fig:4}).
Then the whole domain $ \Omega $ is divided by $ \dot{\Gamma}_{\mr{s}} $ into two parts:
\begin{align}
&\dot{\Omega}_- = \{ (\xi, \eta)\in \mathbb{R}^2 : 0 < \xi < \dot{\xi}_*, \,\,\,0 < \eta < 1\},\\
&\dot{\Omega}_+ = \{ (\xi, \eta)\in \mathbb{R}^2 : \dot{\xi}_* <\xi < L,\,\,\, 0 < \eta < 1\},
\end{align}
where $ \dot{\Omega}_- $ is considered as the initial approximating domain of the supersonic flow ahead of the shock front, and $ \dot{\Omega}_+ $ is the subsonic flow behind the shock front.
Let $ \dot{U}_- =  (\dot{\theta}_-, \dot{p}_-, \dot{q}_-, \dot{s}_-)^{T}$ defined in  $ \dot{\Omega}_- $ satisfies the linearized Euler system at the supersonic state $ \bar{U}_- $ below which will be taken as the initial approximation for the supersonic flow ahead of the shock front:
\begin{align}
&\partial_\eta \dot{p}_- + 2 \bar{q}_- \partial_\xi \dot{\theta}_- = 0,\label{eq80-}\\
&\partial_\eta \dot{\theta}_- + 2 \displaystyle\frac{\bar{M}_{-}^2 -1} {\bar{\rho}_{-}^2 \bar{q}_{-}^3} \partial_\xi \dot{p}_-  + \displaystyle\frac{1}{\eta}\dot{\theta}_-= 0,\label{eq843-}\\
&\bar{\rho}_- \bar{q}_- \partial_\xi \dot{q}_- + \partial_\xi \dot{p}_- = 0,\label{eq844-}\\
&\partial_\xi \dot{s}_- = 0.\label{eq845-}
\end{align}
Moreover, let $ \dot{U}_+ =  (\dot{\theta}_+, \dot{p}_+, \dot{q}_+, \dot{s}_+)^{T}$ defined in  $ \dot{\Omega}_+ $ satisfies the linearized Euler system at the subsonic state $ \bar{U}_+ $ below which will be taken as the initial approximation for the subsonic flow behind the shock front:
\begin{align}
&\partial_\eta \dot{p}_+ + 2 \bar{q}_+ \partial_\xi \dot{\theta}_+ = 0,\label{eq80+}\\
&\partial_\eta \dot{\theta}_+ + 2 \displaystyle\frac{\bar{M}_{+}^2 -1} {\bar{\rho}_{+}^2 \bar{q}_{+}^3} \partial_\xi \dot{p}_+  + \displaystyle\frac{1}{\eta} \dot{\theta}_+= 0,\label{eq843+}\\
&\partial_\xi\Big(\bar{q}_+ \dot{q}_+ + \displaystyle\frac{1}{\bar{\rho}_+} \dot{p}_+ + \bar{T}_+ \dot{s}_+ \Big)= 0,\label{eq844+}\\
&\partial_\xi \dot{s}_+ = 0.\label{eq845+}
\end{align}

Then the following free boundary problem will be served to determine an initial approximation $\dot{U}_-$, $\dot{U}_+$, $\dot{\xi}_*$, and together with the updated approximating shock profile $\dot{\psi}'$.

\vskip 0.5cm

{\bf The Free Boundary Problem {$\textit{IFBPL}$} for the initial approximation}

Try to determine $ (\dot{U}_{-}(\xi,\eta),\ \dot{U}_{+}(\xi,\eta),\ \dot{\xi}_*;\ \dot{\psi}'(\eta)) $ in $ \Omega $ such that:
\begin{enumerate}
	\item $ \dot{U}_{-}(\xi,\eta) $ satisfies the equations \eqref{eq80-}-\eqref{eq845-} in $ \dot\Omega_- $, and the following boundary conditions:
	\begin{align}
	&\dot{U}_- = 0,&\quad &\text{on} \quad \Gamma_1,\label{eq8561}\\
	&\dot\theta_- = \sigma \Theta(\xi), &\quad &\text{on} \quad \Gamma_{4}\cap\overline{\dot\Omega_-}.\label{eq8564}
	\end{align}

	\item  $ \dot{U}_{+}(\xi,\eta) $ satisfies the equations  \eqref{eq80+}-\eqref{eq845+} in $ \dot{\Omega}_+ $, and the following boundary conditions:
	\begin{align}
	&\dot{\theta}_+ = \sigma \dot{\Theta}\defs \sigma {\Theta}(\xi), &\quad &\text{on} \quad \Gamma_{4}\cap\overline{\dot\Omega_+},\label{eq8584}\\
	&\dot{p}_+ = \sigma \dot{P}_{\mr{e}} \defs \sigma P_{\mr{e}}  (\eta),&\quad &\text{on}\quad \Gamma_3.\label{eq8583}
	\end{align}

	\item On the free boundary $ \dot{\Gamma}_{\mr{s}} $,  $\dot{U}_{-}$ and  $\dot{U}_{+}$ satisfy the following linearized R-H conditions: 	
	\begin{align}
	&{\mathbf{\alpha}}_{j+} \cdot {\dot{U}}_+ + {\mathbf{\alpha}}_{j-} \cdot {\dot{U}}_- = 0,  (j = 1,2,3),\label{eq090}\\
	&{\mathbf{\alpha}}_{4+} \cdot {\dot{U}}_+ + {\mathbf{\alpha}}_{4-} \cdot {\dot{U}}_- - \frac12[\bar{p}]{\dot{\psi}}' = 0,\label{eq021}
	\end{align}
	where
	\begin{equation}\label{eq86}
	\begin{aligned}
	&{\mathbf{\alpha}}_{j\pm} = {{\nabla_{{U}_\pm}}}G_j(\bar{U}_+, \bar{U}_-),\,\,\, {\mathbf{\alpha}}_{4\pm} = {{\nabla_{{U}_\pm}}}G_4(\bar{U}_+, \bar{U}_-;0).
	\end{aligned}
	\end{equation}

	\item Finally, on $ \dot{\Gamma}_{2} $, both $\dot{U}_{-}$ and $ \dot{U}_{+}$  satisfy
	\begin{equation}\label{eq451}
	\dot\theta_- = 0 , \quad \partial_\eta (\dot{p}_{-}, \dot{q}_-, \dot{s}_-)=0, \quad \partial_\eta^2 \dot{\theta}_{-}=0,\quad \text{on} \quad \Gamma_{2}\cap\overline{\dot\Omega_-},
	\end{equation}
	and
	\begin{equation}\label{eq461}
	\dot{\theta}_+ = 0, \quad \partial_\eta (\dot{p}_+ , \dot{q}_+, \dot{s}_+)= 0,\quad\partial_\eta^2 \dot{\theta}_{+}=0,\quad \text{on} \quad \Gamma_{2}\cap\overline{\dot\Omega_+}.
	\end{equation}
\end{enumerate}

\vskip 0.5cm

\begin{rem}\label{rem:linearized_RH}
	It should be noted that direct computations yield that the coefficients of the linearized R-H conditions \eqref{eq090} and \eqref{eq021} on the free boundary  $ \dot{\Gamma}_{\mr{s}} $ have explicit forms given below:
	 \begin{align}
	& {\mathbf{\alpha}}_{1\pm} = \pm \displaystyle\frac{[\bar{p}]}{\bar{\rho}_\pm \bar{q}_\pm}\Big(0 ,\, -\displaystyle\frac{1}{\bar{\rho}_\pm \bar{c}_\pm^2},\, -\displaystyle\frac{1}{\bar{q}_\pm},\, \displaystyle\frac{1}{\gamma c_v}\Big)^T,\label{eq0100}\\
	& {\mathbf{\alpha}}_{2\pm} = \pm \displaystyle\frac{[\bar{p}]}{\bar{\rho}_\pm \bar{q}_\pm} \Big(0 ,\, 1 -\displaystyle\frac{\bar{p}_\pm}{\bar{\rho}_\pm \bar{c}_\pm^2},\, \bar{\rho}_\pm \bar{q}_\pm -\displaystyle\frac{\bar{p}_\pm}{\bar{q}_\pm},\, \displaystyle\frac{\bar{p}_\pm}{\gamma c_v}\Big)^T,\\
	 &{\mathbf{\alpha}}_{3\pm} = \pm \Big(0 ,\, \displaystyle\frac{1}{\bar{\rho}_\pm },\, {\bar{q}_\pm}, \, \displaystyle\frac{1}{(\gamma -1)c_v}\cdot \displaystyle\frac{\bar{p}_\pm}{\bar{\rho}_\pm}\Big)^T, \label{eq101}\\
	& {\mathbf{\alpha}}_{4\pm} = \pm \Big(\bar{q}_\pm ,\, 0,\,0 ,\, 0\Big)^T\label{s}.
	 \end{align}
Moreover, by taking $ \dot{U}_{-} $ as known data, the equations \eqref{eq090} form a closed linear algebraic equations for $ (\dot{p}_+, \dot{q}_+, \dot{s}_+) $ such that they can be expressed by $ \dot{U}_{-} $ on the free boundary  $ \dot{\Gamma}_{\mr{s}} $.
	 And it turns out that, in order to determine $\dot{U}_-$, $\dot{U}_+$, and $\dot{\xi}_*$, it is sufficient to only impose \eqref{eq090} on  $ \dot{\Gamma}_{\mr{s}} $. This yields that the condition \eqref{eq021} is only employed to determine $\dot{\psi}'$, which will be used in the next step of the iteration.
\end{rem}

\begin{rem}
	Since $ \bar{U}_+ $ is a subsonic state, the sub-system \eqref{eq80+}-\eqref{eq843+} is an elliptic system of first order and there may exists no solutions for its boundary value problem unless the prescribed boundary data satisfy certain solvability conditions, which will be used to determine the unknown variable $\dot{\xi}_*$ and the free boundary $ \dot{\Gamma}_{\mr{s}} $. The mechanism is similar as the 2-D problem in \cite{FB63}.
	However, different from the linearized elliptic sub-system in \cite{FB63}, there exists an additional lower order term with variable coefficients in \eqref{eq843+} and singularity occurs as $ \eta=0 $.
	These differences bring new difficulties in formulating the solvability condition and establishing the existence of the solution to the problem {\bf{$\textit{IFBPL}$}}. More efforts need to be made to deal with them, which will be done in this paper.
\end{rem}

\subsection{Main results.}

In this paper, we are going to solve the free boundary problem  {\bf{$\textit{IFBPL}$}} and obtain an initial approximation of the shock solution. Then a nonlinear iteration scheme based on this initial approximation can be carried out. The iteration scheme will be shown to be convergent and the limit is a shock solution to the problem  {\bf{$\textit{FBPL}$}}.

Before describing the main theorems in detail, we first introduce the function spaces for the solutions and associated norms.
For the hyperbolic part of the problem, it is natural to use the classical H\"{o}lder spaces. Let $\Omega \subset \mathbb{R}^n$ be a bounded domain, $k\geq 0$ be an integer, and $0< \alpha < 1$, $\mcc^{k,\alpha}$ denote the classical H\"{o}lder spaces with the index $(k, \alpha)$ for functions with continuous derivatives up to $k$-th order, equipped with the classical $\mcc^{k,\alpha}$ norm:
\begin{equation}
   \|u\|_{\mcc^{k,\alpha}(\Omega)} : = \sum_{|\mathbf{m}|\leq k} \sup\limits_{\mathbf{x}\in\Omega}|D^{\mathbf{m}} u(\mathbf{x})|+ \sum_{|\mathbf{m}| =  k} \sup\limits_{\mathbf{x},\mathbf{y}\in\Omega; \mathbf{x}\neq \mathbf{y}}\frac{|D^{\mathbf{m}} u(\mathbf{x}) - D^{\mathbf{m}} u(\mathbf{y})|}{|\mathbf{x}-\mathbf{y}|^\alpha},
\end{equation}
where $D^{\mathbf{m}} = \partial_{x_1}^{m_1}\partial_{x_2}^{m_2}\cdots\partial_{x_n}^{m_n}$, and
$\mathbf{m}=(m_1,m_2,\ldots, m_n)$ is a multi-index with $m_i\geq 0$ an integer and $|\mathbf{m}| = \sum\limits_{i =1}^n m_i$.
For elliptic part of the problem, since the boundary of the domain has corner singularities, weighted H\"{o}lder norms will be employed.
 Let $\Gamma$ be an open portion of $\partial\Omega$, for any $\mathbf{x}$, $\mathbf{y}$ $\in$ $\Omega$, define
\begin{equation}
d_{\mathbf{x}}: = \dist(\mathbf{x}, \Gamma),\,\,\, \text{and} \,\,\,
d_{\mathbf{x},\mathbf{y}}:= \min (d_{\mathbf{x}},d_{\mathbf{y}}).
\end{equation}
Let $\alpha\in(0,1)$ and $\delta\in \mathbb{R}$, we define:
\begin{align}
&[u]_{k ,0;\Omega}^{(\delta;\Gamma)} : = \sup\limits_{\mathbf{x}\in\Omega, |\mathbf{m}| = k}\Big(d_{\mathbf{x}}^{\max(k + \delta,0)}|D^{\mathbf{m}} u(\mathbf{x})| \Big),\\
&[u]_{k,\alpha;\Omega}^{(\delta;\Gamma)} : = \sup\limits_{{\mathbf{x},\mathbf{y}}\in\Omega, \mathbf{x}\neq\mathbf{y}, |\mathbf{m}| = k}\left(d_{\mathbf{x},\mathbf{y}}^{\max( k +\alpha+\delta,0)}\displaystyle\frac
{|D^{\mathbf{m}} u(\mathbf{x})- D^{\mathbf{m}} u(\mathbf{y})|}{|\mathbf{x} - \mathbf{y}|^\alpha}\right),\\
& \| u \|_{k,\alpha;\Omega}^{(\delta;\Gamma)} :  =\sum_{i =0}^{k} [u]_{i,0;\Omega}^{(\delta;\Gamma)} + [u]_{k,\alpha;\Omega}^{(\delta;\Gamma)},
\end{align}
with the corresponding function space defined as
\begin{equation}
 \mathbf{H}_{k,\alpha}^{(\delta;\Gamma)}(\Omega): = \{u: \| u \|_{k,\alpha;\Omega}^{(\delta;\Gamma)} < +\infty\}.
\end{equation}
Moreover, since the Euler system for subsonic flows is elliptic-hyperbolic composite, for the flow state $ U=(\theta,p,q,s)^{T} $, the function spaces for $ (\theta,p) $ are different from $ (q,s) $.

Define
\begin{equation}\label{eq099}
 \|{U} \|_{(\dot{\Omega}_+;\dot{\Gamma}_{\mr{s}})}: = \| (\theta, {p}) \|_{1,\alpha;\dot{\Omega}_+}^{(-\alpha;\{Q_3,Q_4\})}
+ \| ({q} ,{s}) \|_{1,\alpha;\dot{\Gamma}_{\mr{s}}}^{(-\alpha;Q_4)}+ \|({q} ,{s}) \|_{0,\alpha;\dot{\Omega}_+}^{(1-\alpha;\{Q_3,Q_4\})},
\end{equation}
where the points $Q_3$ and $Q_4$ are defined in Figure \ref{fig:4}.

Since the shock front $\Gamma_{\mr{s}}\defs \{{\xi} = {\psi}({\eta})\}$ is a free boundary. To fix the shock front, we introduce the following coordinate transformation
\begin{align*}
\mathcal{T} : \begin{cases}
\tilde{\xi} = L + \displaystyle\frac{L - \dot{\xi}_*}{L - {\psi}(\eta)}(\xi - L),\\
\tilde{\eta} = \eta,
\end{cases}
\end{align*}
with the inverse
\begin{align*}
\mathcal{T}^{-1} : \begin{cases}
\xi = L + \displaystyle\frac{L -  {\psi}(\tilde{\eta})}{L - \dot{\xi}_*}(\tilde{\xi} - L),\\
\eta = \tilde{\eta}.
\end{cases}
\end{align*}
Obviously, under this transformation, the free boundary $\Gamma_{\mr{s}}$ is changed into the fixed boundary $\dot{\Gamma}_{\mr{s}}$. Correspondingly, the domain ${\Omega}_+$ becomes the fixed domain $\dot{\Omega}_+$ (see Figure \ref{fig:3}).

Therefore, we define the norm of $U$ in the domain ${\Omega}_+$ as below:
\begin{equation}
\|{U} \|_{({\Omega}_+;{\Gamma}_{\mr{s}})}: = \| {U}\circ \mathcal{T}^{-1}\|_{(\dot{\Omega}_+;\dot{\Gamma}_{\mr{s}})}.
\end{equation}

Now we are going to describe the main theorems of this paper. We first give a theorem showing the existence of the solution to the free boundary problem {\bf{$\textit{IFBPL}$}}.

\begin{thm}\label{thm:initial_approx_existence}
Let $\alpha\in (\frac12,1)$. Suppose the assumptions in Theorem \ref{thm:ExistenceNozzleShocks} hold. Then there exists a unique solution $ (\dot{U}_{-}(\xi,\eta),\ \dot{U}_{+}(\xi,\eta),\ \dot{\xi}_*;\ \dot{\psi}'(\eta)) $ to the free boundary problem {\bf{$\textit{IFBPL}$}}, where the unknown constant $ 0<\dot{\xi}_*<L $ and the free boundary $ \dot{\Gamma}_{\mr{s}} $ is determined by the following equation:
	\begin{equation}\label{eq:initial_approx_free_bdry}
		\mcr(\dot{\xi}_*) = \mcp_{\mr{e}},
	\end{equation}
	where the function $ \mcr(\cdot) $ is defined in \eqref{eq:criterion_function}, the constant $ \mcp_{\mr{e}} $ is defined in \eqref{eq:criterion_pressure_exit}, and \eqref{eq:067} holds.

	Moreover, it holds that
	\begin{align}
	&\|\dot{U}_-\|_{\mcc^{2,\alpha}(\bar{\Omega})} \leq \dot{C}_{-}\sigma,\label{eq061}\\
	&\| \dot{U}_+ \|_{(\dot{\Omega}_+;\dot{\Gamma}_{\mr{s}})} +\| \dot{\psi}'\|_{1,\alpha;\dot{\Gamma}_{\mr{s}}}^{(-\alpha;Q_4)}
	 \leq \dot{C}_+ \sigma,\label{eq089}
	\end{align}
where the constants $\dot{C}_\pm$ depend on $\bar{U}_\pm$, $L$, $\dot{\xi}_*$ and $\alpha$.
\end{thm}

\begin{rem}
	Similar as the 2-D problem in \cite{FB63}, for a generic function $ \Theta $, the existence of the solution to the free boundary problem {\bf{$\textit{IFBPL}$}} is still valid as long as the condition \eqref{eq:067} holds, while the uniqueness fails since there may exist more than one solutions.
\end{rem}

With the obtained initial approximation $ (\dot{U}_{-}(\xi,\eta),\ \dot{U}_{+}(\xi,\eta),\ \dot{\xi}_*;\ \dot{\psi}'(\eta)) $, we are able to carry out a nonlinear iteration to obtain the following theorem showing the existence of solutions to the free boundary problem {\bf{$\textit{FBPL}$}}.

\begin{thm}\label{thm26}
	
	Let $\alpha \in (\frac12,1)$. Suppose the assumptions in Theorem \ref{thm:ExistenceNozzleShocks} hold. Then there exists a sufficiently small positive constant $\sigma_0$, only depends on $\bar{U}_{\pm}$, $L$, $\dot{\xi}_*$, as well as $\displaystyle\frac{1}{|\Theta(\dot{\xi}_*)|}$, such that for any $0< \sigma \leq \sigma_0$, there exists a solution $(U_-, U_+ ; \psi)$ for the free boundary problem {\bf{$\textit{FBPL}$}}, and the solution satisfies the following estimates:
	\begin{align}
	&|\psi(1) - \dot{\xi}_*|\leq C_{\mr{s}}\sigma,\quad \| {\psi}'\|_{1,\alpha;{\Gamma}_{\mr{s}}}^{(-\alpha;O_4)}\leq C_{\mr{s}} \sigma,\\
	&\| U_- - \bar{U}_- \|_{\mcc^{2,\alpha}(\Omega_-)}  \leq C_-\sigma,\\
	&\| {U}_+ - \bar{U}_+ \|_{({\Omega}_+;{\Gamma}_{\mr{s}})} \leq C_+ \sigma,\\
	&\| U_- - (\bar{U}_- + \dot{U}_-)\|_{\mcc^{1,\alpha}(\Omega_-)}  \leq \frac12\sigma^{\frac32},\\
	&\| U_+\circ \mathcal{T}^{-1}- (\bar{U}_+ + \dot{U}_+)\|_{(\dot{\Omega}_+;\dot{\Gamma}_{\mr{s}})}  \leq \frac12\sigma^{\frac32},\\
&\|{\psi}' - \dot{\psi}'\|_{1,\alpha;\dot{{\Gamma}}_{\mr{s}}}^{(-\alpha;Q_4)}\leq \frac12\sigma^{\frac32},
	\end{align}
	where the constants $C_{\mr{s}}$ and $C_\pm$ depend on $\bar{U}_{\pm}$, $L$, $\dot{\xi}_*$, $\displaystyle\frac{1}{|\Theta(\dot{\xi}_*)|}$ and $\alpha$.

\end{thm}

\section{The Elliptic Sub-problem in the Linearized Problem.}

In solving the free boundary problem {\bf{$\textit{IFBPL}$}}, as well as the linearized problem for the nonlinear iteration, one of the key step is to solve the elliptic sub-problem for $ (\dot{\theta}_+, \dot{p}_+) $. In this section, we extract this elliptic sub-problem and establish a well-posedness theorem for it. Note that the notations used in this section are independent and have no relations to the ones in other parts of the paper.


\begin{figure}[!h]
	\centering
	\includegraphics[width=0.4\textwidth]{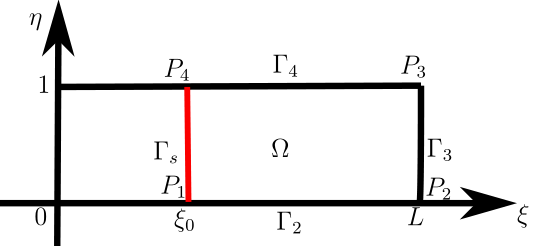}
	\caption{The domain for the boundary value problem in the $(\xi,\eta)$ coordinate.\label{fig:5}}
\end{figure}

Let $\xi_0$ and $L$ be two positive constants, and
\begin{equation}
  {\Omega} = \{ (\xi, \eta)\in \mathbb{R}^2 : \xi_0 < \xi < L, \,\,\,0 < \eta< 1\},
\end{equation}
be a rectangle with the boundaries (see Figure \ref{fig:5})
\begin{align*}
&\Gamma_{\mr{s}} = \{ (\xi,\eta)\in \mathbb{R}^2 : \xi =\xi_0,\,\,\, 0 <\eta <1 \},\\
&\Gamma_2 = \{ (\xi,\eta)\in \mathbb{R}^2 :  \xi_0 <{\xi} < L, \,\,\,\eta =0 \},\\
&\Gamma_3 = \{ (\xi,\eta)\in \mathbb{R}^2 : \xi = L,\,\,\, 0< \eta <1 \},\\
&\Gamma_4 = \{ (\xi,\eta)\in \mathbb{R}^2 :  \xi_0 < {\xi} < L , \,\,\,\eta =1 \}.
\end{align*}

Consider the following boundary value problem for unknowns $(H_1, H_2)$:
\begin{align}
&\partial_\eta  H_1 + \mathcal{A}\partial_\xi  H_2  = \pounds_1, &\quad &\text{in}\quad \Omega\label{eq078}\\
&\Big(\partial_\eta H_2 + \displaystyle\frac{1}{\eta}H_2\Big) - \mathcal{B}\partial_\xi H_1   = \pounds_2, &\quad &\text{in}\quad \Omega \\
&H_1 = \hbar_1, &\quad &\text{on} \quad {\Gamma}_{\mr{s}}\\
&H_2 = 0, &\quad &\text{on}\quad\Gamma_{2}\\
&H_1 = \hbar_3, &\quad &\text{on}\quad\Gamma_3\\
&H_2 = \hbar_4, &\quad &\text{on} \quad \Gamma_{4}\label{eq0078}
\end{align}
where $\mathcal{A}$ and $\mathcal{B}$ are two constants satisfying $\mathcal{A}\mathcal{B}>0$.

\begin{thm}\label{sub}
  Let $\alpha\in(\frac12,1)$. Suppose $\pounds_i\in \mathbf{H}_{0,\alpha}^{(1-\alpha;\{P_3,P_4\} )}({\Omega}) $, $i=1,2$, $\pounds_1(\xi,0) = 0$, $\hbar_1 \in \mathbf{H}_{1,\alpha}^{(- \alpha ;P_4 )}(\Gamma_{\mr{s}})$, $\hbar_{3}\in \mathbf{H}_{1,\alpha}^{(- \alpha ; P_3 )}(\Gamma_3)$, and  $\hbar_{4}\in \mathbf{H}_{1,\alpha}^{(- \alpha ;\{P_3, P_4\} )}(\Gamma_4)$, then for the boundary value problem $\eqref{eq078}$-$\eqref{eq0078}$, there exists a unique solution $(H_1, H_2)$ if and only if
\begin{equation}\label{eq082}
	\int_{\Omega} \eta \pounds_2 \dif \xi \dif \eta = \mathcal{B}\int_{0}^{1} \eta (\hbar_1 - \hbar_3)\dif \eta +\int_{{\xi}_0}^L \hbar_4 \dif \xi.
\end{equation}
Moreover, $(H_1, H_2)$ satisfies the following estimate:
\begin{align}\label{eq083}
&\sum_{i=1}^2 \| H_i \|_{1,\alpha;{\Omega}}^{(-\alpha;\{P_3,P_4\})} \notag\\
\leq& C\left(  \sum_{i=1}^2 \| \pounds_i \|_{0,\alpha;{\Omega}}^{(1-\alpha;\{P_3,P_4\})} + \| \hbar_1 \|_{1,\alpha;\Gamma_{\mr{s}}}^{(-\alpha;P_4)} +  \| \hbar_3 \|_{1,\alpha;\Gamma_3}^{(-\alpha;P_3)} + \| \hbar_4 \|_{1,\alpha;\Gamma_4}^{(-\alpha;\{P_3, P_4\})} \right),
\end{align}
where the constant $C$ depends on $\mathcal{A}$, $\mathcal{B}$, $\xi_0$, $L$ and $\alpha$.
\end{thm}

\begin{proof}
	The proof is divided into four steps.

\textbf{Step 1}: In this step, the problem \eqref{eq078}-\eqref{eq0078} will be reduced to a typical form and in order to solve it, it is further decomposed into two auxiliary problems.

 Let
\begin{equation}\label{e}
  (x,y) = \Big(\sqrt{\frac{1}{\mathcal{A}\mathcal{B}}}\xi, \eta\Big),\quad (V_1, V_2)=\Big(\sqrt{\displaystyle\frac{\mathcal{B}}{\mathcal{A}}}H_1, H_2\Big).
\end{equation}
Denote $x_0 := \sqrt{\frac{1}{\mathcal{A}\mathcal{B}}}\xi_0$, $\tilde{L} :=\sqrt{\frac{1}{\mathcal{A}\mathcal{B}}} L$.

\begin{figure}[!h]
	\centering
	\includegraphics[width=0.4\textwidth]{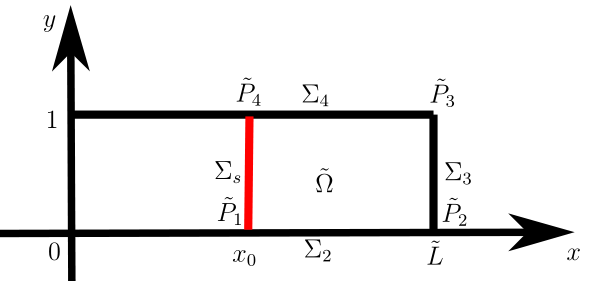}
	\caption{The domain for the boundary value problem in the $(x,y)$ coordinate.\label{fig:55}}
\end{figure}

Then the problem $\eqref{eq078}$-$\eqref{eq0078}$ becomes
\begin{align}
&\partial_y V_1 + \partial_{x} V_2 = \sqrt{\displaystyle\frac{\mathcal{B}}{\mathcal{A}}} {\pounds}_1: = F_1 , &\quad &\text{in}\quad \tilde{\Omega}\label{eq901}\\
&  \partial_{y}(y V_2)- \partial_x(y V_1)  =y{\pounds}_2:=yF_2, &\quad &\text{in}\quad \tilde{\Omega}\\
&V_1 = \sqrt{\displaystyle\frac{\mathcal{B}}{\mathcal{A}}}{\hbar}_1, &\quad &\text{on} \quad{{\Sigma}}_{\mr{s}}\\
&V_2 = 0 , &\quad &\text{on}\quad{\Sigma}_{2}\\
&V_1 = \sqrt{\displaystyle\frac{\mathcal{B}}{\mathcal{A}}}{\hbar}_3 , &\quad &\text{on}\quad\Sigma_3\\
&V_2 = {\hbar}_4, &\quad &\text{on} \quad \Sigma_{4}\label{eq0901}
\end{align}
where $\tilde{\Omega}$ is a rectangle with the boundaries ${{\Sigma}}_{\mr{s}}$ and $\Sigma_i$, $(i=2,3,4)$,(see Figure \ref{fig:55}).

Then we decompose the problem $\eqref{eq901}$-$\eqref{eq0901}$ into two boundary value problems with different inhomogeneous terms as follows.

 Let $(V_1, V_2)^T = (\mathcal{U}_1, \mathcal{U}_2)^T + (\mathcal{W}_1, \mathcal{W}_2)^T$, where $(\mathcal{U}_1, \mathcal{U}_2)^T$ is the solution to the problem
\begin{align}
 &\partial_{y}(y \mathcal{U}_2) - \partial_x(y \mathcal{U}_1) =0, &\quad &\text{in}\quad \tilde{\Omega}\label{A}\\
&\partial_y \mathcal{U}_1 + \partial_{x} \mathcal{U}_2 = F_1,&\quad &\text{in}\quad\tilde{\Omega}\label{xin}\\
&\mathcal{U}_1 = 0, &\quad &\text{on} \quad{{\Sigma}}_{\mr{s}}\cup{\Sigma}_3\\
&\mathcal{U}_2 = 0 ,&\quad &\text{on}\quad{\Sigma}_{2}\cup{\Sigma}_4\label{A1}
\end{align}
 and $ (\mathcal{W}_1, \mathcal{W}_2)^T$ satisfies the following problem
\begin{align}
&\partial_{y}(y \mathcal{W}_2) - \partial_x(y \mathcal{W}_1)  =y F_2, &\quad &\text{in}\quad \tilde{\Omega}\label{B}\\
&\partial_y \mathcal{W}_1 + \partial_{x} \mathcal{W}_2 = 0, &\quad &\text{in}\quad\tilde{\Omega}\label{Ww}\\
&\mathcal{W}_1 = \sqrt{\displaystyle\frac{\mathcal{B}}{\mathcal{A}}}{\hbar}_1,& \quad &\text{on} \quad{{\Sigma}}_{\mr{s}}\\
&\mathcal{W}_2 = 0 , &\quad &\text{on}\quad{\Sigma}_{2}\\
&\mathcal{W}_1 = \sqrt{\displaystyle\frac{\mathcal{B}}{\mathcal{A}}}{\hbar}_3 , &\quad &\text{on}\quad{\Sigma}_3\\
&\mathcal{W}_2 = {\hbar}_4. &\quad &\text{on} \quad {\Sigma}_{4}\label{B1}
\end{align}

\textbf{Step 2:} In this step, the problem $\eqref{A}$-$\eqref{A1}$ will be solved. Note the equation \eqref{A} implies that there exists a potential function $\phi$ such that
\begin{equation}\label{n}
  ( \partial_x {\phi},  \partial_y {\phi}) = (y\mathcal{U}_2, y\mathcal{U}_1).
\end{equation}
Let $\Phi  = \displaystyle\frac{\phi}{y}$, then \eqref{n} yields that
\begin{align}\label{nn1}
  (\mathcal{U}_1, \mathcal{U}_2) =  \Big( \partial_y {\Phi} + \frac{\Phi}{y}, \partial_x {\Phi}\Big).
\end{align}
Then the problem $\eqref{A}$-$\eqref{A1}$ can be rewritten as the following form:
\begin{align}
&\partial_{xx}\Phi + \partial_{yy}\Phi + \displaystyle\frac{\partial_y {\Phi}}{y} -  \displaystyle\frac{\Phi}{y^2}=  F_1, &\quad &\text{in}\quad \tilde{\Omega} \label{AA}\\
&\partial_y {\Phi} + \frac{\Phi}{y} = 0,&\quad &\text{on} \quad {\Sigma}_{\mr{s}} \cup \Sigma_3\label{AA1}\\
&\partial_x {\Phi}= 0.&\quad  &\text{on} \quad \Sigma_{2}\cup\Sigma_{4}\label{AA3}
\end{align}
Without loss of generality, we may assume $\Phi(x_0, 0) = 0$, then it is easy to see that $\Phi = 0$ on the boundary of $ \tilde{\Omega}$ since the boundary condition \eqref{AA1} can be rewritten as $\partial_y(y\Phi) = 0$, which is exactly a tangential derivative.

 It is obvious that, the coefficients of equation $\eqref{AA}$ tend to infinity as $y$ goes to zero, therefore the traditional estimates for the elliptic equations are not valid near $\Sigma_2$. By applying the methods of \cite{BW2018,5d}, the problem \eqref{AA}-\eqref{AA3} can be transformed into a 5-D Laplace equation with Dirichlet boundary conditions. Then the well-posedness theory can be established as below.

	Let ${\Phi} = y \tilde{\Phi}$, the equation $\eqref{AA}$ can be reduced into
	\begin{equation}\label{eq5281}
	\partial_{xx} \tilde{\Phi} + \partial_{yy} \tilde{\Phi} + \displaystyle\frac{3}{y} \partial_y \tilde{\Phi} = \frac{F_1}{y} : = \tilde{F}_1.
	\end{equation}
Define
\begin{equation}
	{F_1^*}(x,y) \defs \frac{1}{{y}^4} \int_{0}^{{y} } \tau^3 \tilde{F}_1(x,\tau)\dif \tau.
	\end{equation}
Then, let
	\begin{equation}
	\mathbf{{F}}_1(x, \mathbf{y}) = \Big(0, {F_1^*}(x,y)y_1, {F_1^*}(x,y)y_2,{F_1^*}(x,y)y_3,{F_1^*}(x,y)y_4\Big),\quad \text{in} \quad \mathcal{D },
	\end{equation}
where
	\begin{equation}
	\mathcal{D }: = \{(x,\mathbf{y}): x\in (x_0,\tilde{L}), \,\mathbf{y} \in \mathbb{R}^4,\, |\mathbf{y}| < 1\},
	\end{equation}
	with $\mathbf{y} = (y_1, y_2, y_3, y_4)$ and $\sum\limits_{i=1}^{4} y_i^2 = y^2$. Denote $\mathbf{Y}\defs (x, \mathbf{y}) \in \mathcal{D}$, then one has
	\begin{equation}
	\tilde{F}_1(x, y ) = {\rm{div}}_{\mathbf{Y}}  {\mathbf{{F}}}_1(x, \mathbf{y}).
	\end{equation}
Therefore, it follows from \eqref{eq5281} that
	\begin{equation}\label{eq528}
	\triangle_{\mathbf{Y}}\tilde{\Phi} = {\rm{div}}_{\mathbf{Y}} \mathbf{{F}}_1, \,\,\, \text{in}\,\,\, \mathcal{D},
	\end{equation}
	with the boundary condition
\begin{align}\label{eq528q}
  \tilde{\Phi} = 0, \quad \text{on} \quad \partial\mathcal{D}.
\end{align}
By Lax-Milgram theorem and Fredholm alternative theorem (cf. \cite{Evans}), there exists a unique solution $\tilde{\Phi}\in H_0^1 (\mathcal{D})$ to the problem \eqref{eq528}-\eqref{eq528q} satisfying
%
\begin{equation}\label{PART1}
	\| \tilde{\Phi} \|_{H_0^1(\mathcal{D})}\leq C \| {F}_1\|_{0,\alpha;\tilde{\Omega}}^{(1 - \alpha;\{\tilde{P}_3,\tilde{P}_4\})}.
	\end{equation}
To obtain \eqref{PART1}, first, by applying $\pounds_1(\xi,0) = 0$ and the definition of $F_1$, it holds that $F_1(x,0) = 0$, then, for all $y\in (0,1)$, one has
	\begin{equation}
	\tilde{F}_1(x,y) = \frac{F_1(x,y) - F_1(x,0)}{y} = \frac{F_1(x,y) - F_1(x,0)}{y^\alpha}y^{\alpha-1}.
	\end{equation}
  Since $F_1\in\mathbf{H}_{0,\alpha}^{(1-\alpha;\{\tilde{P}_3,\tilde{P}_4\})}(\tilde{\Omega})$, it is easy to check that
	\begin{equation}\label{k}
	\| \mathbf{{F}}_1\|_{0,\alpha;\mathcal{D}}^{(1-\alpha;\partial\mathcal{D})}\leq C \| {F}_1\|_{0,\alpha;\tilde{\Omega}}^{(1 - \alpha; \{\tilde{P}_3,\tilde{P}_4\})}.
	\end{equation}

Multiplying $\tilde{\Phi}$ on the both sides of the equation \eqref{eq528}, integrating over $\mathcal{D}$ and then employing the formula of integration by parts, one obtains
 \begin{align}\label{h}
  \int_{\mathcal{D}}|\nabla\tilde{\Phi}|^2 \dif\mathbf{Y} =  \int_{\mathcal{D}}\mathbf{{F}}_1\cdot \nabla\tilde{\Phi}  \dif\mathbf{Y}.
\end{align}
 Applying \eqref{k}, one has $\mathbf{{F}}_1\in \mathbf{H}_{0,\alpha}^{(1 - \alpha;\partial \mathcal{D})}(\mathcal{D})$. According to the definition of the weighted H\"{o}lder norms, it follows that
\begin{equation}\label{xgg}
  |\mathbf{{F}}_1(\mathbf{Y})|\leq  d_{\mathbf{Y}}^{\alpha -1} \| \mathbf{{F}}_1\|_{0,\alpha;\mathcal{D}}^{(1-\alpha;\partial\mathcal{D})},
\end{equation}
where $d_{\mathbf{Y}}: = \dist(\mathbf{Y}, \partial \mathcal{D})$ with $\mathbf{Y}\in \mathcal{D}$, which yields that
\begin{equation}\label{hh}
  \int_{\mathcal{D}}|\mathbf{{F}}_1|^2 \dif\mathbf{Y}\leq
  C \Big(\| \mathbf{{F}}_1\|_{0,\alpha;\mathcal{D}}^{(1-\alpha;\partial\mathcal{D})}
  \Big)^2,\quad \text{for}\quad \alpha \in (\frac12, 1).
\end{equation}
Furthermore, by Cauchy's inequality with $\epsilon$, \eqref{h} implies that
\begin{equation}\label{hh1}
  \| \nabla\tilde{\Phi}\|_{L^2(\mathcal{D})}^2 \leq  \epsilon  \| \nabla\tilde{\Phi}\|_{L^2(\mathcal{D})}^2 + \frac{1}{4\epsilon} \| \mathbf{{F}}_1 \|_{L^2(\mathcal{D})}^2,
\end{equation}
where $\epsilon> 0$ is a sufficiently small constant. Then, applying poincar\'{e} inequality and the inequalities \eqref{k}, \eqref{hh}, one can obtain \eqref{PART1}.

Then we raise the regularity of $\tilde{\Phi}$.

 By employing $\mathbf{{F}}_1\in \mathbf{H}_{0,\alpha}^{(1 - \alpha;\partial \mathcal{D})}(\mathcal{D})$ and Theorem 5.19 of \cite{GM}, it holds that $D\tilde{\Phi}\in \mcc^{0,\alpha}$ for the point $\mathbf{Y}$ away from the boundary of domain $\mathcal{D}$. Therefore, it suffices to estimate $\tilde{\Phi}$ in the case that $\mathbf{Y}$ near the boundary of the domain $\mathcal{D}$. Similar as Theorem 5.21 of \cite{GM}, for any fixed point $\mathbf{Y}_0\in \bar{\mathcal{D}}$ and a constant $R\in\mathbb{R}$ with $0<R<\frac{1}{2}$, let
\begin{align}
  B_R(\mathbf{Y}_0): = \{\mathbf{Y} \in \mathbb{R}^5: |\mathbf{Y} - \mathbf{Y}_0|< R\}, \quad B_R^+\defs B_R(\mathbf{Y}_0)\cap \mathcal{D}.
\end{align}
 Suppose that $\Phi_*\in H^1(B_R^+)$ is a weak solution of
\begin{align}
  &\triangle_{\mathbf{Y}}{\Phi}_* = 0, &\quad &\text{in}\quad {B_R^+},\\
  & \Phi_* = \tilde{\Phi},&\quad &\text{on}\quad \partial B_R^+.
\end{align}
Obviously, the function $\Phi^* : = \tilde{\Phi} - \Phi_*$ satisfies the following relation
\begin{equation}
  \int_{B_R^+}\partial_i \Phi^* \partial_i \zeta \dif\mathbf{Y}=
  \int_{B_R^+}\mathbf{{F}}_1\cdot  \nabla \zeta  \dif\mathbf{Y}, \quad \text{for} \quad \zeta\in H_0^1(B_R^+).
\end{equation}
Taking the test function $\zeta = \Phi^*$, and using H\"{o}lder inequality, it holds that
\begin{align}
  \int_{B_R^+}|\nabla{\Phi}^*|^2 \dif \mathbf{Y} \leq
  \Big(\int_{B_R^+}|\mathbf{{F}}_1|^2 \dif \mathbf{Y}\Big)^{\frac12} \cdot \Big(\int_{B_R^+}|\nabla{\Phi}^*|^2 \dif \mathbf{Y}\Big)^{\frac12},
\end{align}
which yields that
\begin{align}\label{hhh}
  \int_{B_R^+}|\nabla{\Phi}^*|^2 \dif \mathbf{Y}\leq  \int_{B_R^+}|\mathbf{{F}}_1 |^2 \dif \mathbf{Y}.
\end{align}
 Applying $\mathbf{{F}}_1\in \mathbf{H}_{0,\alpha}^{(1 - \alpha;\partial \mathcal{D})}(\mathcal{D})$, one obtains
\begin{align}
 \int_{B_R^+}|\nabla{\Phi}^*|^2 \dif\mathbf{Y} \leq
 \Big(\|\mathbf{F}\|_{0,\alpha;\mathcal{D}}^{(1 - \alpha;\partial \mathcal{D})}\Big)^2\int_{B_R^+}d_{\mathbf{Y}}^{2(\alpha -1)}\dif \mathbf{Y}\leq C \Big(\|\mathbf{F}\|_{0,\alpha;\mathcal{D}}^{(1 - \alpha;\partial \mathcal{D})}\Big)^2 R^{3+2\alpha}.
\end{align}
For any $0< r\leq R$, let $B_r^+ \defs B_r(\mathbf{Y}_0)\cap\mathcal{D}$, similar as step 3 in Theorem 5.21 of \cite{GM}, one has
\begin{equation}
\begin{aligned}
   \int_{B_r^+} |\nabla\tilde{\Phi} |^2 \dif \mathbf{Y}
  &\leq C\left(\Big(\frac{r}{R}\Big)^5 \int_{B_R^+} |\nabla\tilde{\Phi}|^2 \dif \mathbf{Y} +\int_{B_R^+}|\nabla{\Phi}^*|^2 \dif\mathbf{Y}\right)\\
  &\leq C\left(\Big(\frac{r}{R}\Big)^5 \int_{B_R^+} |\nabla\tilde{\Phi}|^2
    \dif\mathbf{Y} + R^{3+ 2\alpha}\Big(\| \mathbf{{F}}_1\|_{0,\alpha;{\mathcal{D}}}^{(1-\alpha;\partial{\mathcal{D}})}
  \Big)^2\right).
  \end{aligned}
  \end{equation}
Then by Lemma 5.13 of \cite{GM}, it follows that
\begin{equation}
  \int_{B_r^+} |\nabla\tilde{\Phi} |^2 \dif \mathbf{Y}\leq C\left(\frac{1}{R^{3+2\alpha}} \int_{B_R^+} |\nabla\tilde{\Phi}|^2
    \dif\mathbf{Y} +\Big( \| \mathbf{{F}}_1\|_{0,\alpha;{\mathcal{D}}}^{(1-\alpha;\partial{\mathcal{D}})} \Big)^2\right) r^{3+ 2\alpha}.
\end{equation}
  Applying Poincar\'{e} inequality (see Proposition 3.12 of \cite{GM}), one has
\begin{align}\label{i}
  \int_{B_r^+} |\tilde{\Phi} - (\tilde{\Phi})_{\mathbf{Y_0}, r}|^2 \dif \mathbf{Y}\leq C \left(\frac{1}{R^{2(3+2\alpha)}}\| \tilde{\Phi}\|_{H^1(B_R^+)} + \| \mathbf{{F}}_1\|_{0,\alpha;{\mathcal{D}}}^{(1-\alpha;\partial{\mathcal{D}})} \right)^2 r^{5+ 2\alpha},
\end{align}
where $(\tilde{\Phi})_{\mathbf{Y}_0, r} = \frac{1}{|B_r^+|}\int_{B_r^+}\tilde{\Phi} \dif \mathbf{Y}$.
Furthermore, applying Theorem 3.1 of \cite{Han}, together with \eqref{i} and Theorem 5.19 of \cite{GM}, one has
\begin{equation}
  \|\tilde{\Phi} \|_{0,\alpha;\mathcal{D}}\leq C \| {F}_1\|_{0,\alpha;\tilde{\Omega}}^{(1 - \alpha; \{\tilde{P}_3,\tilde{P}_4\})}.
\end{equation}
Therefore, by standard elliptic theory (cf.\cite{GI2011,LIE2013}), it follows that
\begin{equation}\label{ic}
	\|\tilde{\Phi} \|_{1,\alpha;\mathcal{D}}^{( - \alpha; \partial \mathcal{D})}\leq C \| {F}_1\|_{0,\alpha;\tilde{\Omega}}^{(1 - \alpha; \{\tilde{P}_3,\tilde{P}_4\})}.
	\end{equation}
By the rotational invariance of the boundary value problem \eqref{eq528}-\eqref{eq528q} and the uniqueness of the solution $\tilde{\Phi}$, the solution itself is rotationally invariant, i.e., it depends only on the variables $x$ and $y $.
Thus, \eqref{ic} implies that
\begin{equation}\label{i1}
	\| \tilde{\Phi} \|_{1,\alpha;\tilde{\Omega}}^{( - \alpha; \{\tilde{P}_3,\tilde{P}_4\})}\leq C \|{F}_1\|_{0,\alpha;\tilde{\Omega}}^{(1 - \alpha; \{\tilde{P}_3,\tilde{P}_4\})}.
	\end{equation}
By \eqref{AA}, which yields that
	\begin{equation}\label{eq30000}
	\partial_{xx} \Phi + \partial_{yy} \Phi  = F_1 - \frac{\partial_y \Phi}{y} + \frac{\Phi}{y^2}= F_1 - \partial_y \tilde{\Phi} \in \mathbf{H}_{0,\alpha}^{(1-\alpha;\{\tilde{P}_3,\tilde{P}_4\})}(\tilde{\Omega}),
	\end{equation}
furthermore, applying Theorem 4.6 of \cite{LIE2013}, one has
	\begin{equation}\label{c1}
	\| {\Phi} \|_{2,\alpha;\tilde{\Omega}}^{(-1-\alpha;\{\tilde{P}_3,\tilde{P}_4\})}
	\leq C \| {F}_1\|_{0,\alpha;\tilde{\Omega}}^{(1 - \alpha;\{\tilde{P}_3,\tilde{P}_4\})}.
	\end{equation}
Therefore, according to the definition of $\Phi$ in \eqref{nn1}, one has
\begin{equation}\label{c20}
	\sum_{i=1}^2\| \mathcal{U}_i \|_{1,\alpha;\tilde{\Omega}}^{(-\alpha;\{\tilde{P}_3,\tilde{P}_4\})}
	\leq C \| {F}_1\|_{0,\alpha;\tilde{\Omega}}^{(1 - \alpha;\{\tilde{P}_3,\tilde{P}_4\})}.
	\end{equation}

	\textbf{Step 3:} In this step, the boundary value problem $\eqref{B}$-$\eqref{B1}$ will be solved. By \eqref{Ww}, there exists a potential function $\Psi$ such that
\begin{equation}\label{x50}
  \nabla {\Psi} = (\partial_x {\Psi}, \partial_y {\Psi}) = (- \mathcal{W}_1,\mathcal{W}_2).
\end{equation}
Then the original problem $\eqref{B}$-$\eqref{B1}$ can be reformulated as the following problem:
\begin{align}
&\partial_x (y \partial_x {\Psi}) + \partial_y(y \partial_y {\Psi}) =  y F_2,&\quad &\text{in}\quad \tilde{\Omega}\label{B2}\\
&-\partial_x {\Psi} = \sqrt{\displaystyle\frac{\mathcal{B}}{\mathcal{A}}}\hbar_1,& \quad &\text{on} \quad {\Sigma}_{\mr{s}}\\
&\partial_y {\Psi} = 0 , &\quad &\text{on} \quad \Sigma_{2}\\
&-\partial_x {\Psi} =  \sqrt{\displaystyle\frac{\mathcal{B}}{\mathcal{A}}}\hbar_3,&\quad &\text{on} \quad \Sigma_3\\
&\partial_y {\Psi} = \hbar_4.&\quad &\text{on} \quad  \Sigma_{4}
\end{align}
 To deal with the singularity near $\Sigma_2$, we define
	\begin{equation}
	\mathbf{x}:= (x_1, x_2, x_3) = (x,y \cos \omega,y \sin \omega),	\end{equation}
	\begin{equation}
	\mathcal {P}: = \{(x_1,x_2,x_3)\in {\mathbb{R}}^3 : x_0 < x_1 < \tilde{L}, 0 \leq x_2^2 + x_3^2 < 1\}.
	\end{equation}
	Then the equation $\eqref{B2}$ becomes
	\begin{equation}\label{t}
	\Delta\Psi= F_2,\quad \text{in}\quad  \mathcal{P},
	\end{equation}
	with the following boundary conditions
	\begin{align}
	&\partial_{x_1} {\Psi}(x_0,x_2,x_3) = -\sqrt{\displaystyle\frac{\mathcal{B}}{\mathcal{A}}}\hbar_1({x}_0, y), &\quad &\text{on} \quad\mathcal{P}_{x_0}\label{eq12000}\\
	&x_2\partial_{x_2} {\Psi}(x_1, x_2, x_3) + x_3\partial_{x_3} \Psi(x_1, x_2, x_3)= y \hbar_4(x_1),&\quad &\text{on} \quad \mathcal{P}_w\\
	&\partial_{x_1}{\Psi}(\tilde{L},x_2, x_3) =  -\sqrt{\displaystyle\frac{\mathcal{B}}{\mathcal{A}}}\hbar_3(\tilde{L},y),&\quad &\text{on} \quad \mathcal{P}_{\tilde{L}}\label{Tt}
	\end{align}
	where $\mathcal{P}_{{x}_0}$, $\mathcal{P}_{\tilde{L}} $ and $\mathcal{P}_w$ represent the entrance, exit and the surface of the pipe $\mathcal{P}$ respectively.

First, by Lax-Milgram theorem, the problem \eqref{t}-\eqref{Tt} has a weak solution ${\Psi}\in H^1 (\mathcal{P})$, if and only if
	\begin{align}\label{eq088}
	\int_{\mathcal{P}}  F_2 \dif \mathbf{x} &=
	\int_{\partial\mathcal{P}} \nabla\Psi\cdot \vec{n} \dif S\notag\\
	& = 2 \pi\left(\sqrt{\frac{\mathcal{B}}{\mathcal{A}}}\int_{0}^{1} y(\hbar_1 - \hbar_3)\dif y  + \int_{x_0}^{\tilde{L}} \hbar_4 \dif x_1\right),
	\end{align}
 where $\vec{n}$ is the unit outer normal vector.

Moreover, assume that $\Psi_1$ and $\Psi_2$ are two solutions to the system \eqref{t}-\eqref{Tt}. Let $\tilde{\Psi}: = \Psi_2- \Psi_1$, then it is easy to see that $\int_{\mathcal{P}} |\nabla\tilde{\Psi}|^2 \dif \mathbf{x}=0$, which yields that $\tilde{\Psi}\equiv constant$. That is, this solution is unique up to an additive constant.

Then we follow the same procedure as the step 2, multiplying $\Psi$ on the both sides of the
equation \eqref{t}, integrating over $\mathcal{P}$ and employing the formula of integration by
parts, for $\alpha\in(\frac12, 1)$, one has
\begin{align}\label{t2}
  \| \nabla\Psi\|_{L^2(\mathcal{P})}^2 \leq&  C\| {F}_2\|_{0,\alpha;\mathcal{P}}^{(1 - \alpha;\partial\mathcal{P} )} \| \Psi\|_{L^2(\mathcal{P})} \\
  &+ C_{(\alpha,\mathcal{P})}\max \Big\{\| \hbar_1 \|_{L^\infty(\mathcal{P}_{{x}_0})},\| \hbar_4\|_{L^\infty(\mathcal{P}_w)},\| \hbar_3\|_{L^\infty(\mathcal{P}_{\tilde{L}})}\Big\} \| \Psi\|_{L^2(\partial\mathcal{P})}\notag.
\end{align}
Without loss of generality, we may assume $\int_{\mathcal{P}} \Psi\dif \mathbf{x} = 0$. Therefore, by Poincar\'{e} inequality, there exists a
constant $C_{(\mathcal{P})}$ such that
 \begin{equation}\label{t1}
   \|\Psi\|_{L^2(\mathcal{P})} \leq C_{(\mathcal{P})} \|\nabla \Psi \|_{L^2(\mathcal{P})}.
 \end{equation}
By applying \eqref{t1} and Trace theorem, then \eqref{t2} implies that
\begin{equation}
	\| \Psi \|_{H^1(\mathcal{P})} \leq C\left( \| F_2 \|_{0,\alpha;\mathcal{P}}^{(1-\alpha;\partial\mathcal{P})} + \| \hbar_1 \|_{1,\alpha;\mathcal{P}_{{x}_0}}^{(-\alpha;\partial\mathcal{P})} +  \| \hbar_3\|_{1,\alpha; \mathcal{P}_{\tilde{L}} }^{(-\alpha;\partial\mathcal{P})} +  \| \hbar_4\|_{1,\alpha;\mathcal{P}_w }^{(-\alpha;\partial\mathcal{P})}\right).
	\end{equation}
Furthermore, by an analogous argument as in step 2, we can obtain the $\mcc^{0,\alpha}$ estimate for $\Psi$. Then, one obtains
\begin{align}\label{eq084}
	&\| \Psi \|_{2,\alpha;\mathcal{P}}^{(-1-\alpha;\partial\mathcal{P})}\notag\\
 \leq& C\left( \| F_2 \|_{0,\alpha;\mathcal{P}}^{(1-\alpha;\partial\mathcal{P})} + \| \hbar_1 \|_{1,\alpha;\mathcal{P}_{{x}_0}}^{(-\alpha;\partial\mathcal{P})} +  \| \hbar_3\|_{1,\alpha; \mathcal{P}_{\tilde{L}} }^{(-\alpha;\partial\mathcal{P})} +  \| \hbar_4\|_{1,\alpha;\mathcal{P}_w }^{(-\alpha;\partial\mathcal{P})}\right).
	\end{align}
By applying the rotational invariance of the boundary value problem and the uniqueness of the solution ${\Psi}$ and \eqref{eq088}, in the $(x,y)$ coordinate, there exists a unique solution $\Psi(x,y)$ if and only if
\begin{align}\label{GGg}
  \int_{\tilde{\Omega}}y F_2 \dif x \dif y =  \sqrt{\frac{\mathcal{B}}{\mathcal{A}}}\int_{0}^{1} y(\hbar_1 - \hbar_3)\dif y  + \int_{x_0}^{\tilde{L}} \hbar_4 \dif x.
\end{align}
Moreover, $\Psi(x,y)$ satisfies
\begin{equation}\label{c2}
\begin{aligned}
	\| \Psi \|_{2,\alpha;\tilde{\Omega}}^{(-1-\alpha;\{\tilde{P}_3,\tilde{P}_4\})}
 \leq& C\left( \| F_2 \|_{0,\alpha;\tilde{\Omega}}^{(1-\alpha;\{\tilde{P}_3,\tilde{P}_4\})}+ \|\hbar_1 \|_{1,\alpha;\Sigma_{\mr{s}}}^{(-\alpha; \tilde{P}_4)}\right)\\
 &+  C\left(\|\hbar_3 \|_{1,\alpha;\Sigma_3}^{(-\alpha;\tilde{P}_3)} +  \|\hbar_4 \|_{1,\alpha;\Sigma_4}^{(-\alpha;\{\tilde{P}_3,\tilde{P}_4\})}\right).
	\end{aligned}
\end{equation}
According to the definition of $\Psi$, it holds that 	
\begin{equation}\label{c21}
\begin{aligned}
	\sum_{i=1}^2\| \mathcal{W}_i \|_{1,\alpha;\tilde{\Omega}}^{(-\alpha;\{\tilde{P}_3,\tilde{P}_4\})}
 \leq& C\left( \| F_2 \|_{0,\alpha;\tilde{\Omega}}^{(1-\alpha;\{\tilde{P}_3,\tilde{P}_4\})}+ \|\hbar_1 \|_{1,\alpha;\Sigma_{\mr{s}}}^{(-\alpha; \tilde{P}_4)}\right)\\
 &+ C\left( \|\hbar_3 \|_{1,\alpha;\Sigma_3}^{(-\alpha;\tilde{P}_3)} +  \|\hbar_4 \|_{1,\alpha;\Sigma_4}^{(-\alpha;\{\tilde{P}_3,\tilde{P}_4\})}\right).
	\end{aligned}
\end{equation}

\textbf{Step 4:}
By recalling that the transformation $\eqref{e}$ and applying \eqref{GGg},
then there exists a unique solution $(H_1, H_2)$ to the boundary value problem $\eqref{eq078}$-$\eqref{eq0078}$ if and only if
\begin{equation}
	\int_{\Omega} \eta \pounds_2 \dif \xi \dif \eta = \mathcal{B}\int_{0}^{1} \eta (\hbar_1 - \hbar_3)\dif \eta +\int_{{\xi}_0}^L \hbar_4 \dif \xi,
	\end{equation}
which is exactly the solvability condition $\eqref{eq082}$.
Moreover, $(H_1, H_2)$ satisfies the following estimate
\begin{align}\label{k11}
&\sum_{i=1}^2 \| H_i \|_{1,\alpha;{\Omega}}^{(-\alpha;\{{P}_3,{P}_4\})} \notag\\
\leq& C\left(  \sum_{i=1}^2 \|\pounds_i \|_{0,\alpha;{\Omega}}^{(1-\alpha;\{{P}_3,{P}_4\})}+ \| \hbar_1 \|_{1,\alpha;\Gamma_{\mr{s}}}^{(-\alpha;P_4)} +  \| \hbar_3 \|_{1,\alpha;\Gamma_3}^{(-\alpha;P_3)} + \| \hbar_4 \|_{1,\alpha;\Gamma_4}^{(-\alpha;\{P_3, P_4\})}   \right).
\end{align}

Thus, the proof of Theorem 3.1 is completed.
\end{proof}

\begin{rem}\label{opp}
In light of Lemma 2.1 of \cite{YH2007},
in the step 2 of the proof, the condition
$ \partial_y\tilde{\Phi}(x,0) = 0$ can be naturally satisfied, which implies that
\begin{align*}
  \Phi(x,0) =\partial_{yy}\Phi(x,0)=0.
\end{align*}
Furthermore, one can obtain that $\mathcal{U}_2 (x,0) = \partial_y \mathcal{U}_1 (x,0)= 0$.
Moreover, in the step 3 of the proof, one has $\partial_y \Psi(x,0) = \partial_{xy}\Psi(x,0) = 0$, it follows that $\mathcal{W}_2(x,0) = \partial_y \mathcal{W}_1(x,0) = 0$.
Therefore, under the coordinate transformations, one can deduce that $ H_2(\xi,0) = \partial_\eta H_1(\xi,0) = 0$ immediately.
\end{rem}

\section{The initial approximation}

In this section, we are going to prove Theorem \ref{thm:initial_approx_existence} and establish the existence of the solution to the free boundary problem {\bf{$\textit{IFBPL}$}}.

\subsection{The solution $ \dot{U}_{-} $ in $ \Omega $}

Since ${M}_- > 1 $, it is obvious that $ \dot{U}_{-} $ is governed by a system of hyperbolic type.
Then it can be solved in $ \Omega $ by applying classical theory for initial-boundary value problems of hyperbolic system of first order.

\begin{lem}\label{lem1}
Suppose $\eqref{eq:assumption_001}$ and $\eqref{eq20000}$ hold, then there exists a unique solution $\dot{U}_-$ satisfying the linearized equations $\eqref{eq80-}$-$\eqref{eq845-}$, and the initial-boundary conditions $\eqref{eq8561}$, $\eqref{eq8564}$, $\eqref{eq451}$. The solution $\dot{U}_-$ satisfies the following estimate:
\begin{equation}\label{g2}
	\|\dot{U}_-\|_{{\mcc}^{2,\alpha}(\bar{\Omega})}  \leq \dot{C} \| \sigma \Theta(\xi) \|_{{\mcc}^{2,\alpha}(\Gamma_4)}  \leq \dot{C}_{-}\sigma,
	\end{equation}
where the constant $\dot{C}_{-}$ depends on $\bar{U}_-$ and $L$.

Moreover, in the domain $\Omega$, it holds that
\begin{align}
	&\dot{p}_- + \bar{\rho}_-\bar{q}_- \dot{q}_- = 0,\label{eq028}\\
& \dot{s}_- = 0.\label{eq028s}
	\end{align}

Finally, for any fixed $\bar{\xi} \in (0, L)$, it holds that
\begin{equation}\label{eq822}
2\frac{1 - {\bar{M}_-}^2}{ {\bar{\rho}_-^2\bar{q}_-^3}} \int_{0}^{1}\eta\dot{p}_-(\bar{\xi}, \eta)\dif \eta
= \int_0^{\bar{\xi}} \sigma \Theta(\xi) \dif \xi.
\end{equation}

\end{lem}

\begin{proof}

First, the equations \eqref{eq80-}-\eqref{eq843-} can be rewritten as
\begin{equation}\label{M}
\mathcal{M}_1\partial_\xi(\dot{\theta}_-, \dot{p}_-)^T + \mathcal{M}_2\partial_\eta(\dot{\theta}_-, \dot{p}_-)^T + \mathfrak{M} = 0,
\end{equation}
where $\mathfrak{M} = (0,\frac{1}{\eta}\dot{\theta}_-)^T$,
\[ \mathcal{M}_1 = \begin{pmatrix}
	2\bar{q}_- & 0\\
  0& 2\displaystyle\frac{\bar{M}_-^2 -1}{\bar{\rho}_-^2 \bar{q}_-^3},
	
\end{pmatrix},\,\, \mathcal{M}_2=
\begin{pmatrix}
   0& 1\\
   1 & 0
\end{pmatrix}.\]

Direct calculation gives us the eigenvalues and the corresponding left eigenvectors of the matrix $(\mathcal{M}_2 - \lambda\mathcal{M}_1 )$, i.e.
\begin{equation}
\lambda_{\pm} = \pm \displaystyle\frac{\bar{\rho}_- \bar{q}_-}{2\sqrt{\bar{M}_-^2 -1}},\quad  \vec{\ell}_{\pm}=(1, 2\bar{q}_- \lambda_{\pm})^T.
\end{equation}
Then, \eqref{M} implies that
\begin{equation}\label{M1}
  \vec{\ell}_{\pm}\mathcal{M}_1 \partial_\xi(\dot{\theta}_-, \dot{p}_-)^T + \lambda_{\pm}\vec{\ell}_{\pm}\mathcal{M}_1 \partial_\eta(\dot{\theta}_-, \dot{p}_-)^T  +  \vec{\ell}_{\pm}\mathfrak{M} =0.
\end{equation}
That is
\begin{align}\label{i4}
  \partial_\xi \Big(2\bar{q}_-\dot{\theta}_- \pm 2 \displaystyle\frac{\sqrt{\bar{M}_-^2 -1}}{\bar{\rho}_- \bar{q}_-}\dot{p}_-\Big) + \lambda_\pm  \partial_\eta \Big(2\bar{q}_-\dot{\theta}_- \pm 2 \displaystyle\frac{\sqrt{\bar{M}_-^2 -1}}{\bar{\rho}_- \bar{q}_-}\dot{p}_- \Big)&\notag\\
  \pm \displaystyle\frac{\bar{\rho}_- \bar{q}_-^2}{\sqrt{\bar{M}_-^2 -1}}\cdot\frac{1}{\eta}\dot{\theta}_- =& 0.
\end{align}
Let
\begin{equation}\label{kk1}
\begin{aligned}
  w_{\pm}
   =2\bar{q}_-\dot{\theta}_- \pm 2 \displaystyle\frac{\sqrt{\bar{M}_-^2 -1}}{\bar{\rho}_- \bar{q}_-}\dot{p}_-,
\end{aligned}
\end{equation}
then \eqref{i4} can be rewritten as the following form
\begin{align}
&\partial_\xi w_- + \lambda_- \partial_\eta w_- - \displaystyle\frac{\bar{\rho}_- \bar{q}_-^2}{\sqrt{\bar{M}_-^2 -1}}\cdot\frac{1}{\eta}\dot{\theta}_-= 0,\label{g}\\
&\partial_\xi w_+ + \lambda_+ \partial_\eta w_+ + \displaystyle\frac{\bar{\rho}_- \bar{q}_-^2}{\sqrt{\bar{M}_-^2 -1}}\cdot\frac{1}{\eta} \dot{\theta}_- =  0,\label{g1}
\end{align}
in addition, the initial-boundary conditions $\eqref{eq8561}$-$\eqref{eq8564}$ and $\eqref{eq451}$ yield that:
\begin{equation}\label{r}
  w_{\pm} (0,\eta) = 0,\quad
\dot{\theta}_-(\xi,0) = 0,\quad \dot{\theta}_-(\xi, 1) = \sigma \Theta(\xi).
\end{equation}

\begin{figure}[!h]
	\centering
	\includegraphics[width=0.3\textwidth]{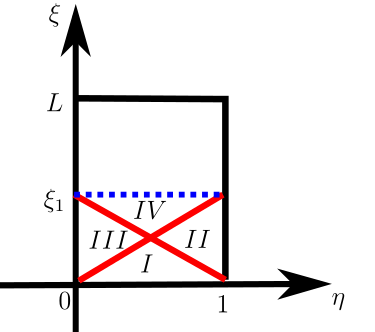}
	\caption{The supersonic flow in the whole domain.\label{fig:6}}
\end{figure}

Using the method of characteristic and for any fixed point $(\xi,\eta)$, letting $\tau\mapsto \Upsilon_\pm(\tau;\xi,\eta)$ be the characteristic line through $(\xi,\eta)$, one has
\begin{equation}\label{eq97}
\displaystyle\frac{\dif \Upsilon_\pm (\tau;\xi,\eta)}{\dif \tau} = \pm \displaystyle\frac{\bar{\rho}_- \bar{q}_-}{2\sqrt{\bar{M}_-^2 -1}},\quad
\Upsilon_\pm (\xi;\xi,\eta)=\eta.
\end{equation}
Direct calculation yields that
\begin{equation}\label{eq98}
  \Upsilon_\pm(\tau;\xi,\eta) = \eta \pm \displaystyle\frac{\bar{\rho}_- \bar{q}_-}{2\sqrt{\bar{M}_-^2 -1}}(\tau - \xi).
\end{equation}

In the domains $I$ and $II$ (see Figure \ref{fig:6}), by the characteristic method and Picard iteration (see \cite{LT1985}), one can obtain the solution $(\dot{\theta}_-,\dot{p}_-)$, which belong to ${\mcc}^{2,\alpha}$. In the domain $III$, since the singularity of the coefficients in the equations \eqref{g}-\eqref{g1}, one needs to pay more attention to the area near the axis $\eta = 0$.

Direct calculations yield that
\begin{equation}\label{r1}
  w_-(\xi ,\eta)= \displaystyle\frac{\bar{\rho}_- \bar{q}_-^2}{\sqrt{\bar{M}_-^2 -1}}\int_0^\xi \frac{\dot{\theta}_-(\tau,\Upsilon_-(\tau;\xi,\eta))}{\Upsilon_-(\tau;\xi,\eta)}\dif \tau.
\end{equation}
Then, employing the equations \eqref{kk1} and \eqref{r}, one has
\begin{equation}
   w_+(\xi ,0)= - w_-(\xi ,0),
\end{equation}
therefore,
\begin{equation}\label{r2}
  \begin{aligned}
     w_+(\xi ,\eta) = - w_-(\xi^* ,0)- \displaystyle\frac{\bar{\rho}_- \bar{q}_-^2}{\sqrt{\bar{M}_-^2 -1}}\int_{\xi^*}^\xi \frac{\dot{\theta}_-(\tau,\Upsilon_+(\tau;\xi,\eta))}{\Upsilon_+(\tau;\xi,\eta)}\dif \tau,
  \end{aligned}
\end{equation}
where $\xi^* = \xi - \displaystyle\frac{2\sqrt{\bar{M}_-^2 -1}}{\bar{\rho}_- \bar{q}_-}\eta$.

 Since $\dot{\theta}_-$ belongs to $\mcc^{2,\alpha}$ in the domain $III$ when $\eta\in (\eta_0, \frac12)$ with any fixed $\eta_0 >0$. Then, by employing the condition $\dot{\theta}_-(\xi,0)= 0$, thus \eqref{r1} and \eqref{r2} are well-defined as $\eta$ tends to zero.

 Moreover, one can follow the proofs in the Chapter1-Chapter2 of the book \cite{LT1985} to obtain the first and second order derivative estimates. For example,
\begin{align}\label{qw}
  &\frac{\partial w_-(\xi ,\eta)}{\partial\eta}\notag\\
  = &\displaystyle\frac{\bar{\rho}_- \bar{q}_-^2}{\sqrt{\bar{M}_-^2 -1}}\int_0^\xi \frac{\Upsilon_-\partial_{\Upsilon_-}\dot{\theta}_-
  (\tau,\Upsilon_-(\tau;\xi,\eta))
   -\dot{\theta}_-(\tau,\Upsilon_-(\tau;\xi,\eta)) }
  {\Upsilon_-^2(\tau;\xi,\eta)}\frac{\partial\Upsilon_-}{\partial\eta}\dif \tau.
\end{align}
Notice that
\begin{align}
  \frac{\eta\partial_\eta \dot{\theta}_-(\xi,\eta)  - \dot{\theta}_-(\xi,\eta)}{\eta^2} \rightarrow -\frac12\partial_\eta^2\dot{\theta}_-(\xi, 0),\quad \text{as}\quad \eta \rightarrow 0+,
\end{align}
thus \eqref{qw} is well-defined as $\eta$ tends to zero. Other cases can be treated in a similar way by applying the conditions $\dot{\theta}_-(\xi,0)=0$ and $\partial_\eta^2\dot{\theta}_-(\xi,0)=0$. Thus, by employing Picard iteration, $w_\pm$ are well defined in $III$ and belong to ${\mcc}^{2,\alpha}$. In the regions $IV$ and $\xi > \xi_1 : = \displaystyle\frac{2\sqrt{\bar{M}_-^2 -1}}{\bar{\rho}_- \bar{q}_-}$, the proof is analogous.
Therefore, $(\dot{\theta}_-, \dot{p}_-)$ satisfies
 \begin{equation}
	\|(\dot{\theta}_-, \dot{p}_-)\|_{{\mcc}^{2,\alpha}(\bar{\Omega})}  \leq \dot{C} \| \sigma \Theta(\xi) \|_{{\mcc}^{2,\alpha}(\Gamma_4)}.
	\end{equation}

 Moreover, by applying \eqref{eq844-} and \eqref{eq845-}, $\eqref{eq028}$-\eqref{eq028s} can be obtained immediately.
 Furthermore, \eqref{g2} holds.

 To show \eqref{eq822} holds, one observes that the equation $\eqref{eq843-}$ implies that
 \begin{align}\label{syy}
   \int_{\Omega_{\bar{\xi}}}\Big(\partial_\eta (\eta \dot{\theta}_-) + 2\displaystyle\frac{{\bar{M}_-}^2 -1}{ {\bar{\rho}_-^2\bar{q}_-^3}}\partial_\xi (\eta \dot{p}_-)\Big) \dif \xi \dif \eta =0,
 \end{align}
 where
 \begin{align*}
  {\Omega}_{\bar{\xi}} = \{ (\xi, \eta)\in \mathbb{R}^2 : 0 < \xi < \bar{\xi}, \,\,\,0 < \eta< 1\}.
\end{align*}
Furthermore, \eqref{syy} yields that
 \begin{align}
  0= &\int_{\partial\Omega_{\bar{\xi}}} 2\displaystyle\frac{{\bar{M}_-}^2 -1}{ {\bar{\rho}_-^2\bar{q}_-^3}}(\eta \dot{p}_-)\dif \eta   - (\eta \dot{\theta}_-) \dif \xi\notag\\
  =&  2\displaystyle\frac{{\bar{M}_-}^2 -1}{ {\bar{\rho}_-^2\bar{q}_-^3}}\int_0^1 \eta \dot{p}_-(\bar{\xi}, \eta)\dif \eta + \int_0^{\bar{\xi}} \sigma \Theta(\xi) \dif \xi,
 \end{align}
which is exactly the equation \eqref{eq822}.
\end{proof}

\subsection{Reformulation of the linearized R-H conditions \eqref{eq090}.}

As previously pointed out in Remark \ref{rem:linearized_RH}, the equations \eqref{eq090} form a closed linear algebraic equations for $ (\dot{p}_+, \dot{q}_+, \dot{s}_+) $ such that they can be expressed by $ \dot{U}_{-} $ on the free boundary  $ \dot{\Gamma}_{\mr{s}} $. Although the computations are almost the same as in \cite{FB63}, we still give them below for convenience of the readers.

First, the equations \eqref{eq090} can be rewritten as the following form:
\begin{equation}\label{A_s}
  A_{\mr{s}} (\dot{p}_+, \dot{q}_+,\dot{s}_+)^T = (\dot{J}_1 , \dot{J}_2,\dot{J}_3)^T,
\end{equation}
where $\dot{J}_i : = -{\mathbf{\alpha}}_{j-} \cdot {\dot{U}}_-$,
\[ A_{\mr{s}} := \displaystyle\frac{[\bar{p}]}{\bar{\rho}_+\bar{q}_+}\begin{pmatrix}
	-\displaystyle\frac{1}{\bar{\rho}_+ \bar{c}_+^2}&-\displaystyle\frac{1}{\bar{q}_+} & \displaystyle\frac{1}{\gamma c_v}\\
1 - \displaystyle\frac{\bar{p}_+}{\bar{\rho}_+ \bar{c}_+^2}& \bar{\rho}_+\bar{q}_+ - \displaystyle\frac{\bar{p}_+}{\bar{q}_+}& \displaystyle\frac{\bar{p}_+}{\gamma c_v}\\
	\displaystyle\frac{\bar{q}_+}{[\bar{p}]}& \displaystyle\frac{\bar{\rho}_+\bar{q}_+^2}{[\bar{p}]}& \displaystyle\frac{\bar{p}_+}{(\gamma - 1)c_v}\displaystyle\frac{\bar{q}_+}{[\bar{p}]}
	\end{pmatrix}. \]	

Then we have the following lemma.

\begin{lem}\label{lem2}
	On $\dot{\Gamma}_{\mr{s}}$, it holds that
	\begin{align}
  & \det A_{\mr{s}}  = \displaystyle\frac{1}{(\gamma -1)c_v} \displaystyle\frac{[\bar{p}]^2 \bar{p}_+}{(\bar{\rho}_+\bar{q}_+)^3}(1 - \bar{M}_+^2)\neq 0,\quad \text{as}\quad \bar{M}_+\neq 1, \label{AS}\\
	&\dot{p}_+  = \displaystyle\frac{\bar{\rho}_+ \bar{q}_+^2}{\bar{M}_+^2 -1}\cdot \displaystyle\frac{\bar{M}_-^2 - 1}{ \bar{\rho}_-\bar{q}_-^2} (1 - \dot{k})\dot{p}_- : = \dot{g}_1,\label{eq10002}\\
	&\dot{q}_+ = \displaystyle\frac{\bar{M}_-^2 - 1}{\bar{\rho}_- \bar{q}_-^2}\left([\bar{p}] - \displaystyle\frac{\bar{\rho}_+ \bar{q}_+^2}{\bar{M}_+^2 -1}(1 - \dot{k})\right)\dot{p}_- : = \dot{g}_2,\label{eq223}\\
	&\dot{s}_+  = - \displaystyle\frac{(\gamma - 1) c_v }{\bar{p}_+}\cdot\displaystyle\frac{\bar{M}_-^2 -1}{\bar{\rho}_- \bar{q}_-^2}[\bar{p}]\dot{p}_- : = \dot{g}_3 ,\label{eq10003}\\
	&{\dot{\psi}}'  = 2\displaystyle\frac{\bar{q}_+\dot{\theta}_+  - \bar{q}_- \dot{\theta}_-}{[\bar{p}]}: = \dot{g}_4,\label{eq10000}
	\end{align}
where $\dot{k} \defs [\bar{p}]\Big(\displaystyle\frac{\gamma -1}{\gamma \bar{p}_+} + \displaystyle\frac{1}{\bar{\rho}_+ \bar{q}_+^2}\Big) >0$.
\end{lem}

\begin{proof}
The proof of \eqref{AS}-\eqref{eq10003} is similar as Lemma 3.2 in \cite{FB63}. Moreover, combining \eqref{eq021} with \eqref{s} yields \eqref{eq10000}.
\end{proof}

\subsection{Determine $ \dot{\xi}_{*} $ and $ \dot{U}_{+} $.}

We are now ready to determine $\dot{\xi}_*$ and $\dot{U}_+$. Applying Theorem 3.1, we can obtain the following lemma.

 \begin{lem}\label{lem3}
Let $\alpha\in (\frac12,1)$. Assume $\eqref{eq:assumption_001}$ and $\eqref{eq20000}$ hold.
 If \begin{equation}\label{eq067}
\mathcal{R}_* < {\mathcal{P}}_{\mr{e}} < \mathcal{R}^*,
\end{equation}
 where
 \begin{align}
   &{\mathcal{P}}_{\mr{e}} := 2\displaystyle\frac{1 - \bar{M}_+^2}{\bar{\rho}_+^2 \bar{q}_+^3}\int_0^1  \eta P_{\mr{e}} (\eta) \dif \eta,\\
   &\mathcal{R}^* : = \int_{0}^{L} \Theta(\xi)\dif \xi,\,\,\, \mathcal{R}_*:
= ( 1 - \dot{k})\int_{0}^{L} \Theta(\xi) \dif \xi,
 \end{align}
 with $\dot{k} \defs [\bar{p}]\Big(\displaystyle\frac{\gamma -1}{\gamma \bar{p}_+} + \displaystyle\frac{1}{\bar{\rho}_+ \bar{q}_+^2}\Big) >0$, then there exists $\dot{\xi}_*\in (0,L)$ such that
\begin{equation}\label{eq066}
\mathcal{R}(\dot{\xi}_*) = {\mathcal{P}}_{\mr{e}}.
\end{equation}

Thus, there exists a unique solution $(\dot{\theta}_+, \dot{p}_+)$ to the equations $\eqref{eq80+}$-$\eqref{eq843+}$ with the boundary conditions $\eqref{eq8584}$-$\eqref{eq8583}$, $\eqref{eq461}$ and $\eqref{eq10002}$. Moreover, the solution satisfies the following estimate:
\begin{align}\label{iz}
 \|(\dot{\theta}_+, \dot{p}_+) \|_{1,\alpha;\dot{\Omega}_+}^{(-\alpha;\{Q_3,Q_4\})}
\leq C \sigma \Big(\| \Theta \|_{\mcc^{2,\alpha}({\Gamma}_4)} + \| P_{\mr{e}} \|_{\mcc^{2,\alpha}(\Gamma_3)}\Big) \leq \dot{C}_+ \sigma,
\end{align}
where the constant $\dot{C}_+$ depends on $\bar{U}_+$, $L$, $\dot{\xi}_*$ and $\alpha$.
 \end{lem}

\begin{proof}
 Applying Theorem \ref{sub} and taking
 \begin{align}\label{eq01}
&\mathcal{A }: = 2\bar{q}_+,\quad \mathcal{B}: = 2\frac{1 - \bar{M}_+^2}{ \bar{\rho}_+^2 \bar{q}_+^3}, \quad  H_1 :=  \dot{p}_+ ,\quad H_2 := \dot{\theta}_+,\notag\\
&\pounds_1 = \pounds_2 = 0,\quad \hbar_1 := \dot{g}_1,\quad \hbar_3 := \sigma \dot{P}_{\mr{e}} ,\quad \hbar_4 := \sigma \dot{\Theta},
\end{align}
 one has that there exists a unique solution $(\dot{\theta}_+, \dot{p}_+)$ to the initial linearized problem $\eqref{eq80+}$-$\eqref{eq843+}$ with the boundary conditions $\eqref{eq8584}$-$\eqref{eq8583}$, $\eqref{eq461}$ and $\eqref{eq10002}$ as long as
\begin{equation}\label{eq150}
\begin{aligned}
0 =  2\displaystyle\frac{1 - \bar{M}_+^2}{\bar{\rho}_+^2 \bar{q}_+^3}\int_0^1  \eta\cdot\Big(\dot{g}_1(\dot{\xi}_*,\eta) - \sigma \dot{P}_{\mr{e}}\Big)\dif \eta
+\int_{\dot{\xi}_*}^{L} \sigma \dot{\Theta} \dif \xi.
\end{aligned}
\end{equation}
Applying \eqref{eq10002} in Lemma \ref{lem2}, one has
\begin{equation}
\begin{aligned}
 0 = &2\displaystyle\frac{1}{\bar{\rho}_+\bar{q}_+}\displaystyle\frac
 {1- \bar{M}_-^2 }{\bar{\rho}_- \bar{q}_-^2}(1 - \dot{k}) \int_0^1 \eta \dot{p}_- (\dot{\xi}_*,\eta)\dif \eta -
 2\displaystyle\frac{1 - \bar{M}_+^2}{\bar{\rho}_+^2 \bar{q}_+^3}\sigma \int_0^1 \eta P_{\mr{e}} (\eta)\dif \eta \\
 &+\int_{\dot{\xi}_*}^{L} \sigma \Theta(\xi) \dif \xi.
\end{aligned}
\end{equation}
Then, by employing $\eqref{eq822}$ in Lemma \ref{lem1}, it holds that
\begin{equation}\label{eq156}
(1 - \dot{k})\int_0^{\dot{\xi}_*} \sigma \Theta(\xi)\dif \xi  -  2\displaystyle\frac{1 - \bar{M}_+^2}{\bar{\rho}_+^2 \bar{q}_+^3}\sigma \int_0^1 \eta P_{\mr{e}} (\eta)\dif \eta  + \int_{\dot{\xi}_*}^{L} \sigma \Theta(\xi) \dif \xi= 0,
\end{equation}
that is
\begin{align}\label{eq063}
\int_{0}^{L} \Theta(\xi) \dif \xi   - \dot{k}\int_0^{\dot{\xi}_*} \Theta(\xi)\dif \xi = 2\displaystyle\frac{1 - \bar{M}_+^2}{\bar{\rho}_+^2 \bar{q}_+^3}\int_0^1 \eta P_{\mr{e}} (\eta)\dif \eta.
\end{align}
Let
\begin{align}
\mathcal{R}(\xi) : =& \int_{0}^{L} \Theta(\tau) \dif \tau  - \dot{k}\int_0^{\xi} \Theta(\tau)\dif \tau,\\
{\mathcal{P}}_{\mr{e}} := & 2\displaystyle\frac{1 - \bar{M}_+^2}{\bar{\rho}_+^2 \bar{q}_+^3}\int_0^1  \eta P_{\mr{e}} (\eta) \dif \eta.
\end{align}
If $\eqref{eq:assumption_001}$ holds, it is easy to see that $\mathcal{R}'(\xi)= -\dot{k} \Theta(\xi) < 0$ for all $\xi\in (0 , L)$, which implies that $\mathcal{R}$ is strictly decreasing, then one has
\begin{align}
&\mathcal{R}^*\defs \sup\limits_{\xi\in(0,L)}\mathcal{R} (\xi) = \mathcal{R}(0) = \int_{0}^{L} \Theta(\xi)\dif \xi,\\
 &\mathcal{R}_*\defs \inf\limits_{\xi\in(0,L)}\mathcal{R}(\xi) =\mathcal{R}(L) = ( 1 - \dot{k})\int_{0}^{L} \Theta(\xi) \dif \xi.
\end{align}
Obviously, there exists a unique $\dot{\xi}_*$ such that $\mathcal{R}(\dot{\xi}_*) =  {\mathcal{P}}_{\mr{e}}$ if and only if \begin{equation}\label{x1}
( 1 - \dot{k})\int_{0}^{L} \Theta(\xi) \dif \xi<  {\mathcal{P}}_{\mr{e}}  < \int_{0}^{L} \Theta(\xi) \dif \xi.
\end{equation}
 Moreover, applying $\eqref{eq083}$ of Theorem \ref{sub}, one obtains
\begin{equation}
 \|(\dot{\theta}_+, \dot{p}_+) \|_{1,\alpha;\dot{\Omega}_+}^{(-\alpha;\{Q_3,Q_4\})} \leq C\left(\|\dot{g}_1 \|_{1,\alpha;{\dot{\Gamma}_{\mr{s}}}}^{(-\alpha;Q_4)} +  \sigma\|\Theta \|_{\mcc^{2,\alpha}(\Gamma_4)} + \sigma \|P_{\mr{e}} \|_{\mcc^{2,\alpha}(\Gamma_3)}\right).
\end{equation}
By employing Lemma \ref{lem1} and Lemma \ref{lem2}, $\eqref{iz}$ can be obtained immediately.

\end{proof}

Once $\dot{\xi}_*$ is determined, then we can determine $\dot{U}_+$ in the domain $\dot{\Omega}_+$, and obtain the following lemma:

 \begin{lem}\label{lem5}
    Let $\alpha\in (\frac12,1)$. Under the assumptions of Lemma \ref{lem3}, then there exists a unique solution $\dot{U}_+$ satisfying the linearized equations $\eqref{eq80+}$-$\eqref{eq845+}$ in $\dot{\Omega}_+$ and the boundary conditions $\eqref{eq8584}$-$\eqref{eq8583}$, $\eqref{eq461}$ and $\eqref{eq10002}$-\eqref{eq10003}. Moreover, one has the following estimate:
\begin{equation}\label{eq089}
\begin{aligned}
 \|\dot{U}_+ \|_{(\dot{\Omega}_+, \dot{\Gamma}_{\mr{s}})}+ \| \dot{\psi}'\|_{1,\alpha;\dot{\Gamma}_{\mr{s}}}^{(-\alpha;Q_4)}
\leq C \sigma \left(\| \Theta \|_{\mcc^{2,\alpha}({\Gamma}_4)} + \| P_{\mr{e}} \|_{\mcc^{2,\alpha}(\Gamma_3)}\right) \leq \dot{C}_+ \sigma,
\end{aligned}
\end{equation}
where the constant $\dot{C}_+$ depends on $\bar{U}_+$, $L$, $\dot{\xi}_*$ and $\alpha$.
 \end{lem}

\begin{proof}
It suffices to show the existence of $(\dot{q}_+, \dot{s}_+,\dot{\psi}')$ and establish the estimates.

 It is obvious that \eqref{eq844+} and \eqref{eq845+} indicate that
\begin{align}
  &\dot{s}_+(\xi,\eta) = \dot{s}_+(\dot{\xi}_*,\eta),\\
  &\dot{q}_+(\xi, \eta) = \displaystyle\frac{1}{\bar{q}_+}\left(\Big(\bar{q}_+ \dot{q}_+ + \displaystyle\frac{1}{\bar{\rho}_+} \dot{p}_+ + \bar{T}_+ \dot{s}_+\Big)(\dot{\xi}_*,\eta) - \Big(\displaystyle\frac{1}{\bar{\rho}_+} \dot{p}_+ + \bar{T}_+ \dot{s}_+\Big)(\xi,\eta)  \right).
\end{align}
Therefore, there exists a unique solution $(\dot{q}_+ ,\dot{s}_+)$ satisfying
\begin{equation}
\begin{aligned}
  &\|(\dot{q}_+, \dot{s}_+) \|_{1,\alpha;\dot{\Gamma}_{\mr{s}}}^{(-\alpha;Q_4)}
  + \| (\dot{q}_+,\dot{s}_+) \|_{0,\alpha;\dot{\Omega}_+}^{(1-\alpha;\{Q_3,Q_4\})}\\
  \leq & C\Big(\sum_{i=1}^{3} \|\dot{g}_i \|_{1,\alpha;{\dot{\Gamma}_{\mr{s}}}}^{(-\alpha;Q_4)}+
  \|\dot{p}_+ \|_{0,\alpha;\dot{{\Omega}}_+}^{(1-\alpha;\{Q_3,Q_4\})} \Big).
\end{aligned}
\end{equation}

Finally, according to the definition of $ \dot{\psi}'$ in \eqref{eq10000}, there exists a unique $\dot{\psi}'$ satisfying
\begin{equation}
  \| \dot{\psi}'\|_{1,\alpha;\dot{\Gamma}_{\mr{s}}}^{(-\alpha;Q_4)} \leq C \Big(\|\dot{\theta}_+ \|_{1,\alpha;{\dot{\Gamma}_{\mr{s}}}}^{(-\alpha;Q_4)} + \|\dot{\theta}_- \|_{1,\alpha;{\dot{\Gamma}_{\mr{s}}}}^{(-\alpha;Q_4)} \Big).
\end{equation}

By employing Lemma \ref{lem1} and Lemma \ref{lem2}, $\eqref{eq089}$ can be obtained immediately.

\end{proof}

\begin{rem}\label{axis}
On the the axis $\eta = 0$, the condition \eqref{eq461} is naturally satisfied.

First, according to Remark \ref{opp}, it follows that $\dot{\theta}_+(\xi,0) = 0$ and $\partial_\eta \dot{p}_+ (\xi,0)=0$.
Moreover, differentiating both sides of equation \eqref{eq843+} with respect to $\eta$, one can obtain that $\partial_\eta^2 \dot{\theta}_+(\xi,0) = 0$.
Finally, according to $\partial_\eta \dot{p}_-(\xi,0) = 0$, it follows from  $ \eqref{eq10002}-\eqref{eq10003}$ that
 \begin{align}\label{wy}
   \partial_\eta (\dot{p}_+, \dot{q}_+,\dot{s}_+)(\dot{\xi}_*, 0) = 0.
 \end{align}
Differentiating both sides of equation \eqref{eq845+} with respect to $\eta$ gives that
\begin{equation}
  \partial_\xi(\partial_\eta\dot{s}_+) = 0,
\end{equation}
by applying \eqref{wy}, it is easy to see that $\partial_\eta\dot{s}_+(\xi,0) = 0$. Differentiating both sides of equation \eqref{eq844+} with respect to $\eta$, one can also obtain $\partial_\eta  \dot{q}_+(\xi,0) = 0$.
\end{rem}

\section{The nonlinear iteration scheme and the linearized problem}

In this section, we shall take solution $ (\dot{U}_{-},\ \dot{U}_{+},\ \dot{\xi}_*;\ \dot{\psi}') $ as an initial approximating solution and design a nonlinear iteration scheme to determine the shock solution to the problem {\bf $\textit{FBPL}$}.

\subsection{The supersonic flow $U_-$ in $\Omega$}

If the shock front does not appear in $\Omega$, then we have the following lemma with respect to $U_-$.

\begin{lem}\label{lem51}
   Suppose $\eqref{eq:assumption_001}$ and $\eqref{eq20000}$ hold, then there exists a positive constant $\sigma_L$ depending on $\bar{U}_-$ and $L$, such that for any $0<\sigma< \sigma_L$, there exists a unique solution $U_-\in {\mcc}^{2,\alpha}(\bar{\Omega})$ to the equations $\eqref{eq2916}$-$\eqref{eq2919}$ with the initial-boundary conditions $\eqref{eq808}$-$\eqref{eq909}$ and $\eqref{45}$, moreover, denote $U_- : = \bar{U}_- + \delta U_-$, then the following estimates hold for $\alpha \in (\frac12,1)$:
\begin{align}
 &\|\delta U_- \|_{{\mcc}^{2,\alpha}(\bar{\Omega})} \leq  C_{L} \sigma,\label{eq837}\\
 &\|\delta U_- - \dot{U}_- \|_{\mcc^{1,\alpha}(\bar{\Omega})}
  \leq C_{L} \sigma^2,\label{eq838}
\end{align}
where the constant $C_{L}$ depends on $\bar{U}_-$ and $L$.

 \end{lem}

\begin{proof}
Similar as Lemma \ref{lem1}, the unique existence of the solution $U_-\in {\mcc}^{2,\alpha}(\bar{\Omega})$ can be obtained by employing the characteristic method and Picard iteration as in the book \cite{LT1985}. Thus, it suffices to show that $\eqref{eq838}$ holds.

The equations $\eqref{eq2916}$-$\eqref{eq2919}$ can be rewritten as the following matrix form:
\begin{equation}\label{eq183}
B_0 \partial_\eta U + B_1 (U) \partial_\xi U + b(U)  = 0,
\end{equation}
where $U= (\theta , p, q, s)^T$, $b(U) = (0, \displaystyle\frac{2\eta}{r^2}\displaystyle\frac{\sin\theta}{\rho q}, 0, 0)^T,$ and
\[ B_0 = \begin{pmatrix}
   0 & 1 & 0 & 0 \\
   1 & 0 & 0 & 0 \\
   0 & 0 & 0 & 0 \\
   0 & 0 & 0 & 0
\end{pmatrix},\quad B_1(U)=
\begin{pmatrix}
   \displaystyle\frac{2\eta}{r}q\cos\theta &-\displaystyle\frac{2\eta}{r}\displaystyle\frac{\sin\theta}{\rho q}  & 0 & 0 \\
 -\displaystyle\frac{2\eta}{r}\displaystyle\frac{\sin\theta}{\rho q}& \displaystyle\frac{2\eta}{r}\displaystyle\frac{\cos\theta}{\rho q }\displaystyle\frac{M^2 -1}{\rho q^2} & 0 & 0 \\
  0 & 1&  \rho q & 0 \\
  0 & 0 &  0 & 1
\end{pmatrix} .\]
Therefore, $\delta U_- - \dot{U}_-$ satisfies the following equation
\begin{align}
   &B_0 \partial_\eta (\delta U_- - \dot{U}_-) + B_1(\bar{U}_-) \partial_\xi(\delta U_- - \dot{U}_-) +\nabla b(\bar{U}_-) (\delta U_- - \dot{U}_-)\notag\\
   = & F(\delta U_-), \quad \text{in} \quad \Omega
\end{align}
with the initial-boundary conditions
\begin{align*}
& \delta U_- - \dot{U}_- = 0,  &\text{on} \quad \Gamma_1\\
& \delta \theta_- - \dot{\theta}_- = 0,  \partial_\eta (\delta p_--\dot{p}_-, \delta q_- - \dot{q}_-, \delta s_- - \dot{s}_-) = 0, \partial_\eta^2 (\delta \theta_- - \dot{\theta}_-) = 0,&\text{on} \quad \Gamma_2\\
&\delta \theta_- - \dot{\theta}_- = 0, &\text{on} \quad \Gamma_4
 \end{align*}
where
\begin{align}
  F(\delta U_-):  =&  \Big( B_1(\bar{U}_-) - B_1({U}_-) \Big)\partial_\xi \delta U_- -\Big(b(U_-) - b(\bar{U}_-) - \nabla b(\bar{U}_-) \delta U_-\Big).
\end{align}

Similar as the proof of Lemma \ref{lem1}, one can obtain that
\begin{equation}\label{eq186}
\begin{aligned}
\| \delta U_- - \dot{U}_- \|_{{\mcc}^{1,\alpha}(\bar{\Omega})}\leq& C \|F(\delta U_-)\|_{{\mcc}^{1,\alpha}(\bar{\Omega})}\\
 \leq& C\Big( \|\partial_\xi \delta U_- \|_{{\mcc}^{1,\alpha}(\bar{\Omega})}\cdot \|\delta U_- \|_{{\mcc}^{1,\alpha}(\bar{\Omega})} + \|\delta U_- \|_{{\mcc}^{1,\alpha}(\bar{\Omega})}^2\Big) \\
 \leq& C_L \sigma^2.
\end{aligned}
\end{equation}

\end{proof}

\subsection{The shock front and subsonic flow}
 Assume that for the given quantities $\Theta$ and $P_{\mr{e}}$, there exists a shock front $\Gamma_{\mr{s}}$ whose location is close to the initial approximating location $\dot{\Gamma}_{\mr{s}}$:
\begin{align}
  \Gamma_{\mr{s}}\defs \{(\xi,\eta)\in \mathbb{R}^2: \xi = \psi(\eta): = \dot{\xi}_* + \delta \psi(\eta),\, 0<\eta<1\}.
\end{align}
Then the subsonic region is
\begin{align}
   {\Omega}_+ = \{ (\xi, \eta)\in \mathbb{R}^2 : \psi(\eta) < \xi < L,\, 0 < \eta < 1\},
\end{align}
and the subsonic flow $U_+$ is supposed to be closed to $\bar{U}_+$.

Thus, $(U_+;\psi)$ satisfies the following free boundary value problem
\begin{align}
&B_0 \partial_\eta U_+ + B_1 (U_+) \partial_\xi U_+ + b(U_+)  = 0,&\quad &\text{in}\quad \Omega_+\label{w1}\\
&\text{The R-H conditions} \quad\eqref{eq82}-\eqref{eq3999},&\quad &\text{on}\quad \Gamma_{\mr{s}}\\
&\theta_+ = 0 , \quad \partial_\eta (p_+, q_+, s_+) = 0,\quad \partial_\eta^2\theta_+ = 0,&\quad &\text{on} \quad \Gamma_{2}\cap\overline{\Omega_+}\\
&\theta_+ = \sigma \Theta(\xi),&\quad  &\text{on} \quad \Gamma_{4}\cap\overline{\Omega_+}\\
&p_+ =  \sigma P_{\mr{e}}  (r(L,\eta)),&\quad  &\text{on}\quad \Gamma_3\label{w2}
\end{align}
 where, in the R-H conditions \eqref{eq82}-\eqref{eq3999}, $U_-$ are given by the supersonic flow determined in Lemma 5.1. Thus, the next step is to solve this free boundary value problem near $(\bar{U}_+; \dot{\psi})$. It should be pointed out that the free boundary $\psi$ will be determined by the shape of the shock front $\psi'$ and an exact point $\xi_* : = \psi(1)$ on the nozzle. That is, $\psi(\eta)$ will be rewritten as below:
\begin{equation}
  \psi(\eta ) = \xi_* - \int_{\eta}^1 \delta \psi'(\tau)\dif \tau,
\end{equation}
where $\xi_* := \dot{\xi}_* + \delta \xi_*$, $\delta \xi_*$ will be determined by the solvability condition for the existence of the solution $U_+$ and $\psi'$ will be determined by the R-H conditions.

First, the following transformation will be employed
\begin{align*}
\mathcal{T} : \begin{cases}
\tilde{\xi} = L + \displaystyle\frac{L - \dot{\xi}_*}{L - {\psi}(\eta)}(\xi - L),\\
\tilde{\eta} = \eta,
\end{cases}
\end{align*}
with the inverse
\begin{align*}
 \mathcal{T}^{-1} : \begin{cases}
\xi = L + \displaystyle\frac{L -  {\psi}(\tilde{\eta})}{L - \dot{\xi}_*}(\tilde{\xi} - L),\\
\eta = \tilde{\eta}.
\end{cases}
\end{align*}
Under this transformation, the domain $\Omega_+$ becomes
 \begin{equation}
   \dot{\Omega}_+ = \{ (\xi, \eta)\in \mathbb{R}^2 : \dot{\xi}_* < \xi < L, 0 < \eta < 1\},
 \end{equation}
which is exactly the domain of initial approximating subsonic domain.

Let $\tilde{U}(\tilde{\xi}, \tilde{\eta})\defs U_+\circ \mathcal{T}^{-1}(\tilde{\xi}, \tilde{\eta})$. Direct calculations yield that $\tilde{U}$ satisfies the following equations in $\dot{\Omega}_+$
\begin{align}
  \frac{(\tilde{\xi} - L)\cdot \psi'(\tilde{\eta})}{L- \psi (\tilde{\eta})}B_0\partial_{\tilde{\xi}}\tilde{U} + B_0 \partial_{\tilde{\eta}}\tilde{U}+ \frac{L - \dot{\xi}_*}{L - {\psi}(\tilde{\eta})}B_1(\tilde{U})\partial_{\tilde{\xi}}\tilde{U} + b(\tilde{U}) = 0,\label{w}
\end{align}
with the boundary conditions
\begin{align}
&\tilde{\theta} = 0 , \, \mathscr{P}_1(\tilde{p}, \tilde{q}, \tilde{s}) = 0,\, \mathscr{P}_2\tilde{\theta} = 0,&\quad &\text{on} &\quad &\Gamma_{2}\cap\overline{\dot{\Omega}_+}\\
  &\tilde{\theta} = \sigma \Theta(\xi)\circ \mathcal{T}^{-1}(\tilde{\xi}, 1), &\quad &\text{on} &\quad &\Gamma_{4}\cap\overline{\dot{\Omega}_+}\\
&\tilde{p} =  \sigma P_{\mr{e}}  (r(L,\tilde{\eta})),&\quad &\text{on}&\quad &\Gamma_3
\end{align}
where the operators
\begin{align*}
  \mathscr{P}_1 \defs& (\zeta_1  \partial_{\tilde{\xi}}+ \partial_{\tilde{\eta}}),\\
   \mathscr{P}_2\defs&(\zeta_1^2 \partial_{\tilde{\xi}}^2 +2\zeta_1\partial_{\xi\eta}+\partial_{\tilde{\eta}}^2 + \zeta_2\partial_{\tilde{\xi}}),
\end{align*}
with
\begin{align*}
   \zeta_1(\eta)\defs \displaystyle\frac{(\tilde{\xi} - L)\cdot \psi'(\tilde{\eta})}{L- \psi (\tilde{\eta})},\quad \zeta_2(\eta)\defs (\tilde{\xi} - L)\displaystyle\frac{\psi'' (\tilde{\eta})(L- \psi (\tilde{\eta})) + (\psi'(\tilde{\eta}))^2}{(L- \psi (\tilde{\eta}))^2}.
\end{align*}
By axisymmetry, it can be anticipated that $\zeta_1(0) =0$.

Moreover, the R-H conditions \eqref{eq82}-\eqref{eq3999} become
\begin{align}
  &G_i(\tilde{U}, U_-(\psi',\xi_*))=0, \quad i=1,2,3,\quad &\text{on}\quad \dot{\Gamma}_{\mr{s}},\\
  &G_4(\tilde{U}, U_-(\psi',\xi_*);\psi)=0, \quad &\text{on}\quad \dot{\Gamma}_{\mr{s}},\label{ww}
\end{align}
where $U_-(\psi',\xi_*)\defs U_-(\psi(\tilde{\eta}),\tilde{\eta})$.

In particular, the nonlinear and nonlocal term $r(\xi,\eta)$ becomes
\begin{align}
\tilde{r}(\tilde{\xi}, \tilde{\eta}) = \left(2 \int_{0}^{\tilde{\eta}}\displaystyle\frac{2 t }{\tilde{\rho} \tilde{q} \cos\tilde{\theta}\Big(\vartheta(\tilde{\xi},t), t\Big)} \dif t\right)^{\frac12},
\end{align}
with $\vartheta(\tilde{\xi}, t)\defs \displaystyle\frac{\Big(L - \xi_* + \int_{t}^{1}\delta \psi'(\tau)\dif \tau \Big)\tilde{\xi }+ \Big(\delta \xi_* - \int_{t}^1 \delta \psi'(\tau)\dif \tau\Big)L }{L - \dot{\xi}_*}$.

Therefore, the free boundary problem \eqref{w1}-\eqref{w2} becomes the fixed boundary problem \eqref{w}-\eqref{ww}. Then we will design an iteration scheme to prove the existence of the solutions.

To simplify the notations, we drop `` $ \tilde{} $ '' in the sequel arguments.

\subsection{The linearized problem for the iteration}

This subsection is devoted to describe the linearized problem for the nonlinear iteration scheme, which will be used to prove the existence of solution to the problem \eqref{w}-\eqref{ww} in the next section.

Given approximating states $U=\bar{U}_+  + \delta U$ of the subsonic flow behind the shock front, as well as approximating shape of the shock front $\psi' = \delta\psi'$, then we update them by a new state ${U}^{*} = \bar{U}_+ + \delta {U}^{*}$ of the subsonic flow and the shape of the shock front ${\psi^*}'= \delta{{\psi}^{*}}'$, which are the solution to the problem described below.

Then try to determine $ ( \delta{U}^*(\xi,\eta),\ \delta{\xi}_*;\ \delta{\psi^*}'(\eta)) $ in $ \dot{\Omega}_+$ such that:
\begin{enumerate}
\item $\delta {U}^{*}:=(\delta \theta^*, \delta p^*, \delta q^*, \delta s^*)$ satisfies the following linearized equations in $\dot{\Omega}_+$
\begin{align}
&\partial_\eta \delta p^* + 2\bar{q}_+ \partial_\xi \delta \theta^* = f_1(\delta U, \delta \psi', \delta \xi_{*}),\label{eq978}\\
& (\partial_\eta \delta \theta^* + \displaystyle\frac{\delta\theta^*}{\eta})-  2 \displaystyle\frac{1- \bar{M}_{+}^2} {\bar{\rho}_{+}^2 \bar{q}_{+}^3}\partial_\xi \delta p^{*} = f_2(\delta U, \delta \psi', \delta \xi_{*}),\label{eq0234}\\
&\partial_\xi \Big(\bar{q }_+ \delta q^*
+ \displaystyle\frac{1}{\bar{\rho}_+} \delta p^* +\bar{ T}_+ \delta s^*\Big) = \partial_\xi f_3(\delta U),\label{eq979}\\
&\partial_\xi \delta s^* = 0,\label{eq980}
\end{align}
where
\begin{align*}
f_1(\delta U, \delta \psi', \delta \xi_{*}) \defs & 2\bar{q}_+ \partial_\xi \delta \theta - \displaystyle\frac{2\eta}{r}\Big(-\displaystyle\frac{\sin\theta}{\rho q}\partial_\xi p +  q\cos\theta \partial_\xi \theta \Big)
 - \frac{({\xi} - L)\cdot \psi'({\eta})}{L- \psi ({\eta})} \partial_\xi p\\
 &+\displaystyle\frac{2\eta}{r}\displaystyle\frac{\delta \xi_* - \int_{\eta}^1 \delta \psi'(\tau)\dif \tau}{L -\psi(\eta)}\Big(\displaystyle\frac{\sin\theta}{\rho q}\partial_\xi p - q\cos\theta \partial_\xi \theta\Big),\\
 f_2(\delta U, \delta \psi', \delta \xi_{*}) \defs &\Big (2 \displaystyle\frac{\bar{M}_{+}^2 -1} {\bar{\rho}_{+}^2 \bar{q}_{+}^3}\partial_\xi \delta p + \displaystyle\frac{\delta\theta}{\eta}\Big)- \Big(\displaystyle\frac{2\eta}{r}\displaystyle\frac{\cos\theta}{\rho q}\displaystyle\frac{M^2 -1}{\rho q^2}\partial_\xi p + \displaystyle\frac{2\eta}{r^2} \displaystyle\frac{\sin\theta}{\rho q}\Big)\\
&-\frac{({\xi} - L)\cdot \psi'({\eta})}{L- \psi ({\eta})} \partial_\xi \theta + \displaystyle\frac{2\eta}{r}\displaystyle\frac{\delta \xi_* - \int_{\eta}^1 \delta \psi'(\tau)\dif \tau}{L -\psi(\eta)}\\
&\cdot\Big( \displaystyle\frac{\sin\theta}{\rho q}\partial_\xi \theta- \displaystyle\frac{\cos\theta}{\rho q}\displaystyle\frac{M^2 -1}{\rho q^2}\partial_\xi p \Big),\\
  f_3 (\delta U) \defs & \Big(\bar{q}_+ \delta q+ \displaystyle\frac{1}{\bar{\rho}_+}\delta p + \bar{T}_+ \delta s \Big) - B(U).
\end{align*}

	\item On the nozzle walls $\Gamma_2$ and $\Gamma_4$,
\begin{align}
  &\delta \theta^* = 0,\, \partial_{{\eta}}( \delta p^*, \delta q^*, \delta s^*)= 0,\, \partial_\eta^2\delta \theta^* = 0, &\quad &\text{on}\quad\Gamma_{2}\cap\overline{\dot\Omega_+},\label{eq240}\\
  & \delta \theta^{*} = \sigma \Theta^* (\xi, \delta \xi_*),&\quad &\text{on } \quad\Gamma_{4}\cap\overline{\dot\Omega_+},\label{eq868}
\end{align}
where $\Theta^* (\xi, \delta \xi_*):= \Theta(\xi)\circ \mathcal{T}^{-1}({\xi}, 1) = \Theta \Big(\displaystyle\frac{L - {\xi}_{*}}{L - \dot{\xi}_*}{\xi} + \displaystyle\frac{\delta \xi_{*}}{L - \dot{\xi}_*}L\Big)$;

	\item  On the exit of the nozzle $\Gamma_3$,
\begin{equation}\label{eq867}
  \delta p^* = \sigma P_{\mr{e}}^*( r(L,\eta;\delta U)): = \sigma P_{\mr{e}}\Big((2 \int_{0}^{{\eta}}\displaystyle\frac{2 t }{\rho u (L, t)} \dif t)^{\frac12}\Big),\quad \text{on}\quad\Gamma_3;
\end{equation}

\item On the fixed shock front $\dot{\Gamma}_{\mr{s}}$, the linearized R-H conditions are as below:
\begin{align}
&\alpha_{j+}\cdot \delta U^* = G_j^*(\delta U,\delta U_-, \delta \psi', \delta \xi_*),\quad {j = 1,2,3},\label{eq251}\\
&\alpha_{4+} \cdot \delta U^*-\frac12 [\bar{p}]\delta {\psi^*}' = G_4^* (\delta U,\delta U_-, \delta \psi', \delta \xi_*),\label{eq11000}
  \end{align}
  where
  \begin{align}
    &G_j^*(\delta U,\delta U_-, \delta \psi', \delta \xi_*): = \alpha_{j+}\cdot \delta U - G_j(U, U_-(\psi', \xi_*)),\\
    &G_4^* (\delta U,\delta U_-, \delta \psi', \delta \xi_*): = \alpha_{4+} \cdot \delta U - \frac12[\bar{p}]\delta \psi' - G_4(U, U_-(\psi', \xi_*);\psi')\label{xxl}.
  \end{align}
  \end{enumerate}

  \begin{rem}
 The boundary conditions \eqref{eq251} can be rewritten as
 \begin{equation}
  A_{\mr{s}} (\delta {p}^*,\delta {q}^*, \delta {s}^*)^T =  (G_1^*, G_2^*, G_3^*)^T : = \mathbf{G},
\end{equation}
 where $A_{\mr{s}}$ is defined by \eqref{A_s}.
 By Lemma \ref{lem2}, we have $\det A_{\mr{s}} \neq 0$. Thus, one has
 \begin{align}
  \delta {p}^*: =& g_1^*,\label{p^*}\\
  \delta {q}^*: = & g_2^*,\label{q^*}\\
   \delta {s}^* : =& g_3^*\label{s^*},
\end{align}
 where $(g_1^*, \, g_2^*, \, g_3^*) = A_{\mr{s}}^{-1} \mathbf{G}$.

Moreover, by \eqref{eq11000}, one has
  \begin{equation}\label{eq870}
    \delta {\psi^*}' = 2\left(\displaystyle\frac{  \alpha_4^+ \cdot \delta U^*- G_4^*(\delta U,\delta U_-, \delta \psi', \delta \xi_*)}{[\bar{p}]}\right)\defs g_4^*.
  \end{equation}
 \end{rem}

Obviously, one needs to construct a suitable function space for $(\delta U, \delta \psi')$ such that $\delta\xi_*$ can be determined, and the iteration mapping
\begin{align*}
  \mathbf{\Pi} : (\delta U; \delta \psi') \mapsto (\delta {U}^*; \delta {\psi^*}';\delta \xi_*)
\end{align*}
 is well defined and contractive.

For simplicity of notations, define the solution $(\delta {U}^*; \delta {\psi^*}';\delta \xi_*)$ to the linearized problem \eqref{eq978}-\eqref{eq11000} near $(\dot{U}_+; \dot{\psi}';0)$ as an operator:
\begin{align}\label{iteration}
  (\delta {U}^*; \delta {\psi^*}';\delta \xi_*) = \mathscr{T}_e(\mathscr{F}; \mathscr{G}; \sigma \Theta^*; \sigma P_{\mr{e}}^*),
\end{align}
where $\mathscr{F}\defs (f_1, f_2, f_3)$, $\mathscr{G}\defs(G_1^*,G_2^*,G_3^*, G_4^* ) $.
In particular,
\begin{align}\label{initial}
   (\dot{U}_+; \dot{\psi}';0) = \mathscr{T}_e(\dot{\mathscr{F}}; \dot{\mathscr{G}}; \sigma \dot{\Theta}; \sigma \dot{P_{\mr{e}}}),
\end{align}
where $\dot{\mathscr{F}}\defs (0,0, B(\bar{U}_+))$, $\dot{\mathscr{G}}\defs (\dot{g}_1, \dot{g}_2, \dot{g}_3, \dot{g}_4)$.

When $\delta\xi_*$ is omitted, it will be denoted by
\begin{align}\label{iteration1}
  (\delta {U}^*; \delta {\psi^*}') = \mathscr{T}(\mathscr{F}; \mathscr{G}; \sigma \Theta^*; \sigma P_{\mr{e}}^*),
\end{align}
and
\begin{align}\label{initial1}
   (\dot{U}_+; \dot{\psi}') = \mathscr{T}(\dot{\mathscr{F}}; \dot{\mathscr{G}}; \sigma \dot{\Theta}; \sigma \dot{P_{\mr{e}}})
\end{align}
respectively.

\vskip 0.5cm
Applying Theorem \ref{sub} and taking
\begin{align}\label{eq05}
&\mathcal{A }: = 2\bar{q}_+,\quad \mathcal{B}: = 2\frac{1 - \bar{M}_+^2}{ \bar{\rho}_+^2 \bar{q}_+^3}, \quad  H_1 := \delta p^* ,\quad H_2 := \delta{\theta}^*,\quad \pounds_1: = f_1,
\notag
\\&\pounds_2 := f_2,\quad \hbar_1 := {g}_1^*,\quad \hbar_3 :=\sigma P_{\mr{e}}^*( r(L,\eta;\delta U)) ,\quad \hbar_4 := \sigma \Theta^*(\xi, \delta \xi_*),
\end{align}
one obtains that the boundary value problem $\eqref{eq978}$-$\eqref{eq0234}$ with the boundary conditions $\eqref{eq240}$-$\eqref{eq867}$ and $\eqref{p^*}$ can be solved if and only if
\begin{align}\label{eq248}
    &\int_{\dot{\Omega}_+} \eta f_2 (\delta U, \delta \psi', \delta \xi_{*})\dif \xi \dif \eta
    \notag\\
=&  2\frac{1 - \bar{M}_+^2}{ \bar{\rho}_+^2 \bar{q}_+^3} \int_0^1 \eta\cdot\Big(g_1^*(\delta U, \delta U_-,\delta \psi', \delta \xi_{*}) - \sigma P_{\mr{e}}^*( r(L,\eta;\delta U)) \Big) \dif\eta\notag\\
&  +\int_{\dot{\xi}_*}^{L} \sigma \Theta^* (\xi, \delta \xi_*) \dif\xi.
\end{align}
Then the following lemma holds:
\begin{lem}\label{lem52}
Suppose that, for given $(\delta U; \delta\psi')$ satisfies $\delta U\in \mathbf{H}_{1,\alpha}^{(-\alpha;\{Q_3,Q_4\})}(\dot{\Omega}_+) $, $\delta\psi'\in\mathbf{H}_{1,\alpha}^{(-\alpha;Q_4)}(\dot{\Gamma}_s) $, $f_1(\xi,0) = 0$, and there exists a $\delta\xi_*$ such that $\eqref{eq248}$ holds. Then there exists a solution $(\delta U^*; \delta{\psi^*}')$ to the linearized problem \eqref{eq978}-\eqref{eq11000}, and satisfying the following estimates:
\begin{align}
  &\| (\delta\theta^*,\delta p^*) \|_{1,\alpha;\dot{\Omega}_+}^{(-\alpha;\{Q_3,Q_4\})}
  \notag\\
  \leq &C\Big(\sum_{i=1}^2 \| f_i \|_{0,\alpha;\dot{\Omega}_+}^{(1-\alpha;\{Q_3,Q_4\})}+ \| g_1^*\|_{1,\alpha;\dot{\Gamma}_{\mr{s}}}^{(-\alpha;Q_4)}\Big)\notag\\
  & + C\left(\| \sigma P_{\mr{e}}^* \|_{1,\alpha;{\Gamma}_3}^{(-\alpha;Q_3)}
  + \| \sigma \Theta^* \|_{1,\alpha;{\Gamma}_4}^{(-\alpha;\{Q_3,Q_4\})}\right),\label{eq03}\\
   &\| (\delta q^*,\delta s^*) \|_{1,\alpha;\dot{\Gamma}_{\mr{s}}}^{(-\alpha;Q_4)}
  + \| (\delta q^*,\delta s^*) \|_{0,\alpha;\dot{\Omega}_+}^{(1-\alpha;\{Q_3,Q_4\})}
  \notag \\
  \leq & C \Big(\sum_{i=1}^{3}\| g_i^* \|_{1,\alpha;\dot{\Gamma}_{\mr{s}}}^{(-\alpha;Q_4)}+ \|f_3 - f_3(\dot{\xi}_*,\eta)\|_{0,\alpha;\dot{\Omega}_+}^{(1-\alpha;\{Q_3,Q_4\})} + \|\delta p^*\|_{0,\alpha;\dot{\Omega}_+}^{(1-\alpha;\{Q_3,Q_4\})}\Big), \label{eq012}\\
& \|\delta {\psi^*}' \|_{1,\alpha;\dot{\Gamma}_{\mr{s}}}^{(-\alpha;Q_4)}
 \leq C\Big( \|\delta U^* \|_{1,\alpha;\dot{\Gamma}_{\mr{s}}}^{(-\alpha;Q_4)} +  \| G_4^* \|_{1,\alpha;\dot{\Gamma}_{\mr{s}}}^{(-\alpha;Q_4)}   \Big) ,\label{eq011}
\end{align}
where the constant $C$ depends on $\bar{U}_{\pm}$, $L$, $\dot{\xi}_*$ and $\alpha$.
 \end{lem}
\begin{proof}
 By employing Theorem \ref{sub}, there exists a unique solution $(\delta \theta^*, \delta p^*)$ to the boundary value problem $\eqref{eq978}$-$\eqref{eq0234}$ with the boundary conditions $\eqref{eq240}$-$\eqref{eq867}$ and $\eqref{p^*}$
  and satisfying the estimate \eqref{eq03}.

  Moreover, by the equations $\eqref{eq979}$-$\eqref{eq980}$ with the initial data \eqref{p^*}-\eqref{s^*}, direct calculation implies that
\begin{align}
 & \delta s^* = g_3^*,\\
 & \bar{q }_+ \delta q^*
+ \displaystyle\frac{1}{\bar{\rho}_+} \delta p^* +\bar{ T}_+ \delta s^* =\Big( \bar{q }_+ \delta q^*
+ \displaystyle\frac{1}{\bar{\rho}_+} \delta p^* +\bar{ T}_+ \delta s^*\Big)(\dot{\xi}_*,\eta) \notag\\
&\qquad \qquad \qquad \qquad \qquad \quad + \Big(f_3 - f_3(\dot{\xi}_*,\eta)\Big),
\end{align}
which yields that there exists a unique solution $(\delta q^*, \delta s^*)$ and it satisfies the estimate \eqref{eq012}.

Finally, by $\eqref{eq870}$, one can obtain \eqref{eq011} immediately.
\end{proof}

\section{Well-posedness and contractiveness of the iteration scheme}
In order to carry out the iteration scheme, one needs to construct a suitable function space for $(\delta U, \delta \psi')$ such that $\delta\xi_*$ can be determined, and the iteration mapping $\mathbf{\Pi}$ is well defined and contractive.

Let $\varepsilon >0$. Define
\begin{align*}
  \textsl{N} (\varepsilon): = &\Big\{( \delta U, \delta\psi'): \| \delta U\|_{(\dot{\Omega}_+;\dot{\Gamma}_{\mr{s}})} +  \| \delta\psi'\|_{1,\alpha;\dot{\Gamma}_{\mr{s}}}^{(-\alpha;Q_4)} \leq \varepsilon,\quad \delta\psi'(0)= 0,\\
 & \quad \delta \theta(\xi,0) = 0,\quad \partial_\eta ( \delta p, \delta q, \delta s) (\xi,0) = 0, \quad \partial_\eta^2\delta \theta(\xi,0) = 0\Big\}.
\end{align*}

First, one needs to show that for given $(\delta U, \delta \psi')$, there exists a $\delta \xi_*$ such that the solvability condition \eqref{eq248} holds. We have the following lemma.

\begin{lem}\label{lem61}
   There exists a positive constant $\sigma_1$ with $0<\sigma_1\ll1$, such that for any $0<\sigma\leq \sigma_1$, if $(\delta U - \dot{U}_+;  \delta \psi' - \dot{\psi}')\in \textsl{N}(\frac12 \sigma^{\frac32})$, then there exists a solution $\delta \xi_{*}$ to the equation $\eqref{eq248}$ satisfying the following estimate:
\begin{equation}\label{d}
  |\delta \xi_{*}|\leq C_{\mr{s}}\sigma,
\end{equation}
where the constant $C_{\mr{s}}$ depends on $\dot{\xi}_*$, $\displaystyle\frac{1}{|\Theta(\dot{\xi}_*)|}$, $\bar{U}_{\pm}$, $L$ and $\alpha$.
 \end{lem}

 \begin{proof}
Define
\begin{equation}\label{eq249}
\begin{aligned}
&I(\delta \xi_*,  \delta {U},\delta \psi', \delta U_-)\\
: =&
 -\int_{\dot{\Omega}_+} \eta f_2 (\delta U, \delta \psi', \delta \xi_{*})\dif \xi \dif \eta+\int_{\dot{\xi}_*}^{L} \sigma \Theta^* (\xi, \delta \xi_*) \dif \xi\\
  &+  2\frac{1 - \bar{M}_+^2}{ \bar{\rho}_+^2 \bar{q}_+^3} \int_0^1 \eta\Big(g_1^*(\delta U, \delta U_-,\delta \psi', \delta \xi_{*})- \sigma P_{\mr{e}}^*( r(L,\eta;\delta U)) \Big) \dif \eta.
\end{aligned}
\end{equation}
It is easy to verify that
\begin{equation}
 I(0, 0, 0, \dot{U}_-) = 0.
\end{equation}
We claim that there exists a sufficiently small constant $\sigma_1>0$, such that for any $0<\sigma\leq \sigma_1$, if $(\delta U - \dot{U}_+;  \delta \psi' - \dot{\psi}')\in \textsl{N}(\frac12 \sigma^{\frac32})$, it holds that
\begin{align}
  \frac{\partial I}{\partial(\delta \xi_*)}(0, 0, 0, \dot{U}_-) \neq 0.
\end{align}
Therefore, by applying the implicit function theorem, there exists a $\delta \xi_{*}$ to the equation \eqref{eq249}.
To prove this, one needs to analyze each term of $I$.

First, by applying $(\delta U - \dot{U}_+;  \delta \psi' - \dot{\psi}')\in \textsl{N}(\frac12 \sigma^{\frac32})$ and Lemma \ref{lem5}, it is easy to see that
\begin{align}\label{kk2}
  \|\delta U\|_{(\dot{\Omega}_+;\dot{\Gamma}_{\mr{s}})} +  \| \delta\psi'\|_{1,\alpha;\dot{\Gamma}_{\mr{s}}}^{(-\alpha;Q_4)}\leq C_* \sigma.
\end{align}
By \eqref{kk2}, it follows that
\begin{equation}\label{eq0103}
\begin{aligned}
   P_{\mr{e}}^*( r(L,\eta;\delta U))&={P}_{\mr{e}}(\eta) + \Big(P_{\mr{e}}( (2 \int_{0}^{{\eta}}\displaystyle\frac{2 t }{\rho u (L, t)} \dif t)^{\frac12}) -P_{\mr{e}}( (2 \int_{0}^{{\eta}}\displaystyle\frac{2 t }{\bar{\rho}_+ \bar{q}_+} \dif t)^{\frac12})\Big)\\
  & = {P}_{\mr{e}}(\eta) + O(1)\sigma,
\end{aligned}
\end{equation}
where we use the assumption $\bar{\rho}_+\bar{q}_+ = 2$, and $O(1)$ depends on $C_*$ and ${P}_{\mr{e}}^{'}$. Thus,
\begin{align}
  \int_0^1 \eta\Big(\sigma  P_{\mr{e}}^*( r(L,\eta;\delta U))\Big) \dif \eta = \int_0^1 \eta\Big(\sigma  {P}_{\mr{e}}(\eta)\Big)\dif \eta + O(1)\sigma^2.
\end{align}
Moreover, direct calculations yield that
\begin{align}
&\int_{\dot{\xi}_*}^{L} \sigma \Theta^* (\xi, \delta \xi_*) \dif \xi\notag\\
=& \sigma\int_{\dot{\xi}_*}^{L}  \Theta(\tau) \dif \tau + \sigma\int_{\dot{\xi}_* + \delta \dot{\xi}_*}^{\dot{\xi}_*} \Theta(\tau)\dif \tau + \sigma\displaystyle\frac{\delta\xi_*}{L- \xi_*}\int_{\dot{\xi}_* + \delta\xi_*}^L \Theta(\tau)\dif \tau.
\end{align}
To estimate $g_1^*$, recalling \eqref{eq251}, for $j=1,2,3$, one has
\begin{align}\label{eq0104}
  G_j^* = &\alpha_j^+\cdot \delta U - G_j \Big(U, U_-(\delta \psi' , \xi_*)\Big)
  \notag\\
   =& \Big(\alpha_j^+\cdot \delta U +\alpha_j^-\cdot \delta U_-(\delta \psi',\xi_*) - G_j (U, U_-(\delta \psi' , \xi_*))\Big)
   \notag\\
   &-\alpha_j^-\cdot \Big(\delta U_-(\delta \psi',\xi_*) - \dot{U}_-(\dot{\xi}_*,\eta)\Big)- \alpha_j^- \cdot \dot{U}_-(\dot{\xi}_*,\eta),
\end{align}
where
\begin{equation}
\begin{aligned}
  &\alpha_j^+\cdot \delta U +\alpha_j^-\cdot \delta U_-(\delta \psi',\xi_*) - G_j (U, U_-(\delta \psi' , \xi_*))\\
   = &\frac{1}{2} \int_{0}^{1} D^2 G_j (\bar{U}_+ + t \delta U; \bar{U}_- + t\delta U_-)\dif t \cdot(\delta U; \delta U_-)^2\\
   =& O(1)\sigma^2.
  \end{aligned}
\end{equation}
Moreover, by Lemma \ref{lem1} and Lemma \ref{lem51}, it holds that
  \begin{align}
   &\delta U_-(\delta \psi',\xi_*) - \dot{U}_-(\dot{\xi}_*,\eta)\notag\\
    =& \Big(\delta U_-(\delta \psi',\xi_*)- \dot{U}_-(\xi_* - \int_{\eta}^{1}\delta \psi'(\tau)\dif \tau,\eta)\Big)\notag\\
    &+
     \Big(\dot{U}_-(\xi_* - \int_{\eta}^{1}\delta \psi'(\tau)\dif \tau
     ,\eta) - \dot{U}_-({\xi}_*,\eta)\Big)+\Big(\dot{U}_-({\xi}_*,\eta) - \dot{U}_-(\dot{\xi}_*,\eta)\Big)\notag\\
      =& O(1) \sigma^2 + \Big(\dot{U}_-({\xi}_*,\eta) - \dot{U}_-(\dot{\xi}_*,\eta)\Big).
  \end{align}
Therefore
\begin{equation}
  G_j^* = - \alpha_j^- \cdot \dot{U}_-(\dot{\xi}_*+ \delta \xi_*,\eta) +  O(1) \sigma^2,
\end{equation}
which yields that
\begin{equation}\label{eq0105}
  g_j^* = \dot{g}_j (\dot{\xi}_*+ \delta \xi_*,\eta) + O(1)\sigma^2,
\end{equation}
where $O(1)$ depends on $C_*$, $\bar{U}_\pm$, $L$ and $\Theta$.
Thus,
\begin{align}
  \int_0^1 \eta g_1^*(\delta U, \delta U_-,\delta \psi', \delta \xi_{*})\dif \eta = \int_0^1 \eta  \dot{g}_j (\dot{\xi}_*+ \delta \xi_*,\eta)\dif \eta + O(1)\sigma^2.
\end{align}
It remains to estimate $f_2$. One has
\begin{align}\label{eq050}
 & f_2(\delta U; \delta \psi', \delta \xi_*)
 \notag\\
 =&\Big (2 \displaystyle\frac{\bar{M}_{+}^2 -1} {\bar{\rho}_{+}^2 \bar{q}_{+}^3} - \displaystyle\frac{2\eta}{r}\displaystyle\frac{\cos\theta}{\rho q}\displaystyle\frac{M^2 -1}{\rho q^2}\Big)\partial_\xi \delta p + \Big(\displaystyle\frac{\delta\theta}{\eta} - \displaystyle\frac{2\eta}{r^2} \displaystyle\frac{\sin\theta}{\rho q} \Big)\notag\\
 &+\displaystyle\frac{L -\xi }{L - \psi(\eta)}\delta \psi'(\eta) \partial_\xi \delta\theta  + \displaystyle\frac{2\eta}{r}\displaystyle\frac{\delta \xi_* - \int_{\eta}^1 \delta \psi'(\tau)\dif \tau}{L -\psi(\eta)} \displaystyle\frac{\sin\theta}{\rho q}\partial_\xi \delta\theta
 \notag\\
 & + \displaystyle\frac{2\eta}{r}\displaystyle\frac{\int_{\eta}^1 \delta \psi'(\tau)\dif \tau}{L -\psi(\eta)}\displaystyle\frac{\cos\theta}{\rho q}\displaystyle\frac{M^2-1}{\rho q^2}\partial_\xi p
 \notag\\
 &+ \Big(\displaystyle\frac{2\eta}{r}\displaystyle\frac{\delta \xi_*}{L -\psi(\eta)}\displaystyle\frac{\cos\theta}{\rho q}\displaystyle\frac{1-M^2}{\rho q^2}\partial_\xi p\Big).
\end{align}
Notice that by \eqref{kk2}, it holds that
\begin{align}\label{aq}
  r =& \bar{r} + \left( \Big(2 \int_0^\eta \frac{2t}{\rho q \cos\theta} \dif t\Big)^\frac12 - \Big(2 \int_0^\eta \frac{2t}{\bar{\rho}_+\bar{q}_+}\dif t\Big)^\frac12\right)\notag\\
  =& \eta + O(1) \sigma\cdot \eta,
\end{align}
where $\bar{r}=\eta$ under the assumption $\bar{\rho}_+\bar{q}_+ = 2$, and $O(1)$ depends on $C_*$. Thus, \eqref{kk2} and \eqref{aq} yield that
\begin{align}\label{6.17}
  &\int_{\dot{\Omega}_+} \eta\left(\Big (2 \displaystyle\frac{\bar{M}_{+}^2 -1} {\bar{\rho}_{+}^2 \bar{q}_{+}^3} - \displaystyle\frac{2\eta}{r}\displaystyle\frac{\cos\theta}{\rho q}\displaystyle\frac{M^2 -1}{\rho q^2}\Big)\partial_\xi \delta p + \Big(\displaystyle\frac{\delta\theta}{\eta} - \displaystyle\frac{2\eta}{r^2} \displaystyle\frac{\sin\theta}{\rho q} \Big)\right) \dif \xi \dif \eta\notag\\
    =& O(1) \|\delta U\|_{(\dot{\Omega}_+;\dot{\Gamma}_{\mr{s}})}^2 =O(1) \sigma^2,
\end{align}
where $O(1)$ depends on $\bar{U}_+$, $\dot{\xi}_*$, $C_*$ and $\alpha$. Moreover,
\begin{align}
  &\int_{\dot{\Omega}_+} \eta\Big(\displaystyle\frac{L -\xi }{L - \psi(\eta)}\delta \psi'(\eta) \partial_\xi \delta\theta  + \displaystyle\frac{2\eta}{r}\displaystyle\frac{\delta \xi_* - \int_{\eta}^1 \delta \psi'(\tau)\dif \tau}{L -\psi(\eta)} \displaystyle\frac{\sin\theta}{\rho q}\partial_\xi \delta\theta\Big) \dif \xi \dif \eta
 \notag\\
 & +\int_{\dot{\Omega}_+} \eta\Big( \displaystyle\frac{2\eta}{r}\displaystyle\frac{\int_{\eta}^1 \delta \psi'(\tau)\dif \tau}{L -\psi(\eta)}\displaystyle\frac{\cos\theta}{\rho q}\displaystyle\frac{M^2-1}{\rho q^2}\partial_\xi p\Big) \dif \xi \dif \eta\notag\\
 = & O(1) \|\delta \psi'\|_{L^\infty(\dot{\Gamma}_s)}\| \delta\theta \|_{1,\alpha;\dot{\Omega}_+}^{(-\alpha;\{Q_3,Q_4\})}
  +  O(1) \|\delta \psi'\|_{L^\infty(\dot{\Gamma}_s)}\|\delta U\|_{(\dot{\Omega}_+;\dot{\Gamma}_{\mr{s}})}^2 \notag\\
  &+ O(1) \|\delta U\|_{(\dot{\Omega}_+;\dot{\Gamma}_{\mr{s}})}^2\delta \xi_* + O(1) \|\delta \psi'\|_{L^\infty(\dot{\Gamma}_s)}\|\delta U\|_{(\dot{\Omega}_+;\dot{\Gamma}_{\mr{s}})}\notag\\
 = & O(1)\sigma^2 +  O(1)\sigma^2\cdot\delta \xi_*,
\end{align}
where $O(1)$ depends on $\bar{U}_+$, $\dot{\xi}_*$, $C_*$ and $\alpha$. Finally, it holds that
  \begin{align}
    &\displaystyle\frac{2\eta}{r}\displaystyle\frac{\delta \xi_*}{L -\psi(\eta)}\displaystyle\frac{\cos\theta}{\rho q}\displaystyle\frac{1 - M^2}{\rho q^2}\partial_\xi p
    \notag\\
    =&\Big(\displaystyle\frac{2\eta}{r} -\displaystyle\frac{2\eta}{\bar{r}} \Big)\displaystyle\frac{\delta \xi_*}{L -\psi(\eta)}\displaystyle\frac{\cos\theta}{\rho q}\displaystyle\frac{1 - M^2}{\rho q^2}\partial_\xi p
    \notag\\
    &+ 2 \displaystyle\frac{\delta \xi_*}{L -\psi(\eta)}\Big(\displaystyle\frac{\cos\theta}{\rho q}\displaystyle\frac{1 - M^2}{\rho q^2} - \displaystyle\frac{1}{\bar{\rho}_+ \bar{q}_+}\cdot\displaystyle\frac{1-\bar{M}_+^2}{\bar{\rho}_+ \bar{q}_+^2}\Big)\partial_\xi \delta p
     \notag\\
    & + 2\displaystyle\frac{\delta \xi_*}{L -\psi(\eta)} \displaystyle\frac{1}{\bar{\rho}_+ \bar{q}_+}\cdot\displaystyle\frac{1-\bar{M}_+^2}{\bar{\rho}_+ \bar{q}_+^2}\partial_\xi (\delta p - \dot{p}_+)
    \notag\\
    &- 2\displaystyle\frac{\delta \xi_*\cdot \int_{\eta}^{1}\delta \psi'(\tau)d\tau}{(L -\psi(\eta))(L - \xi_*)} \cdot \displaystyle\frac{1}{\bar{\rho}_+ \bar{q}_+}\cdot\displaystyle\frac{1-\bar{M}_+^2}{\bar{\rho}_+ \bar{q}_+^2}\partial_\xi \dot{p}_+
    \notag\\
    & + 2\displaystyle\frac{\delta \xi_*}{L -\xi_*} \displaystyle\frac{1}{\bar{\rho}_+ \bar{q}_+}\cdot\displaystyle\frac{1-\bar{M}_+^2}{\bar{\rho}_+ \bar{q}_+^2}\partial_\xi \dot{p}_+.
  \end{align}
By employing $\eqref{eq843+}$, it holds that
\begin{equation}
  \begin{aligned}
    \int_{\dot{\Omega}_+} 2\displaystyle\frac{1}{\bar{\rho}_+ \bar{q}_+}\cdot\displaystyle\frac{1-\bar{M}_+^2}{\bar{\rho}_+ \bar{q}_+^2}\partial_\xi (\eta\dot{p}_+)\dif \xi \dif \eta = \int_{\dot{\xi}_*}^{L} \int_{0}^{1} \partial_\eta (\eta \dot{\theta}_+ )\dif \eta \dif\xi= \sigma \int_{\dot{\xi}_*}^{L} \Theta (\xi)\dif \xi.
  \end{aligned}
\end{equation}
Thus, one has
\begin{align}\label{6.21}
  &\int_{\dot{\Omega}_+} \eta\Big(\displaystyle\frac{2\eta}{r}\displaystyle\frac{\delta \xi_*}{L -\psi(\eta)}\displaystyle\frac{\cos\theta}{\rho q}\displaystyle\frac{1-M^2}{\rho q^2}\partial_\xi p\Big)\dif \xi \dif \eta\notag\\
  =& O(1)\sigma \| \delta p \|_{1,\alpha;\dot{\Omega}_+}^{(-\alpha;\{Q_3,Q_4\})}\cdot \delta \xi_* +O(1)\|\delta U\|_{L^\infty(\dot{\Omega}_+)}\| \delta p \|_{1,\alpha;\dot{\Omega}_+}^{(-\alpha;\{Q_3,Q_4\})} \notag\\
  &+ O(1) \| \delta p - \dot{p}_+\|_{1,\alpha;\dot{\Omega}_+}^{(-\alpha;\{Q_3,Q_4\})}\cdot \delta \xi_* \notag\\
  &+ O(1)\|\delta \psi'\|_{L^\infty(\dot{\Gamma}_s)}\| \dot{p}_+ \|_{1,\alpha;\dot{\Omega}_+}^{(-\alpha;\{Q_3,Q_4\})}\cdot \delta \xi_* + \displaystyle\frac{\delta \xi_*}{L -\xi_*} \sigma \int_{\dot{\xi}_*}^{L} \Theta (\xi)\dif \xi\notag\\
= &\displaystyle\frac{\delta \xi_*}{L -\xi_*} \sigma \int_{\dot{\xi}_*}^{L} \Theta (\xi)\dif \xi+  O(1)\sigma^{\frac32}\cdot \delta \xi_*,
\end{align}
where $O(1)$ depends on $\bar{U}_+$, $\dot{\xi}_*$, $C_*$ and $\alpha$.

Therefore, concluding the estimates \eqref{6.17}-\eqref{6.21} for terms of $f_2$ in \eqref{eq050}, one obtains
\begin{align}
   &-\int_{\dot{\Omega}_+} \eta f_2 (\delta U, \delta \psi', \delta \xi_{*})\dif \xi \dif \eta\notag\\
    =& -\displaystyle\frac{\delta \xi_*}{L -\xi_*} \sigma \int_{\dot{\xi}_*}^{L} \Theta (\xi)\dif \xi +  O(1)\sigma^{\frac32}\cdot \delta \xi_* + O(1) \sigma^2.
\end{align}

Hence, it holds that
\begin{align}\label{eq285}
&I(\delta \xi_*,  \delta {U},\delta \psi', \delta U_-)
\notag \\
= & -\displaystyle\frac{\delta \xi_*}{L -\xi_*} \sigma \int_{\dot{\xi}_*}^{L} \Theta (\xi)\dif \xi +  O(1)\sigma^{\frac32}\cdot \delta \xi_* + O(1) \sigma^2
\notag\\
   &+ \sigma\int_{\dot{\xi}_*}^{L}  \Theta(\tau) \dif \tau + \sigma\int_{\dot{\xi}_* + \delta \dot{\xi}_*}^{\dot{\xi}_*} \Theta(\tau)\dif \tau + \sigma\displaystyle\frac{\delta\xi_*}{L- \xi_*}\int_{\dot{\xi}_* + \delta\xi_*}^L \Theta(\tau)\dif \tau
   \notag\\
    &+2\frac{1 - \bar{M}_+^2}{ \bar{\rho}_+^2 \bar{q}_+^3}\int_0^1 \eta \dot{g}_1(\dot{\xi}_* + \delta\xi_*, \eta) \dif \eta + O(1)\sigma^2
    \notag\\
    &- 2\frac{1 - \bar{M}_+^2}{ \bar{\rho}_+^2 \bar{q}_+^3} \sigma\int_0^1 \eta {P}_{\mr{e}}(\eta)\dif \eta + O(1)\sigma^2
    \notag\\
     =& 2\frac{1 - \bar{M}_+^2}{ \bar{\rho}_+^2 \bar{q}_+^3}\int_0^1 \eta \left(\dot{g}_1(\dot{\xi}_*, \eta) - \sigma {P}_{\mr{e}} (\eta) \right)\dif \eta + \sigma\int_{\dot{\xi}_*}^{L}  \Theta(\tau) \dif \tau
     \notag\\
    & +  2\frac{1 - \bar{M}_+^2}{ \bar{\rho}_+^2 \bar{q}_+^3} \int_0^1 \eta \left(\dot{g}_1(\dot{\xi}_* + \delta \xi_*, \eta) - \dot{g}_1(\dot{\xi}_* , \eta) \right)\dif \eta + \sigma\int_{\dot{\xi}_*+\delta\xi_*}^{\dot{\xi}_*}  \Theta(\tau) \dif \tau
    \notag\\
    &+\displaystyle\frac{\delta \xi_*}{L -\xi_*} \sigma \left( -\int_{\dot{\xi}_*}^{L} \Theta (\tau)\dif \tau + \int_{\dot{\xi}_*+\delta\xi_*}^{L}  \Theta(\tau) \dif \tau\right) + O(1)\sigma^{\frac32}\cdot \delta \xi_* + O(1) \sigma^2.
\end{align}
Then, by employing the equations $\eqref{eq150}$, \eqref{eq10002} and \eqref{eq822}, it holds that
  \begin{align}\label{kxx}
  &I(\delta \xi_*,  \delta {U},\delta \psi', \delta U_-)
  \notag\\
   = &(1- \dot{k})\sigma\left( \int_{0}^{\dot{\xi}_* + \delta\xi_*} \Theta(\tau)\dif \tau - \int_{0}^{\dot{\xi}_*}\Theta(\tau)\dif \tau \right)
  + \sigma\int_{\dot{\xi}_*+\delta\xi_*}^{\dot{\xi}_*}  \Theta(\tau) \dif \tau
  \notag\\
    &+\displaystyle\frac{\delta \xi_*}{L -\xi_*} \sigma \int_{\dot{\xi}_*+\delta\xi_*}^{\dot{\xi}_*}  \Theta(\tau) \dif \tau + O(1)\sigma^{\frac32}\cdot \delta \xi_* + O(1) \sigma^2
    \notag\\
     = &\Big( -\sigma \dot{k} \Theta(\dot{\xi}_*) \delta \xi_* + O(1) \sigma \cdot \delta \xi_*^2\Big) + \sigma \cdot \displaystyle\frac{\delta \xi_*}{L - \xi_*} \Big(-\Theta(\dot{\xi}_*)\delta \xi_*+ O(1) \delta\xi_*^2\Big)
     \notag\\
     &+ O(1)\sigma^{\frac32}\cdot \delta \xi_* + O(1) \sigma^2
     \notag\\
      =& \Big( -\sigma \dot{k} \Theta(\dot{\xi}_*)+ O(1)\sigma^{\frac32}\Big)\cdot \delta \xi_* + O(1) \sigma \cdot \delta \xi_*^2 + O(1) \sigma^2,
  \end{align}
  where $O(1)$ depends on $\dot{\xi}_*$, $\bar{U}_\pm$, $L$, $\alpha$ and $P_{\mr{e}}^{'}$.

Obviously, \eqref{kxx} implies that, as long as $ \Theta(\dot{\xi}_*) \neq 0$ and $\sigma$ small enough, one has
\begin{equation}
  \frac{\partial I}{\partial{\delta \xi_*}}(0, 0, 0, \dot{ U}_-) =
  -\sigma \dot{k} \Theta(\dot{\xi}_*)+ O(1)\sigma^{\frac32}\neq 0.
\end{equation}
By applying the implicit function theorem, there exists a solution $\delta \xi_{*}$ satisfying the equation $\eqref{eq248}$, and
\begin{equation}
   |\delta \xi_*| \leq \Big|\frac{O(1) \sigma}{\dot{k}\Theta(\dot{\xi}_*)} \Big| \leq C_{\mr{s}}\sigma.
 \end{equation}

\end{proof}

Define
\begin{align*}
  \mathscr{N}(\dot{U}_+; \dot{\psi}')\defs \Big\{(\delta U; \delta \psi'): (\delta U - \dot{U}_+;  \delta \psi' - \dot{\psi}')\in \textsl{N}\Big(\frac12 \sigma^{\frac32}\Big)\Big\}.
\end{align*}
Lemma \ref{lem52} and Lemma \ref{lem61} imply that the existence of the solution $(\delta U^*, \delta{\psi^*}'; \delta \xi_*)$ to the linearized problem \eqref{eq978}-\eqref{eq11000} as $(\delta U; \delta \psi')\in \mathscr{N}(\dot{U}_+; \dot{\psi}')$.
Furthermore, it can be proved that $(\delta {U}^* ; \delta {\psi^*}'  )\in \mathscr{N}(\dot{U}_+; \dot{\psi}')$ if $\sigma$ sufficiently small, i.e., the iteration mapping $\mathbf{\Pi}$ is well defined, as the following lemma shows:
\begin{lem}\label{lem62}
   There exists a positive constant $\sigma_2$ with $0<\sigma_2 \ll 1$, such that for any $0< \sigma \leq \sigma_2$, if $(\delta U ;  \delta \psi' )\in \mathscr{N}(\dot{U}_+; \dot{\psi}')$, then there exists a solution $(\delta {U}^*; \delta {\psi^*}')$ to the linearized problem \eqref{eq978}-\eqref{eq11000} and satisfies
 $(\delta {U}^* ; \delta {\psi^*}' )\in \mathscr{N}(\dot{U}_+; \dot{\psi}')$.
 \end{lem}

\begin{proof}
The proof is divided into three steps.

\textbf{Step 1}: In this step, we prove the existence of the solution $(\delta {U}^*; \delta {\psi^*}')$.

By Lemma \ref{lem52} and Lemma \ref{lem61}, it suffices to verify that $f_1(\xi,0)=0$ and $f_i\in \mathbf{H}_{0,\alpha}^{(1-\alpha;\{Q_3,Q_4\})}(\dot{\Omega}_+) $, $(i=1,2)$.

Since $(\delta U;  \delta \psi' )\in \mathscr{N}(\dot{U}_+; \dot{\psi}')$, then there exist positive constants $\kappa_1$ and $\kappa_2$ depending only on the background solution $\bar{U}_+$, such that
\begin{align}
  \kappa_1 \eta \leq r\leq \kappa_2 \eta,
\end{align}
which implies that
\begin{align}\label{mn}
 \frac{1}{\kappa_2}\leq  \displaystyle\frac{\eta}{r}\leq \frac{1}{\kappa_1}.
\end{align}
Recalling the definition of $f_1$, since $\delta \theta (\xi,0) = 0$ and $\delta\psi'(0) = 0$, it is easy to check that $f_1(\xi,0) =0$ and $f_1\in \mathbf{H}_{0,\alpha}^{(1-\alpha;\{Q_3,Q_4\})}(\dot{\Omega}_+)$ by employing \eqref{mn}.

Moreover, notice that
\begin{align}
 \frac{\delta\theta}{\eta}- \displaystyle\frac{2\eta}{r^2} \displaystyle\frac{\sin\theta}{\rho q} = O(1)\frac{(\delta\theta)^2}{\eta},
\end{align}
where we use the assumption $\bar{\rho}_+\bar{q}_+ = 2$ and $O(1)$ depends on $\bar{U}_+$ and $\dot{\xi}_*$. Then it is easy to see that $\displaystyle\frac{(\delta\theta)^2}{\eta}(\xi,0+) =0$ due to $\delta \theta(\xi,0+) = 0$. Thus $f_2$ is not singular on $\{\eta=0\}$ and one has $f_2\in \mathbf{H}_{0,\alpha}^{(1-\alpha;\{Q_3,Q_4\})}(\dot{\Omega}_+)$.
Then, by employing Theorem \ref{sub}, there exists a solution $(\delta \theta^*,\delta {p}^*)$ to the linearized problem $\eqref{eq978}$-$\eqref{eq0234}$ with the boundary conditions $\eqref{eq240}$-$\eqref{eq867}$ and $\eqref{p^*}$. Furthermore, by Lemma \ref{lem52}, the existence of the solution $(\delta {U}^*; \delta {\psi^*}')$ to the linearized problem \eqref{eq978}-\eqref{eq11000} can be established.

\textbf{Step 2}: In this step, we will establish the estimate for the solution $(\delta {U}^* - \dot{U}_+ ; \delta {\psi^*}' - \dot{\psi}' )$.

Applying the definitions \eqref{iteration1}-\eqref{initial1}, one has
\begin{align}
  (\delta {U}^* - \dot{U}_+ ; \delta {\psi^*}' - \dot{\psi}' ) = \mathscr{T}(\mathscr{F} - \dot{\mathscr{F}};\mathscr{G} - \dot{\mathscr{G}}; \sigma \Theta^* - \sigma \dot{\Theta}; \sigma P_{\mr{e}}^* - \sigma \dot{P}_{\mr{e}} ),
\end{align}
where
\begin{align*}
  \mathscr{F}\defs& (f_1(\delta U, \delta \psi', \delta \xi_{*}),
  f_2(\delta U, \delta \psi', \delta \xi_{*}), f_3(\delta U)),\\
  \dot{\mathscr{F}}\defs& (0,0, B(\bar{U}_+)),\\
  \mathscr{G}\defs &(G_j^*(\delta U,\delta U_-, \delta \psi', \delta \xi_*); j=1,2,3,4),\\
  \dot{\mathscr{G}}\defs& (\dot{g}_1, \dot{g}_2, \dot{g}_3, \dot{g}_4).
\end{align*}
Similar as the proof of Lemma \ref{lem52}, one has
\begin{align}\label{eq013}
  &\| \delta U^*- \dot{U}_+ \|_{(\dot{\Omega}_+;\dot{\Gamma}_{\mr{s}})} + \|\delta {\psi^*}' -\dot{\psi}' \|_{1 ,\alpha;\dot{\Gamma}_{\mr{s}}}^{( -\alpha;Q_4)}\notag \\
  \leq & C \Big(\sum_{j=1}^2 \| f_j \|_{0,\alpha;\dot{\Omega}_+}^{(1-\alpha;\{Q_3,Q_4\})} + \sum_{j=1}^{4}\| g_j^* - \dot{g}_j\|_{1,\alpha;\dot{\Gamma}_{\mr{s}}}^{(-\alpha; Q_4)}
   +  \|f_3 +B(\bar{U}_+)\|_{0,\alpha;\dot{\Omega}_+}^{(1-\alpha;\{Q_3,Q_4\})}\notag \\
  &+\sigma\cdot \|  P_{\mr{e}}^* - \dot{P}_{\mr{e}} \|_{1,\alpha;{\Gamma}_3}^{(-\alpha;Q_3)}+ \sigma\cdot\|  \Theta^* -\dot{\Theta}\|_{1,\alpha;{\Gamma}_4}^{(-\alpha;\{Q_3,Q_4\})} \Big).
  \end{align}
 Now, we analyze the terms on the right hand side of \eqref{eq013}.
By the definition of $f_1$ and the estimate $\eqref{d}$ in Lemma \ref{lem61}, one has
 \begin{align}
&\|f_1(\delta U, \delta \psi', \delta \xi_{*})\|_{0,\alpha;\dot{\Omega}_+}^{(1-\alpha;\{Q_3,Q_4\})}\notag\\
\leq &\Big\|2 \Big(\bar{q}_+ - \displaystyle\frac{\eta}{r}  q\cos\theta\Big) \partial_\xi \delta \theta + \displaystyle\frac{2\eta}{r}\displaystyle\frac{\sin\theta}{\rho q}\partial_\xi \delta p\Big\|_{0,\alpha;\dot{\Omega}_+}^{(1-\alpha;\{Q_3,Q_4\})}\notag\\
 &+ \Big\|\frac{(L-{\xi})\cdot \delta \psi'({\eta})}{L- \psi ({\eta})} \partial_\xi \delta p\Big\|_{0,\alpha;\dot{\Omega}_+}^{(1-\alpha;\{Q_3,Q_4\})}\notag\\
&+ \Big\|\displaystyle\frac{2\eta}{r}\displaystyle\frac{\delta \xi_* - \int_{\eta}^1 \delta \psi'(\tau)\dif \tau}{L -\psi(\eta)}\Big(\displaystyle\frac{\sin\theta}{\rho q}\partial_\xi\delta p - q\cos\theta \partial_\xi \delta \theta\Big) \Big\|_{0,\alpha;\dot{\Omega}_+}^{(1-\alpha;\{Q_3,Q_4\})}\notag\\
\leq& C\Big(\|\delta U\|_{L^\infty(\dot{\Omega}_+)} \|(\delta \theta, \delta p)\|_{1,\alpha;\dot{\Omega}_+}^{(-\alpha;\{Q_3,Q_4\})} +
\|\delta \psi'\|_{L^\infty(\dot{\Gamma}_s)} \|\delta p\|_{1,\alpha;\dot{\Omega}_+}^{(-\alpha;\{Q_3,Q_4\})} \Big)\notag\\
&+ C(\dot{\xi}_*)\Big(|\delta \xi_*|\|(\delta \theta, \delta p)\|_{1,\alpha;\dot{\Omega}_+}^{(-\alpha;\{Q_3,Q_4\})} + \|\delta \psi'\|_{L^\infty(\dot{\Gamma}_s)}\|(\delta \theta, \delta p)\|_{1,\alpha;\dot{\Omega}_+}^{(-\alpha;\{Q_3,Q_4\})}
\Big)\notag\\
\leq & C \sigma^2,
 \end{align}
 where the constant $C$ depends on $\bar{U}_+$, $C_{\mr{s}}$ and $\dot{\xi}_*$.

 Similarly, one has
\begin{equation}\label{eq293}
 \|f_2(\delta U, \delta \psi', \delta \xi_{*})\|_{0,\alpha;\dot{\Omega}_+}^{(1-\alpha;\{Q_3,Q_4\})}\leq C \sigma^2.
\end{equation}

Moreover, since
\begin{equation}
  \begin{aligned}
    f_3(\delta U) + B(\bar{U}_+) & = -\left(B(U) - B(\bar{U}_+) - (\bar{q}_+ \delta q + \displaystyle\frac{1}{\bar{\rho}_+}\delta p +\bar{T}_+\delta s)  \right)\\
    & = - \int_{0}^{1} D_{U}^2 B(\bar{U}_+ + t\delta U)\dif t \cdot (\delta U)^2,
  \end{aligned}
\end{equation}
one can obtain that
\begin{equation}
  \|f_3 +B(\bar{U}_+)\|_{0,\alpha;\dot{\Omega}_+}^{(1-\alpha;\{Q_3,Q_4\})}
  \leq C\sigma^2.
\end{equation}

For the boundary conditions, on $\Gamma_4$, one has
\begin{equation}\label{eq296}
\begin{aligned}
\Theta^* - \dot{\Theta} &=  \Theta \Big(\xi + \frac{L - \xi}{L - \dot{\xi}_*} \delta \xi_* \Big) - \Theta(\xi)\\
& = \int_{0}^{1} \Theta'\Big(\xi + s \frac{L - \xi}{L - \dot{\xi}_*} \delta \xi_*\Big) \dif s \cdot \frac{L - \xi}{L - \dot{\xi}_*} \delta \xi_*.
\end{aligned}
\end{equation}
By \eqref{d}, it follows that
\begin{equation}\label{eq297}
\|\Theta^* - \dot{\Theta} \|_{1,\alpha;{\Gamma}_4}^{(-\alpha;\{Q_3,Q_4\})} \leq C  \| \Theta'\|_{\mcc^{1,\alpha}(\Gamma_4)} \cdot |\delta \xi_*| \leq C\cdot C_{\mr{s}}\sigma \leq  C\sigma.
\end{equation}
On the exit $\Gamma_3$, by recalling \eqref{eq0103}, it holds that
\begin{equation}
  P_{\mr{e}}^* = {P}_{\mr{e}}(\eta)+ O(1)\sigma,
\end{equation}
thus,
\begin{equation}
   \|P_{\mr{e}}^* -\dot{P}_{\mr{e}} \|_{1,\alpha;{\Gamma}_3}^{(-\alpha;Q_3)}\leq  C\sigma.
\end{equation}
Finally, on the fixed boundary $\dot{\Gamma}_{\mr{s}}$, employing \eqref{eq0105}, \eqref{d} and \eqref{g2}, for $j=1,2,3$, one has
\begin{align}\label{eq535}
      \|g_j^* - \dot{g}_j\|_{1,\alpha;\dot{\Gamma}_{\mr{s}}}^{(-\alpha;Q_4)}
      \leq C\|\partial_\xi \dot{U}_-\|_{{\mcc}^{1,\alpha}(\bar{\Omega})}\cdot|\delta\xi_*|+ O(1) \sigma^2
       \leq C \sigma^2.
  \end{align}
A similar argument yields that
\begin{align}
  \|g_4^* - \dot{g}_4\|_{1,\alpha;\dot{\Gamma}_{\mr{s}}}^{(-\alpha;Q_4)}\leq C\sigma^2.
\end{align}

Therefore, for sufficiently small $\sigma$, \eqref{eq013} implies that
\begin{equation}\label{xll}
   \|\delta U^*- \dot{U}_+ \|_{(\dot{\Omega}_+;\dot{\Gamma}_{\mr{s}})} + \|\delta {\psi^*}' -\dot{\psi}' \|_{1 ,\alpha;\dot{\Gamma}_{\mr{s}}}^{( -\alpha;Q_4)} \leq C\sigma^2 \leq \frac12 \sigma^{\frac32}.
\end{equation}

\textbf{Step 3}: Finally, it remains to show that the conditions in the space $\mathscr{N}(\dot{U}_+; \dot{\psi}')$ hold for $(\delta {U}^* - \dot{U}_+ ; \delta {\psi^*}' - \dot{\psi}' )$. By Remark \ref{axis}, it suffices to show that the condition \eqref{eq240} holds.

According to Remark \ref{opp}, one has $\delta \theta^*(\xi,0) = \partial_\eta \delta p^* (\xi,0) = 0$.
Moreover, differentiating both sides of equation \eqref{eq0234} with respect to $\eta$, applying the conditions $\delta \theta^*(\xi,0) = \partial_\eta \delta p^*(\xi,0) =0$, one has
 \begin{align}
  \frac12 \partial_\eta^2 \delta \theta^* (\xi,0)= \partial_\eta f_2(\xi,0).
 \end{align}
 Then by employing the conditions on the axis $\Gamma_2$ for $\delta U$, direct calculation shows that $\partial_\eta f_2(\xi,0) = 0$. Thus, one has $\partial_\eta^2 \delta {\theta}^*(\xi,0) = 0$.
Moreover, by the expression of $\delta {\psi^*}'$ in \eqref{eq870}, it is easy to check that $\delta {\psi^*}'(0) = 0$.
Finally, by \eqref{eq251}, it can be easily verified that $\partial_\eta (\delta p^*,\delta q^*, \delta s^*)(\dot{\xi}_*, 0) = 0$. Then differentiating both sides of equation \eqref{eq980} with respect to $\eta$, one has
\begin{equation}
  \partial_\xi(\partial_\eta\delta s^*) = 0.
\end{equation}
 Therefore, $\partial_\eta\delta s^*(\xi,0) = 0$. Differentiating both sides of equation \eqref{eq979} with respect to $\eta$, one can also obtain $\partial_\eta \delta q^*(\xi,0) = 0$ directly.
Therefore, the condition \eqref{eq240} holds.

Thus, we complete the proof of the Lemma \ref{lem62}.

\end{proof}

The Theorem \ref{thm26} will be proved as long as that the mapping $\mathbf{\Pi}$ can be proved to be contractive in $\textsl{N}(\frac12 \sigma^{\frac32})$, which will be done in the following lemma.

\begin{lem}\label{lem63}
    There exists a positive constant $\sigma_3$ with $0<\sigma_3 \ll 1$, such that for any $0< \sigma \leq \sigma_3$, the mapping $\mathbf{\Pi}$ is contractive.
  \end{lem}

  \begin{proof}

Suppose that $(\delta U_k; \delta \psi_k^{'})\in \mathscr{N}(\dot{U}_+; \dot{\psi}')$, $(k=1,2)$, then by Lemma \ref{lem61} and Lemma \ref{lem62}, there exists $\delta \xi_{*k}$ satisfying the estimate \eqref{d} and $(\delta {U}_k^* ; \delta \psi_k^{*'})\in \mathscr{N}(\dot{U}_+; \dot{\psi}')$ such that
\begin{align}\label{iteration2}
  (\delta {U}_k^*; \delta {\psi_k^*}';\delta \xi_{*k}) = \mathscr{T}_e(\mathscr{F}_k; \mathscr{G}_k; \sigma \Theta^*(\xi; \delta \xi_{*k}); \sigma P_{\mr{e}}^*(\eta; \delta U_k)),
\end{align}
where
\begin{align*}
  \mathscr{F}_k\defs& (f_1(\delta U_k, \delta \psi_k^{'}, \delta \xi_{*k}), f_2(\delta U_k, \delta \psi_k^{'}, \delta \xi_{*k}), f_3(\delta U_k)),\\
  \mathscr{G}_k\defs& (G_j^*(\delta U_k,\delta U_-,  \delta \psi_k^{'}, \delta \xi_{*k}); j=1,2,3,4).
\end{align*}
  To prove the mapping $\mathbf{\Pi}$ is contractive, it suffices to show that, for sufficiently small $\sigma>0$, it holds that
  \begin{align}\label{ee}
     &\| \delta U_2^* - \delta U_1^* \|_{(\dot{\Omega}_+;\dot{\Gamma}_{\mr{s}})} + \| \delta \psi_{2}^{*'} - \delta \psi_{1}^{*'}\|_{1,\alpha;\dot{\Gamma}_{\mr{s}}}^{(-\alpha;Q_4)}
      \notag \\
     \leq& \frac12\Big(\| \delta U_2 - \delta U_1\|_{(\dot{\Omega}_+;\dot{\Gamma}_{\mr{s}})} + \| \delta {\psi_2}' - \delta {\psi_1}'\|_{1,\alpha;\dot{\Gamma}_{\mr{s}}}^{(-\alpha;Q_4)}\Big).
   \end{align}
By \eqref{iteration2}, one has
\begin{align}
  &(\delta U_2^* - \delta U_1^* ; \delta \psi_{2}^{*'} - \delta \psi_{1}^{*'})\label{iteration3}\\
  =&  \mathscr{T}_e(\mathscr{F}_2 - \mathscr{F}_1; \mathscr{G}_2 -\mathscr{G}_1 ; \sigma \Theta^*(\xi; \delta \xi_{*2}) -\sigma \Theta^*(\xi; \delta \xi_{*1}) ; \sigma P_{\mr{e}}^*(\eta; \delta U_2) - \sigma P_{\mr{e}}^*(\eta; \delta U_1))\notag.
\end{align}
Since the right hand side of \eqref{iteration3} includes $\delta \xi_{*k}$, which is determined by Lemma \ref{lem61} with given $(\delta U_k^*; \delta \psi_{k}^{*'})$, one has to estimate $|\delta \xi_{*2} - \delta \xi_{*1}|$ first.

According to the definition $I$ in $\eqref{eq249}$, it follows that
\begin{align}\label{eq304}
 0 =& I(\delta \xi_{*2},  \delta {U}_2,\delta {\psi_2}', \delta U_-) - I(\delta \xi_{*1},  \delta {U}_1,\delta {\psi_1}', \delta U_-)
  \notag \\
 = &I(\delta \xi_{*2}, \delta U_2 , \delta {\psi_2}' , \delta U_-) - I(\delta \xi_{*1}, \delta U_2 , \delta {\psi_2}' , \delta U_-)
  \notag \\
& + I(\delta \xi_{*1}, \delta U_2 , \delta {\psi_2}' , \delta U_-) -  I(\delta \xi_{*1}, \delta U_1 , \delta {\psi_1}', \delta U_-)
 \notag \\
= &\int_{0}^{1} \frac{\partial I }{\partial (\delta \xi_*)} (\delta \xi_{*t}, \delta U_2,\delta {\psi_2}', \delta U_-)\dif t \cdot (\delta \xi_{*2} - \delta \xi_{*1})
 \notag \\
& + \int_{0}^{1} \nabla_{(\delta U, \delta \psi')} I (\delta \xi_{*1} , \delta U_t, \delta {\psi_t}' , \delta U_-)\dif t \cdot (\delta U_2 - \delta U_1, \delta {\psi_2}' - \delta {\psi_1}' ),
\end{align}
where
\begin{align}
  &\delta \xi_{*t} : = t \delta \xi_{*2} + (1 - t) \delta \xi_{*1},\quad
    \delta U_t : = t \delta U_2 + (1 - t) \delta U_1,
    \notag \\
    & \delta {\psi_t}' : = t \delta {\psi_2}' + (1 - t)\delta {\psi_1}'.
\end{align}
Similar calculations as in Lemma \ref{lem61}, one has
\begin{align}\label{eq305}
&\frac{\partial I }{\partial (\delta \xi_*)} (\delta \xi_{*t}, \delta U_2,\delta {\psi_2}', \delta U_-)
 \notag \\
 = &\frac{\partial I }{\partial (\delta \xi_*)}(0 , 0, 0; \dot{U}_-) + \int_{0}^{1} \nabla_{(\delta \xi_*, \delta U, \delta \psi' , \delta U_-)} \frac{\partial I}{\partial (\delta \xi_*)}(s\delta \xi_{*t}, s\delta U_{2}, s\delta {\psi_{2}}', s\delta U_{-})\dif s
  \notag \\
&\cdot (\delta \xi_{*t}, \delta U_2  , \delta {\psi_2}', \delta U_- - \dot{U}_-)
 \notag \\
 =& - \sigma \dot{k} \Theta(\dot{\xi}_*) + O(1) \sigma^{\frac32} + O(1) \sigma^2.
\end{align}
Moreover,
\begin{align}
  \nabla_{(\delta U, \delta \psi')} I (\delta \xi_{*1} , \delta U_t, \delta {\psi_t}' , \delta U_-) = O(1)\sigma,
\end{align}
where $O(1)$ depends on $\dot{k} \Theta (\dot{\xi}_*) $, $\dot{\xi}_*$, $\bar{U}_\pm$, $L$, $P_{\mr{e}}$ ,$\Theta$ and $\alpha$.

Therefore, \eqref{eq304} yields that
\begin{equation}\label{eq306}
\begin{aligned}
|\delta \xi_{*2} - \delta \xi_{*1}| & \leq \Big|\frac{O(1) \sigma \Big(\| \delta U_2 - \delta U_1\|_{1,\alpha;\dot{\Gamma}_{\mr{s}}}^{(-\alpha;Q_4)} + \| \delta {\psi_2}' - \delta {\psi_1}'\|_{1,\alpha;\dot{\Gamma}_{\mr{s}}}^{(-\alpha;Q_4)}\Big)} { - \sigma \dot{k} \Theta (\dot{\xi}_*) + O(1) \sigma^{\frac32} + O(1) \sigma^2}\Big|\\
&\leq C \Big(\| \delta U_2 - \delta U_1\|_{1,\alpha;\dot{\Gamma}_{\mr{s}}}^{(-\alpha;Q_4)} + \|\delta {\psi_2}' - \delta {\psi_1}'\|_{1,\alpha;\dot{\Gamma}_{\mr{s}}}^{(-\alpha;Q_4)}\Big),
\end{aligned}
\end{equation}
where the constant $C$ depends on $\dot{k} \Theta (\dot{\xi}_*) $, $\dot{\xi}_*$, $\bar{U}_\pm$, $L$, $P_{\mr{e}}$ ,$\Theta$ and $\alpha$.

Finally, employing the estimate \eqref{eq306}, by similar computations as in Lemma \ref{lem62},  one has that the estimate \eqref{ee} holds for sufficiently small $\sigma$, which completes the proof.

\end{proof}

\appendix

\section*{Acknowlegements}

The research of Beixiang Fang was supported in part by Natural Science
Foundation of China under Grant Nos. 11971308, 11631008 and 11371250,
the Shanghai Committee of Science and Technology (Grant No. 15XD1502300). The research
of Xin Gao was supported in part by China Scholarship Council (No.201906230072).

\end{document}